\documentclass[a4paper,11pt,reqno]{amsart}

\usepackage[margin=1in,includehead,includefoot,asymmetric,marginparwidth=1cm]{geometry}
\usepackage[utf8]{inputenc}
\usepackage[T1]{fontenc}
\usepackage{lmodern,microtype}
\usepackage{hyperref}
\usepackage{url}
\usepackage{booktabs}
\usepackage{mathtools,amssymb,amsthm}
\usepackage[foot]{amsaddr}
\usepackage{graphicx, color}
\usepackage[margin=5mm]{caption}
\usepackage{enumitem}
\usepackage[textsize=footnotesize,disable]{todonotes}
\usepackage{mycommands}
\graphicspath{{./graphics/}}

\newtheorem{theorem}{Theorem}
\newtheorem{proposition}[theorem]{Proposition}
\newtheorem{lemma}[theorem]{Lemma}
\newtheorem{corollary}[theorem]{Corollary}


\hypersetup{
  pdftitle={Kneser graphs are Hamiltonian},
  pdfauthor={Arturo Merino, Torsten M\"utze, Namrata}
}

\title{Kneser graphs are Hamiltonian}

\author{Arturo Merino}
\address[Arturo Merino]{Department of Mathematics, TU Berlin, Germany}
\email{merino@math.tu-berlin.de}

\author{Torsten M\"utze}
\address[Torsten M\"utze]{Department of Computer Science, University of Warwick, United Kingdom \& Department of Theoretical Computer Science and Mathematical Logic, Charles University, Prague, Czech Republic}
\email{torsten.mutze@warwick.ac.uk}

\author{Namrata}
\address[Namrata]{Department of Computer Science, University of Warwick, United Kingdom}
\email{namrata@warwick.ac.uk}

\thanks{This work was supported by Czech Science Foundation grant GA~22-15272S. The three authors participated in the workshop `Combinatorics, Algorithms and Geometry' in March 2024, which was funded by German Science Foundation grant~522790373.}

\thanks{An extended abstract of this work appeared in the Proceedings of the 55th Annual ACM Symposium on the Theory of Computing (STOC 2023) \cite{MR4617441}.}

\begin{document}

\begin{abstract}
For integers~$k\geq 1$ and $n\geq 2k+1$, the Kneser graph~$K(n,k)$ has as vertices all $k$-element subsets of an $n$-element ground set, and an edge between any two disjoint sets.
It has been conjectured since the 1970s that all Kneser graphs admit a Hamilton cycle, with one notable exception, namely the Petersen graph~$K(5,2)$.
This problem received considerable attention in the literature, including a recent solution for the sparsest case $n=2k+1$.
The main contribution of this paper is to prove the conjecture in full generality.
We also extend this Hamiltonicity result to all connected generalized Johnson graphs (except the Petersen graph).
The generalized Johnson graph~$J(n,k,s)$ has as vertices all $k$-element subsets of an $n$-element ground set, and an edge between any two sets whose intersection has size exactly~$s$.
Clearly, we have $K(n,k)=J(n,k,0)$, i.e., generalized Johnson graphs include Kneser graphs as a special case.
Our results imply that all known natural families of vertex-transitive graphs defined by intersecting set systems have a Hamilton cycle, which settles an interesting special case of Lov\'asz' conjecture on Hamilton cycles in vertex-transitive graphs from~1970.
Our main technical innovation is to study cycles in Kneser graphs by a kinetic system of multiple gliders that move at different speeds and that interact over time, reminiscent of the gliders in Conway's Game of Life, and to analyze this system combinatorially and via linear algebra.
\end{abstract}

\maketitle

\section{Introduction}

For integers~$k\geq 1$ and~$n\geq 2k+1$, the \emph{Kneser graph~$K(n,k)$} \marginpar{$K(n,k)$} has as vertices all $k$-element subsets of~$[n]:=\{1,2,\ldots,n\}$, and an edge between any two sets~$A$ and~$B$ that are disjoint, i.e., $A\cap B=\emptyset$.
Kneser graphs were introduced by Lov\'asz~\cite{MR514625} in his celebrated proof of Kneser's conjecture.
Using the Borsuk-Ulam theorem, he proved that the chromatic number of~$K(n,k)$ equals $n-2k+2$, and his proof gave rise to the field of topological combinatorics.
We proceed to list a few other important properties of Kneser graphs.
The maximum independent set in~$K(n,k)$ has size $\binom{n-1}{k-1}$ by the famous Erd{\H o}s-Ko-Rado~\cite{MR0140419} theorem.
Furthermore, the graph~$K(n,k)$ is vertex-transitive, i.e., it `looks the same' from the point of view of any vertex, and all vertices have degree~$\binom{n-k}{k}$.
Lastly, note that when $n<ck$, the Kneser graph~$K(n,k)$ does not contain cliques of size~$c$, whereas it does contain such cliques when $n\geq ck$.
Many other properties of Kneser graphs have been studied, for example their diameter~\cite{MR2186709}, treewidth~\cite{MR3177543}, boxicity~\cite{caoduro-lichev:22}, and removal lemmas~\cite{MR3766248}.

\subsection{Hamilton cycles in Kneser graphs}

In this work we investigate Hamilton cycles in Kneser graphs, i.e., cycles that visit every vertex exactly once.
Kneser graphs have long been conjectured to have a Hamilton cycle, with one notable exception, the Petersen graph~$K(5,2)$ (see Figure~\ref{fig:petersen}), which only admits a Hamilton path.
This conjecture goes back to the 1970s, and in the following we give a detailed account of this long history.
As Kneser graphs are vertex-transitive, this is a special case of Lov\'asz' famous conjecture~\cite{MR0263646}, which asserts that every connected vertex-transitive graph admits a Hamilton path.
A stronger form of the conjecture asserts that every connected vertex-transitive graph admits a Hamilton cycle, apart from five exceptional graphs, one of them being the Petersen graph.
So far, the conjecture for Hamilton cycles in Kneser graphs has been tackled from two angles, namely for sufficiently dense Kneser graphs, and for the sparsest Kneser graphs.
From the aforementioned results about the degree and cliques in~$K(n,k)$, we see that $K(n,k)$ is relatively dense when $n$ is large w.r.t.\ $k$, and relatively sparse otherwise.
The sparsest case is when $n=2k+1$, and the graphs $O_k:=K(2k+1,k)$ \marginpar{$O_k$} are also known as \emph{odd graphs}.
Intuitively, proving Hamiltonicity should be easier for the dense cases, and harder for the sparse cases.

We first recap the known results for dense Kneser graphs.
Heinrich and Wallis~\cite{MR510592} showed that~$K(n,k)$ has a Hamilton cycle if $n\geq 2k+k/(\sqrt[k]{2}-1)=(1+o(1))k^2/\ln 2$.
This was improved by B.~Chen and Lih~\cite{MR888679}, whose results imply that~$K(n,k)$ has a Hamilton cycle if $n\geq (1+o(1))k^2/\log k$; see~\cite{MR1405991}.
In another breakthrough, Y.~Chen~\cite{MR1778200} showed that~$K(n,k)$ is Hamiltonian when~$n\geq 3k$.
A~particularly nice and clean proof for the cases where~$n=ck$, $c\in\{3,4,\ldots\}$, was obtained by Y.~Chen and F\"uredi~\cite{MR1883565}.
Their proof uses Baranyai's well-known partition theorem for complete hypergraphs~\cite{MR0416986} to partition the vertices of~$K(ck,k)$ into cliques of size~$c$.
This proof method was extended by Bellmann and Sch\"ulke to any $n\geq 4k$~\cite{MR4247012}.
The asymptotically best result known to date, again due to Y.~Chen~\cite{MR1999733}, is that~$K(n,k)$ has a Hamilton cycle if $n\geq (3k+1+\sqrt{5k^2-2k+1})/2=(1+o(1))2.618\ldots\cdot k$.
With the help of computers, Shields and Savage~\cite{MR2020936} found Hamilton cycles in~$K(n,k)$ for all~$n\leq 27$ (except for the Petersen graph).

We now briefly summarize the Hamiltonicity story of the sparsest Kneser graphs, namely the odd graphs.
Note that $O_k=K(2k+1,k)$ has degree~$k+1$, which is only logarithmic in the number of vertices.
The conjecture that~$O_k$ has a Hamilton cycle for all $k\geq 3$ originated in the 1970s, in papers by Meredith and Lloyd~\cite{MR0457282,MR0321782} and by Biggs~\cite{MR556008}.
Already Balaban~\cite{balaban:72} exhibited a Hamilton cycle for the cases~$k=3$ and~$k=4$, and Meredith and Lloyd described one for~$k=5$ and~$k=6$.
Later, Mather~\cite{MR0389663} solved the case~$k=7$.
M\"utze, Nummenpalo and Walczak~\cite{MR4273468} finally settled the problem for all odd graphs, proving that $O_k$ has a Hamilton cycle for every~$k\geq 3$.
In fact, they even proved that~$O_k$ admits double-exponentially (in $k$) many distinct Hamilton cycles.
Already much earlier, Johnson~\cite{MR2836824} provided an inductive argument that establishes Hamiltonicity of~$K(n,k)$ provided that the existence of Hamilton cycles is known for several smaller Kneser graphs.
Combining his result with the unconditional results from~\cite{MR4273468} yields that~$K(2k+2^a,k)$ has a Hamilton cycle for all $k\geq 3$ and~$a\geq 0$.
These results still leave infinitely many open cases, the sparsest one of which is the family~$K(2k+3,k)$ for $k\geq 1$.

Another line of attack towards proving Hamiltonicity is to find long cycles in~$K(n,k)$.
To this end, Johnson~\cite{MR2046083} showed that there exists a constant~$c>0$ such that the odd graph~$O_k$ has a cycle that visits at least a $(1-c/\sqrt{k})$-fraction of all vertices, which is almost all vertices as $k$~tends to infinity.
This was generalized and improved in~\cite{MR3759914}, where it was shown that~$K(n,k)$ has a cycle visiting a $2k/n$-fraction of all vertices.
For $n=2k+1$ this fraction is $(1-1/(2k+1))$, and more generally for $n=2k+o(k)$ it is~$1-o(1)$.

The main contribution of this paper is to settle the conjecture on Hamilton cycles in Kneser graphs affirmatively in full generality.

\begin{theorem}
\label{thm:Knk}
For all $k\geq 1$ and $n\geq 2k+1$, the Kneser graph~$K(n,k)$ has a Hamilton cycle, unless it is the Petersen graph, i.e., $(n,k)=(5,2)$.
\end{theorem}

In the following we present generalizations of this result that we establish in this paper, and we discuss how they extend previously known Hamiltonicity results.
The relations between these results for different families of vertex-transitive graphs are illustrated in Figure~\ref{fig:theorems}.
In fact, our proof of Theorem~\ref{thm:Knk} enables us to settle all known natural instances of Lov\'asz' conjecture for vertex-transitive graphs defined by intersecting set systems.
As we shall see, Kneser graphs are the hardest cases among them to prove.
Indeed, the more general families of graphs can be settled easily once Hamiltonicity is established for Kneser graphs.

\begin{figure}[h!]
\includegraphics[page=1]{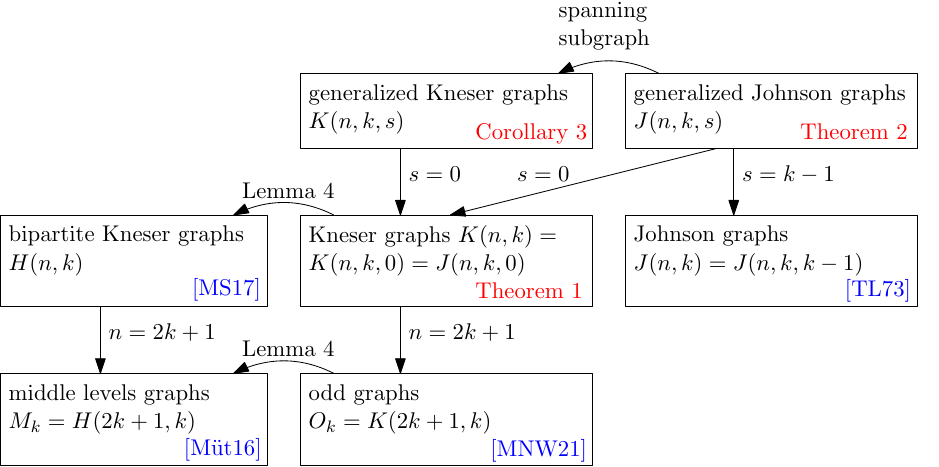}
\caption{Relation between Hamiltonicity results established in this paper and previous papers.
Arrows indicate implications.}
\label{fig:theorems}
\end{figure}

\subsection{Generalized Johnson graphs}

The \emph{generalized Johnson graph~$J(n,k,s)$} \marginpar{$J(n,k,s)$}  has as vertices all $k$-element subsets of~$[n]$, and an edge between any two sets~$A$ and~$B$ that satisfy~$|A\cap B|=s$, i.e., the intersection of~$A$ and~$B$ has size exactly~$s$.
To ensure that the graph is connected, we assume that $s<k$ and $n\geq 2k-s+\mathbf{1}_{[s=0]}$, where $\mathbf{1}_{[s=0]}$ denotes the indicator function that equals~1 if $s=0$ and~0 otherwise.
Generalized Johnson graphs are sometimes called `uniform subset graphs' in the literature, and they are also vertex-transitive.
Furthermore, by taking complements, we see that $J(n,k,s)$ is isomorphic to~$J(n,n-k,n-2k+s)$.
Clearly, Kneser graphs are special generalized Johnson graphs obtained for~$s=0$.
On the other hand, the graphs obtained for~$s=k-1$ are known as (ordinary) \emph{Johnson graphs}~$J(n,k):=J(n,k,k-1)$. \marginpar{$J(n,k)$}

Chen and Lih~\cite{MR888679} conjectured that all graphs~$J(n,k,s)$ admit a Hamilton cycle except the Petersen graph~$J(5,2,0)=J(5,3,1)$, and this problem was reiterated in Gould's survey~\cite{MR1106528}.
In their original paper, Chen and Lih settled the cases $s\in\{k-1,k-2,k-3\}$.
It is known that a Hamilton cycle in the Johnson graph~$J(n,k)=J(n,k,k-1)$ can be obtained by restricting the binary reflected Gray code for bitstrings of length~$n$ to those strings with Hamming weight~$k$~\cite{MR0349274}.
In fact, for Johnson graphs~$J(n,k)$ much stronger Hamiltonicity properties are known~\cite{MR1298971,MR1301949}.
Other properties of generalized Johnson graphs were investigated in~\cite{MR2466973,MR3713392,MR4055024,kozhevnikov-zhukovskii:22}.

We generalize Theorem~\ref{thm:Knk} further, by showing that all connected generalized Johnson graphs admit a Hamilton cycle.
This resolves Chen and Lih's conjecture affirmatively in full generality.

\begin{theorem}
\label{thm:Jnks}
For all $k\geq 1$, $0\leq s<k$, and $n\geq 2k-s+\mathbf{1}_{[s=0]}$ the generalized Johnson graph~$J(n,k,s)$ has a Hamilton cycle, unless it is the Petersen graph, i.e., $(n,k,s)\in\{(5,2,0),(5,3,1)\}$.
\end{theorem}

\subsection{Generalized Kneser graphs}

The \emph{generalized Kneser graph~$K(n,k,s)$} \marginpar{$K(n,k,s)$} has as vertices all $k$-element subsets of~$[n]$, and an edge between any two sets~$A$ and~$B$ that satisfy~$|A\cap B|\leq s$, i.e., the intersection of~$A$ and~$B$ has size at most~$s$.
The definition is very similar to generalized Johnson graphs, only the equality condition on the size of the set intersection is replaced by an inequality.
As a consequence, we clearly have $K(n,k,s)=\bigcup_{t\leq s}J(n,k,t)$, i.e., $K(n,k,s)$ has the same vertex set as~$J(n,k,s)$, but more edges.
In other words, $J(n,k,s)$ is a spanning subgraph of~$K(n,k,s)$.
Generalized Kneser graphs are also vertex-transitive, and they have been studied heavily in the literature; see e.g.~\cite{MR797510,MR1468332,MR2427759,MR3937815,MR4040062,MR4311584,MR4406218,metsch:22}.

As $J(n,k,s)$ is a spanning subgraph of~$K(n,k,s)$, Theorem~\ref{thm:Jnks} yields the following immediate corollary.

\begin{corollary}
\label{cor:Jnks}
For all $k\geq 1$, $0\leq s<k$, and $n\geq 2k-s+\mathbf{1}_{[s=0]}$ the generalized Kneser graph~$K(n,k,s)$ has a Hamilton cycle, unless it is the Petersen graph, i.e., $(n,k,s)\in\{(5,2,0),(5,3,1)\}$.
\end{corollary}

\subsection{Bipartite Kneser graphs and the middle levels problem}

For integers $k\geq 1$ and $n\geq 2k+1$, the \emph{bipartite Kneser graph~$H(n,k)$} \marginpar{$H(n,k)$} has as vertices all $k$-element and $(n-k)$-element subsets of~$[n]$, and an edge between any two sets~$A$ and~$B$ that satisfy~$A\seq B$.
It is easy to see that bipartite Kneser graphs are also vertex-transitive.
The following simple lemma shows that Hamiltonicity of~$K(n,k)$ is harder than the Hamiltonicity of~$H(n,k)$.

\begin{lemma}
\label{lem:KHnk}
If $K(n,k)$ admits a Hamilton cycle, then $H(n,k)$ admits a Hamilton cycle or path.
\end{lemma}

\begin{proof}
Given a Hamilton cycle~$C=(x_1,x_2,\ldots,x_N)$ in~$K(n,k)$, the sequences $P:=(x_1,\overline{x_2},x_3,\overline{x_4},\ldots)$ and $P':=(\overline{x_1},x_2,\overline{x_3},x_4,\ldots)$, where $\overline{x_i}:=[n]\setminus x_i$, are two spanning paths in~$H(n,k)$.
Consequently, if~$N=\binom{n}{k}$ is odd, then the concatenation $PP'$ is a Hamilton cycle in~$H(n,k)$, and if $N$ is even, then $P$ and~$P'$ are two disjoint cycles that together span the graph and that can be joined to a Hamilton path.
\end{proof}

We do not know how to strengthen Lemma~\ref{lem:KHnk} to obtain a Hamilton cycle in~$H(n,k)$ independently of the parity of~$N=\binom{n}{k}$ and without additional knowledge about the Hamilton cycle in~$K(n,k)$.

The sparsest bipartite Kneser graphs~$M_k:=H(2k+1,k)$ \marginpar{$M_k$} are known as \emph{middle levels graphs}, as they are isomorphic to the subgraph of the $(2k+1)$-dimensional hypercube induced by the middle two levels.
The well-known \emph{middle levels conjecture} asserts that~$M_k$ has a Hamilton cycle for all $k\geq 1$.
This conjecture was raised in the 1980s, settled affirmatively in~\cite{MR3483129}, and a short proof was given in~\cite{MR3819051}.
More generally, all bipartite Kneser graphs~$H(n,k)$ were shown to have a Hamilton cycle in~\cite{MR3759914}, via a short argument that uses the sparsest case~$M_k$ as a basis for induction.
These papers completed a long line of previous partial results on these problems; see the papers for more references and historical remarks.
Via Lemma~\ref{lem:KHnk} and its proof shown before, our Theorem~\ref{thm:Knk} thus also yields a new alternative proof for the Hamiltonicity of bipartite Kneser graphs (though we only obtain a Hamilton path when $\binom{n}{k}$ is even).

Consequently, our results in this paper settle Lov\'asz' conjecture for all known natural families of vertex-transitive graphs that are defined by intersecting set systems.

\subsection{Algorithmic considerations}

A \emph{combinatorial Gray code}~\cite{MR1491049,MR4649606} is an algorithm that computes a listing of combinatorial objects such that any two consecutive objects in the list satisfy a certain adjacency condition.
Many such algorithms are covered in depths in Knuth's book `The Art of Computer Programming Vol.~4A'~\cite{MR3444818}, and several of them correspond to computing a Hamilton cycle in a vertex-transitive graph, thus algorithmically solving one special case of Lov\'asz' conjecture.
For example, the classical binary reflected Gray code computes a Hamilton cycle in the $n$-dimensional hypercube, which can be seen as the Cayley graph of~$\mathbb{Z}_2^n$ given by the standard generators.
Another example is the well-known Steinhaus-Johnson-Trotter algorithm, which computes a Hamilton cycle in the Cayley graph of the symmetric group when the generators are adjacent transpositions.
Similarly, the recent solution~\cite{MR4060409} of Nijenhuis and Wilf's sigma-tau problem~\cite[Ex.~6]{MR0396274} computes a Hamilton cycle in the Cayley (di)graph of the symmetric group with the two generators being cyclic left-shift or transposition of the first two elements.
Similar Gray code algorithms have been discovered for the symmetric group with other generators, such as prefix reversals~\cite{DBLP:journals/cacm/Ord-Smith67,MR753548}, prefix shifts~\cite{corbett_1992,MR1308693,MR2682614}, and for other groups such as the alternating group~\cite{MR900932,MR3599935}.

Subsets of size~$k$ of an $n$-element ground set are known as \emph{$(n,k)$-combinations} in the Gray code literature.
Many different algorithms have been devised for generating $(n,k)$-combinations by element exchanges, i.e., any two consecutive combinations differ in removing one element from the subset and adding another one~\cite{MR0349274,MR782221,MR995888,MR737262,MR821383,MR936104}.
This is equivalent to saying that any two consecutive sets intersect in exactly~$k-1$ elements, i.e., such a Gray code computes a Hamilton cycle in the Johnson graph~$J(n,k)$.

Computing a Hamilton cycle in the Kneser graph~$K(n,k)$ thus corresponds to computing a Gray code for $(n,k)$-combinations where the adjacency condition is disjointness.
Our proof of the existence of a Hamilton cycle in~$K(n,k)$ is constructive, and it translates straightforwardly into an algorithm for computing the cycle whose running time is polynomial in the size~$N:=\binom{n}{k}$ of the Kneser graph.
It remains open whether there exists a more efficient algorithm, i.e., one with running time that is polynomial in~$n$ and~$k$ per generated combination (note that $N$ is exponential in~$k$), similarly to the previously mentioned combination generation algorithms; see also the discussion at the end of this paper.

\subsection{Proof ideas}

In Section~\ref{sec:Jnks} below we demonstrate how Theorem~\ref{thm:Knk} can be used to establish Theorem~\ref{thm:Jnks} by a simple inductive construction.
Consequently, the main work in this paper is to prove Theorem~\ref{thm:Knk}.

As mentioned before, M\"utze, Nummenpalo and Walczak~\cite{MR4273468} proved that $K(n,k)$ has a Hamilton cycle for $n=2k+1$ and all~$k\geq 3$.
Combining this result with Johnson's construction~\cite{MR2836824} shows that $K(n,k)$ has a Hamilton cycle for $n=2k+2^a$ and all~$k\geq 3$ and~$a\geq 0$, in particular for $n=2k+2$.
The techniques developed in this paper work whenever $n\geq 2k+3$, and thus they settle all remaining cases of Theorem~\ref{thm:Knk}.
It should be noted that our proof does not work in the cases $n=2k+1$ and $n=2k+2$, so the two earlier constructions do not become obsolete.

We follow a two-step approach to construct a Hamilton cycle in~$K(n,k)$ for $n\geq 2k+3$.
In the first step, we construct a \emph{cycle factor} in the graph, i.e., a collection of disjoint cycles that together visit all vertices.
In the second step, we join the cycles of the factor to a single cycle.
In the following we discuss both of these steps in more detail, outlining the main obstacles and novel ingredients to overcome them.
This outline reflects the structure of the remainder of this paper.

\subsubsection{Cycle factor construction}
\label{sec:idea-factor}

The starting point is to consider the characteristic vectors of the vertices of~$K(n,k)$.
For every $k$-element subset of~$[n]$, this is a bitstring of length~$n$ with exactly $k$ many 1s at the positions corresponding to the elements of the set.
For example, the vertex~$\{1,7,9\}$ of $K(9,3)$ is represented by the bitstring~$100000101$; see also Figure~\ref{fig:petersen}.
In this figure and the following ones, 1s are often represented by black squares, and 0s by white squares.
Clearly, two sets $A$ and~$B$ that are vertices of~$K(n,k)$ are disjoint if and only if the corresponding bitstrings have no 1s at the same positions.

\begin{figure}
\includegraphics{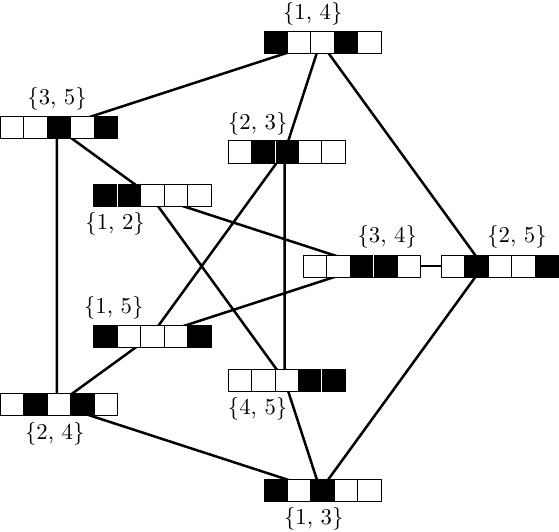}
\caption{The Petersen graph~$K(5,2)$.
The vertices are all 2-elements subsets of $[5]=\{1,2,3,4,5\}$, and in the corresponding bitstrings, 1s are represented by black squares and 0s by white squares.
}
\label{fig:petersen}
\end{figure}

Our construction of a cycle factor in the Kneser graph~$K(n,k)$ uses the following simple rule based on parenthesis matching, which is a technique pioneered by Greene and Kleitman~\cite{MR0389608} (in a completely different context):
Given a vertex represented by a bitstring~$x$, we interpret the 1s in~$x$ as opening brackets and the 0s as closing brackets, and we match closest pairs of opening and closing brackets in the natural way, which will leave some 0s unmatched.
This matching is done \emph{cyclically} across the boundary of~$x$, i.e., $x$ is considered as a cyclic string.
We write $f(x)$ for the vertex obtained from~$x$ by complementing all matched bits, leaving the unmatched bits unchanged.
For example, $x=100000101$ is interpreted as $x=()))))()({}=())\hyph\hyph\hyph()($, where each $\hyph$ denotes an unmatched closing bracket, and then complementing matched bits (the first three and last three in this case) yields the vertex~$f(x)=011000010$.
Repeatedly applying~$f$ to every vertex partitions the vertices of the Kneser graph into cycles, and we write $C(x):=(x,f(x),f^2(x),\ldots)$ for the cycle containing~$x$.
For example, for $x$ from before we obtain~$C(x)=(100000101,011000010,000110001,100001100,010000011,\ldots,000011010)$.
Figure~\ref{fig:gliders} shows several more examples of cycles generated by this parenthesis matching rule.
The reason that this rule indeed generates disjoint cycles is that~$f$ is invertible and that $f(x)\neq x$ and $f^2(x)\neq x$.
Indeed, $x$ is obtained from~$f(x)$ by applying the same parenthesis matching procedure as before, but with interpreting the 1s as closing brackets and the 0s are opening brackets instead.

\begin{figure}
\includegraphics[page=1]{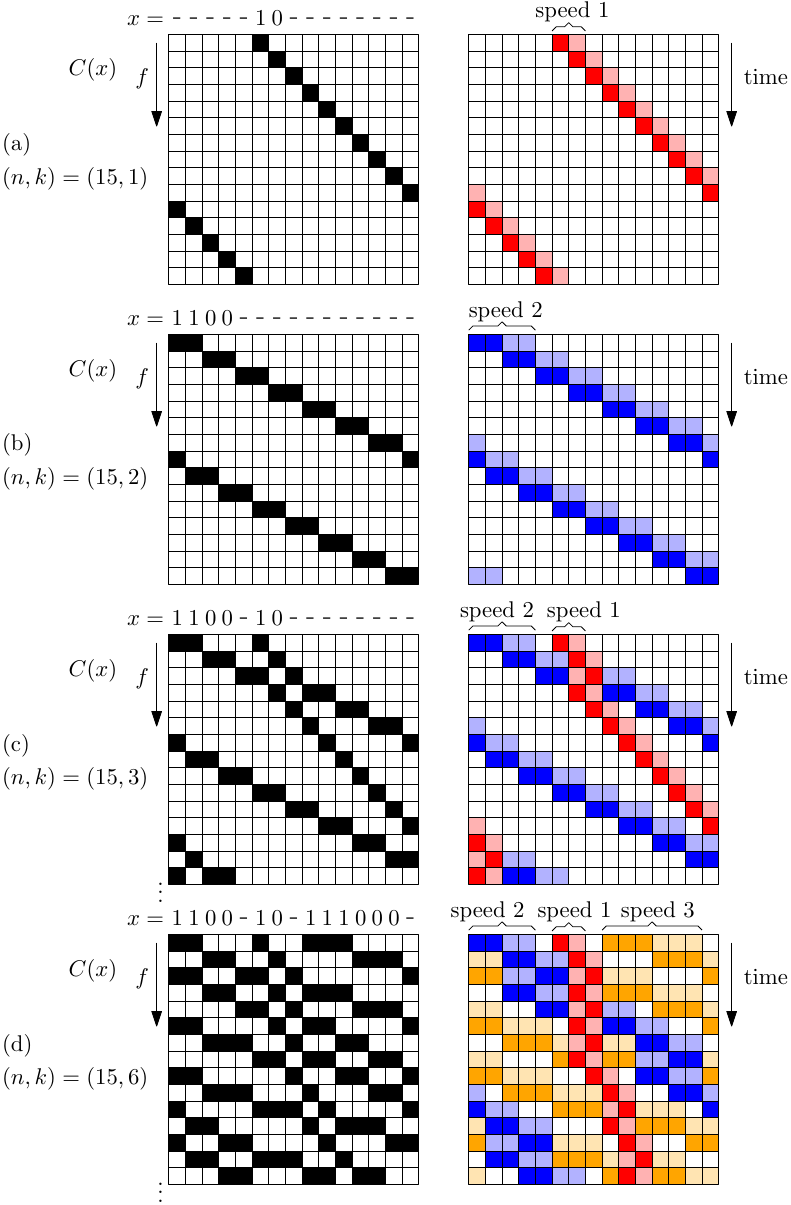}
\caption{Cycles of the factor $\cC_{n,k}$ in several different Kneser graphs~$K(n,k)$.
The cycles in~(a) and~(b) are shown completely, whereas in~(c) and~(d) only the first 15 vertices are shown.
The right hand side shows the interpretation of certain groups of bits as gliders, and their movement over time.
Matched bits belonging to the same glider are colored in the same color, with the opaque filling given to 1-bits, and the transparent filling given to 0-bits.
(a) one glider of speed~1; (b) one glider of speed~2; (c) two gliders with speeds~1 and~2 that participate in an overtaking; (d) three gliders of speeds~1, 2 and~3 that participate in multiple overtakings.
Animations of these examples are available at~\cite{gliders}.
}
\label{fig:gliders}
\end{figure}

\subsubsection{Analysis via gliders}
\label{sec:idea-gliders}

The next key step is to understand the structure of the cycles generated by~$f$, as this is important for joining the cycles to a single Hamilton cycle.
Unfortunately, the number of cycles and their lengths in our factor are governed by intricate number-theoretic phenomena, which we are unable to understand fully.
Instead, we describe the evolution of a bitstring~$x$ under repeated applications of~$f$ combinatorially, which enables us to extract some important cycle properties and invariants (other than the number of cycles and the cycle lengths).
Specifically, we describe this evolution by a kinetic system of multiple gliders that move at different speeds and that interact over time, reminiscent of the gliders in Conway's Game of Life.
This physical interpretation and its analysis are one of the main innovations of this paper.
Specifically, we view each application of~$f$ as one unit of time moving forward.
Furthermore, we partition the matched bits of~$x$ into groups, and each of these groups is called a \emph{glider}.
A glider has a \emph{speed} associated to it, which is given by the number of 1s in its group.
As a consequence of this definition, the sum of speeds of all gliders equals~$k$.
For example, in the cycle shown in Figure~\ref{fig:gliders}~(a), there is a single matched~1 and the corresponding matched~0, and together these two bits form a glider of speed~1 that moves one step to the right in every time step.
Applying $f$ means going down to the next row in the picture, so the time axis points downwards.
Similarly, in Figure~\ref{fig:gliders}~(b), there are two matched~1s and the corresponding two matched~0s, and together these four bits form a glider of speed~2 that moves two steps to the right in every time step.
As we see from these examples, a single glider of speed~$v$ simply moves uniformly, following the basic physics law
\begin{equation*}
s(t)=s(0)+v\cdot t,
\end{equation*}
where $t$ is the time (i.e., the number of applications of~$f$) and $s(t)$ is the position of the glider in the bitstring as a function of time.
The position~$s(t)$ has to be considered modulo~$n$, as bitstrings are considered as cyclic strings and the gliders hence wrap around the boundary.
The situation gets more interesting and complicated when gliders of different speeds interact with each other.
For example, in Figure~\ref{fig:gliders}~(c), there is one glider of speed~2 and one glider of speed~1.
As long as these groups of bits are separated, each glider moves uniformly as before.
However, when the speed~2 glider catches up with the speed~1 glider, an overtaking occurs.
During an overtaking, the faster glider receives a boost, whereas the slower glider is delayed.
This can be captured by augmenting the corresponding equations of motion by introducing additional terms, making them non-uniform.
In the simplest case of two gliders of different speeds, the equations become
\begin{align*}
s_1(t)&=s_1(0)+v_1\cdot t-2v_1 c_{1,2}, \\
s_2(t)&=s_2(0)+v_2\cdot t+2v_1 c_{1,2},
\end{align*}
where the subscript~1 stands for the slower glider and the subscript~2 stands for the faster glider, and the additional variable~$c_{1,2}$ counts the number of overtakings.
Note that the terms~$2v_1c_{1,2}$ occur with opposite signs in both equations, capturing the fact that the faster glider is boosted by the same amount that the slower glider is delayed.
This can be seen as `energy conservation' in the system of gliders.
Overall, the slower glider stands still for two time steps during an overtaking, as $v_1\cdot 2-2v_1\cdot 1=0$, and the faster glider's position changes by an additional amount of $2v_1$ (compared to its movement without overtaking).
For more than two gliders, the equations of motion can be generalized accordingly, by introducing additional overtaking counters between any pair of gliders (see Proposition~\ref{prop:motion}).
Nevertheless, as the reader may appreciate from Figure~\ref{fig:gliders}~(d), in general it is highly nontrivial to recognize from an arbitrary bitstring~$x$ which of its matched bits belong to which glider, and consequently which glider is currently overtaking which other glider.
Note that in general the gliders will not be nicely separated, but will be involved in simultaneous interactions, so that the groups of bits forming the gliders will be interleaved in complicated ways.
Our general rule that achieves the glider partition is based on a recursion that uses an interpretation of~$x$ as a Motzkin path, where every matched~1 becomes an~$\ustep$-step in the Motzkin path, every matched~0 becomes a~$\dstep$-step, and every unmatched~0 becomes a $\fstep$-step (see Section~\ref{sec:glider}).

One important property that we extract from the aforementioned physics interpretation is that the number of gliders and their speeds are invariant along each cycle (see Lemma~\ref{lem:Vx-invariant}).
For example, in Figure~\ref{fig:gliders}~(d), every bitstring along this cycle has three gliders of speeds~1, 2 and~3.
Note in this example that the speeds do not necessarily correspond to the lengths of inclusion maximal runs of consecutive 1s in the bitstrings, due to the interleaving of gliders.
We also use the equations of motion to derive a seemingly innocent, but very crucial property, namely that no glider stands still forever, but will move eventually (see Lemma~\ref{lem:forward}).
Note that the speed~1 glider in Figure~\ref{fig:gliders}~(d) stands still between time steps~2--8, as during those steps it is overtaken once by the speed~2 glider, and twice by the speed~3 glider (wrapping around the boundary).
We establish this fact by linear algebra, by showing that the determinant of the linear systems of equations that governs the gliders' movements is non-singular (see Lemma~\ref{lem:detM}).

For the reader's entertainment, we programmed an interactive animation of gliders over time, and we encourage experimentation with this code, which can be found at~\cite{gliders}.
In particular, this link contains animations of many examples used in figures from our paper, which greatly improves their educational value.

The cycle factor construction discussed before and our analysis via gliders actually work for all $n\geq 2k+1$, not just for $n\geq 2k+3$.
The assumption $n\geq 2k+3$ will become crucial in the next step, though.

\subsubsection{Gluing the cycles together}
\label{sec:idea-gluing}

To join the cycles of our factor to a single Hamilton cycle, we consider a 4-cycle~$D$ that shares two opposite edges with two cycles~$C,C'$ from our factor.
Clearly, the symmetric difference of the edge sets $(C\cup C')\Delta D$ yields a single cycle on the same vertex set as~$C\cup C'$.
We may repeatedly apply such gluing operations, each time reducing the number of cycles in the factor by one, until the resulting factor has a single cycle, namely a Hamilton cycle.
It turns out that the cycle factor defined by~$f$ admits a lot of such gluing 4-cycles.
Note that $K(n,k)$ does not have any 4-cycles for $n=2k+1$, so the assumption $n\geq 2k+2$ is needed here.

The two main technical obstacles we have to overcome are the following:
(a)~All of the 4-cycles used for the gluing must be edge-disjoint, so that none of the gluings interfere with each other.
(b)~We must use sufficiently many gluings to achieve connectivity, i.e., every cycle must be connected to every other cycle via a sequence of gluings.
These two objectives are somewhat conflicting with each other, so satisfying both at the same time is nontrivial.
The final gluings that we use and that satisfy both conditions are described by a set of nine intricate regular expressions (see~\eqref{eq:alphai}).

The 4-cycles that we use for the gluings are based on local modifications of two bitstrings~$x$ and~$y$ that satisfy certain conditions and that lie on two different cycles~$C(x)$ and~$C(y)$ from our factor, by considering the gliders in~$x$ and~$y$.
Specifically, this local modification changes the speed sets of the gliders in~$x$ and~$y$ in a controllable way.
Recall that the speeds of gliders are invariant along each cycle, so these speeds will only change along the gluing 4-cycles.
To control the gluing, we consider the speeds of gliders in a bitstring~$x$ in non-increasing order.
Recall that the sum of speeds equals~$k$, so such a sorted sequence forms a number partition of~$k$.
To establish~(b) we choose gluings that guarantee a lexicographic increase in those number partitions.
This ensures that every cycle is joined, via a sequence of gluings, to a cycle that has the lexicographically largest number partition, namely the number~$k$ itself.
This corresponds to a single glider of the maximum speed~$k$, i.e., to a bitstring~$x$ in which all 1s are consecutive.

\begin{figure}[h!]
\includegraphics[page=3]{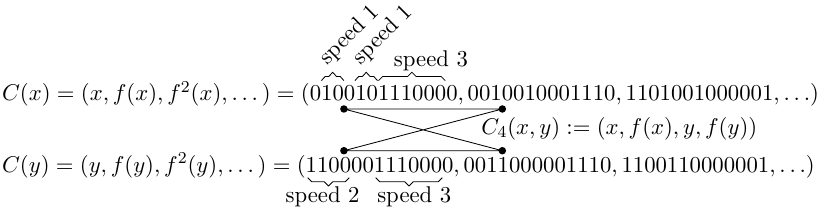}
\caption{Gluing of two cycles from the factor via a 4-cycle in~$K(13,5)$.}
\label{fig:gluing}
\end{figure}

For example, consider the two cycles~$C(x)$ and~$C(y)$ shown in Figure~\ref{fig:gluing}, which can be glued together using the 4-cycle $C_4(x,y):=(x,f(x),y,f(y))$.
Note that in $C(x)$, there are two gliders of speed~1 and one glider of speed~3, whereas in~$C(y)$ there is one glider of speed~2 and one of speed~3.
Consequently, via the gluing we have moved from the number partition~$(3,1,1)$ to the lexicographically larger partition~$(3,2)$.

The general idea for choosing the gluings~$C_4(x,y)$, which can already be seen in this example, is such that in~$x$ we decrease the speed of a glider of the minimum speed by~1, and instead we increase the speed of any other glider by~1, which ensures that the number partition associated with~$y$ is lexicographically larger than that of~$x$.
Unfortunately, it is not always possible to use gluings that guarantee such immediate lexicographic improvement.
In some cases we have to use gluings where a small lexicographic decrease occurs.
It is then argued that a subsequent gluing compensates for this defect such that the overall effect of the resulting sequence of gluings is again a lexicographic improvement.
For example, from a vertex with associated number partition~$(4,4)$, the first gluing may lead to a vertex with number partition~$(4,3,1)$, and the next gluing may lead to~$(5,3)$.
While the step $(4,4)\rightarrow (4,3,1)$ is a lexicographic decrease instead of an increase, overall $(4,4)\rightarrow (4,3,1)\rightarrow (5,3)$ is a lexicographic increase.
In this step of the proof the assumption~$n\geq 2k+3$ finally enters the picture, as it gives us the necessary flexibility in choosing gluings that are guaranteed to achieve this improvement in all cases. 

The arguments so far show that every cycle is connected, via a sequence of gluings, to a cycle in which all 1s are consecutive.
Note however, that there may be several such cycles, depending on the values of~$n$ and~$k$.
Specifically, there are exactly $\gcd(n,k)$ such cycles.
To join those, we observe that the subgraph of~$K(n,k)$ induced by those special cycles is isomorphic to a Cayley graph of~$\mathbb{Z}/n\mathbb{Z}$, which admits many gluing 4-cycles to join them (see Lemma~\ref{lem:join}).

\subsection{Proof of Theorem~\ref{thm:Jnks}}
\label{sec:Jnks}

We first show how Theorem~\ref{thm:Knk} can be used to establish the more general Theorem~\ref{thm:Jnks} quite easily.
Chen and Lih showed the following about generalized Johnson graphs.

\begin{lemma}[{\cite[Thm.~1]{MR888679}}]
\label{lem:JindHC}
If $J(n-1,k-1,s-1)$ and $J(n-1,k,s)$ have a Hamilton cycle, then $J(n,k,s)$ also has a Hamilton cycle.
\end{lemma}

The proof of Lemma~\ref{lem:JindHC} given in~\cite{{MR888679}} is based on a straightforward partitioning of the graph~$J(n,k,s)$ into two subgraphs that are isomorphic to~$J(n-1,k-1,s-1)$ and~$J(n-1,k,s)$.
Specifically, this partition is obtained by considering all vertices (=sets) that contain a fixed element, $n$ say, and those that do not contain it.
One can then join the cycles in the two subgraphs to one, by taking the symmetric difference with a 4-cycle that has one edge in each of the two subgraphs, using the fact that Johnson graphs are edge-transitive, i.e., we can force each of the cycles in the two subgraphs to use this edge from the 4-cycle.
All that is needed now for the proof of Theorem~\ref{thm:Jnks} is the following simple observation.

\begin{lemma}
\label{lem:JindPart}
If $J(n,k,s)$ is a generalized Johnson graph, then either it is a Kneser graph or $J(n-1,k-1,s-1)$ and $J(n-1,k,s)$ are both generalized Johnson graphs.
\end{lemma}

In the proof we will use the aforementioned observation that $J(n,k,s)$ is isomorphic to~$J(n,n-k,n-2k+s)$.

\begin{proof}
Let $k\geq 1$, $0\leq s<k$ and $n\geq 2k-s+\mathbf{1}_{[s=0]}$.
If $s=0$, then $J(n,k,s)=J(n,k,0)=K(n,k)$ is a Kneser graph.
This happens in particular if $k=1$.
If $s>0$ and $n=2k-s$, then $J(n,k,s)=J(n,n-k,n-2k+s)=J(n,k-s,0)=K(n,k-s)$ is also a Kneser graph.
Otherwise, we have $k>1$, $s>0$ and $n>2k-s$, and we consider the graphs $H_1:=J(n-1,k-1,s-1)$ and~$H_0:=J(n-1,k,s)$.
From $k>1$ we obtain $k-1\geq 1$, and from $s>0$ and $s<k$ we obtain that $0\leq s-1<k-1$.
Furthermore, the inequality $n>2k-s$ is equivalent to $n-1>2(k-1)-(s-1)$, which implies $n-1\geq 2(k-1)-(s-1)+1$.
Combining these observations shows that the graph~$H_1$ is indeed a valid generalized Johnson graph.
Similarly, the inequality $n>2k-s$ implies that $n-1\geq 2k-s=2k-s+\mathbf{1}_{[s=0]}$ (since $s>0$), and consequently the graph~$H_0$ is also a valid generalized Johnson graph.
\end{proof}

\begin{proof}[Proof of Theorem~\ref{thm:Jnks}]
We argue by induction on~$n$.
To settle the base cases $n\leq 6$, we argue as follows.
There are 19 triples of values~$(n,k,s)$ with $n\leq 6$ that satisfy the assumptions in the theorem, namely $\{(3,1,0),(4,1,0),(5,1,0),(6,1,0),(5,2,0),(6,2,0),(3,2,1),(4,2,1),(5,2,1),(6,2,1)$, $(5,3,1),(6,3,1),(4,3,2),(5,3,2),(6,3,2),(6,4,2),(5,4,3),(6,4,3),(6,5,4)\}$.
Using that $J(n,k,s)=J(n,n-k,n-2k+s)$ and eliminating the exceptional Petersen graph $J(5,2,0)=J(5,3,1)$ leaves ten triples to be checked, namely $\{(3,1,0),(4,1,0),(5,1,0),(6,1,0),(6,2,0),(4,2,1),(5,2,1)$, $(6,2,1),(6,3,1),(6,3,2)\}$.
The four triples $(n,k,s)\in\{(3,1,0),(4,1,0),(5,1,0),(6,1,0)\}$ yield complete graphs ($k=1$), which clearly have a Hamilton cycle.
The four triples $(n,k,s)\in \{(4,2,1),(5,2,1),(6,2,1),(6,3,2)\}$ are Johnson graphs ($s=k-1$), which have a Hamilton cycle by the results from~\cite{MR0349274}.
The remaining two triples $(n,k,s)\in\{(6,2,0),(6,3,1)\}$ are settled by the results from~\cite{MR888679} ($s=k-2$).

For the induction step consider a generalized Johnson graph~$J(n,k,s)$ with $n\geq 7$.
By Lemma~\ref{lem:JindPart}, $J(n,k,s)$ is either a Kneser graph, or $J(n-1,k-1,s-1)$ and $J(n-1,k,s)$ are both generalized Johnson graphs.
In the first case, Theorem~\ref{thm:Knk} yields that $J(n,k,s)$ has a Hamilton cycle.
In the second case, both $J(n-1,k-1,s-1)$ and $J(n-1,k,s)$ have a Hamilton cycle by induction, and then Lemma~\ref{lem:JindHC} shows that~$J(n,k,s)$ has a Hamilton cycle as well.
\end{proof}

\section{Cycle factor construction}

In this section we describe in detail the construction of a cycle factor in the Kneser graph~$K(n,k)$ outlined in Section~\ref{sec:idea-factor}.
This construction is valid for the entire range of values $n\geq 2k+1$.

\subsection{Preliminaries}

We let $X_{n,k}$ \marginpar{$X_{n,k}$} denote the set of all bitstrings of length~$n$ with exactly $k$ many 1s.
We interpret the vertices of the Kneser graph~$K(n,k)$ as bitstrings in~$X_{n,k}$, by considering the corresponding characteristic vectors.
Every pair of disjoint sets, which is an edge in the Kneser graph, corresponds to a pair of bitstrings that have no 1s at the same positions.
These definitions are illustrated in Figure~\ref{fig:petersen}.

We write $\varepsilon$ for the empty string. \marginpar{$\varepsilon,\ol{x}$}
Moreover, for any bitstring~$x$, we write $\ol{x}$ for the complemented bitstring.
We also write $x\,y$ for the concatenation of the bitstrings~$x$ and~$y$, and $x^a$ for the $a$-fold repetition of~$x$. \marginpar{$x\,y,x^a$}

For integers~$a\leq b$ we define $[a,b]:=\{a,a+1,\ldots,b\}$, and we refer to this set of integers as an \emph{interval}. \marginpar{$[a,b]$}

For a set~$A$ of integers and an integer~$b$ we define $A+b:=\{a+b\mid a\in A\}$. \marginpar{$A+b$}

Throughout this paper, important terminology and symbols are printed on the page boundaries at the place where they are first defined, to facilitate going back and looking up the definitions.

\subsection{Cycle factor construction}
\label{sec:factor}

We consider a bitstring~$x \in X_{n,k}$, and we apply \emph{parenthesis matching} to it, which is a technique developed by Greene and Kleitman~\cite{MR0389608} in the context of symmetric chain partitions of posets.
For this we interpret the 1s in~$x$ as opening brackets and the 0s as closing brackets, and we match closest pairs of opening and closing brackets in the natural way; see Figure~\ref{fig:matching}.
This matching is done \emph{cyclically} across the boundary of~$x$, i.e., $x$ is considered as a cyclic string.
Specifically, we find an opening bracket followed directly (cyclically) by a closing bracket, i.e., a substring~$10$, match them to each other, delete both of them, and repeat.
Note that the resulting matching does not depend on the order in which the procedure is applied in situations when multiple substrings~$10$ are present.

The parenthesis matching procedure can be described equivalently as follows:
For every 1-bit in~$x$, we consider the shortest (cyclic) substring starting from this bit to the right that contains the same number of~0s as~1s, and we match it to the last 0-bit of this substring.

As $x$ is considered as a cyclic string, in the following we take indices in~$x$ modulo~$n$, with $1,\ldots,n$ as representatives of the equivalence classes.
As $n \geq 2k+1$, there are more 0s than 1s in~$x$, and consequently every~1 is matched to some~0, but not every~0 is matched to a~1.
Whenever we want to emphasize that we consider a bitstring~$x$ with parenthesis matching applied to it, we write every unmatched~0 in~$x$ as $\hyph$.
For example, we write $x=001100001=0\hyph1100\hyph\hyph1$.

\begin{figure}[b!]
\includegraphics[page=1]{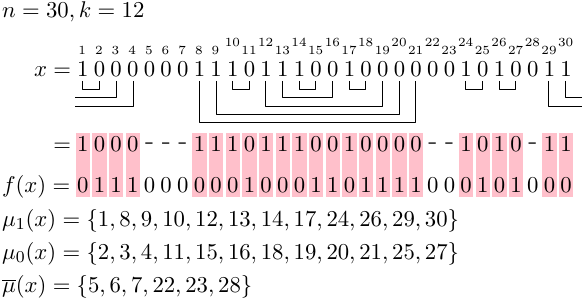}
\caption{Parenthesis matching for a bitstring $x\in X_{30,12}$.
Matched pairs of bits are indicated by square brackets, and bits that are complemented to obtain~$f(x)$ are highlighted.
}
\label{fig:matching}
\end{figure}

We let $\mu(x)\seq [n]$ \marginpar{$\mu(x),\ol{\mu}(x)$} denote the set of positions of bits that are matched in~$x$, and we write $\ol{\mu}(x):=[n]\setminus \mu(x)$ for the positions of unmatched bits.
Moreover, we partition $\mu(x)$ into the sets~$\mu_1(x)$ and~$\mu_0(x)$ \marginpar{$\mu_1(x),\mu_0(x)$} of the positions of matched~1s and~0s, respectively.
By these definitions, the sets~$\mu_1(x)$, $\mu_0(x)$ and $\ol{\mu}(x)$ are all disjoint, their union is~$[n]$, the union of~$\mu_1(x)$ and~$\mu_0(x)$ is~$\mu(x)$, and the sizes of the four sets $\mu_1(x)$, $\mu_0(x)$, $\mu(x)$, and $\ol{\mu}(x)$ are $k$, $k$, $2k$, and $n-2k$, respectively.
As every 1-bit is matched, $\mu_1(x)$ is the set of positions of \emph{all} 1s in~$x$.

For any $x\in X_{n,k}$, we let $f(x)\in X_{n,k}$ \marginpar{$f(x)$} denote the bitstring obtained from~$x$ by complementing all matched bits, i.e., the bits at all positions in~$\mu(x)$; see Figure~\ref{fig:matching}.

By this definition we have $\mu_1(f(x)) = \mu_0(x)$.
Consequently, $x$ and~$f(x)$ have no 1s at the same positions, i.e., they are characteristic vectors of disjoint sets.
It follows that $(x,f(x))$ is an edge in the Kneser graph~$K(n,k)$.

\begin{subequations}
\label{eq:factor}
For any $x\in X_{n,k}$ we define a sequence of vertices in~$K(n,k)$ by
\marginpar{$C(x)$}
\begin{equation}
\label{eq:Cx}
C(x):= \big(x,f(x),f^2(x),\ldots\big),
\end{equation}
i.e., we repeatedly apply~$f$ to~$x$ until we obtain~$x$ again.

\begin{lemma}
\label{lem:Cx}
Let $n \geq 2k+1$.
For any $x\in X_{n,k}$, the sequence $C(x)$ defined in \eqref{eq:Cx} describes a cycle of length at least~3 in the Kneser graph~$K(n,k)$.
\end{lemma}

\begin{proof}
We first argue that the mapping~$f$ is invertible.
Specifically, the inverse mapping is obtained by interchanging the roles of~0s and~1s, i.e., given a bitstring~$x$, we interpret the 0s in~$x$ as opening brackets and the 1s as closing brackets, and we match closest pairs of opening and closing brackets in the natural way (cyclically across the boundary).
Then $f^{-1}(x)$ is the bitstring obtained from~$x$ by complementing all matched bits.

As $f$ is invertible, there are no two distinct bitstrings $x,x'\in X_{n,k}$ with $f(x)= f(x')$.
Furthermore, as the set $X_{n,k}$ is finite, the first duplicate bitstring in the sequence $C(x)$ is the first string~$x$, so the sequence $C(x)$ is indeed cyclic.

Now consider three consecutive bitstrings~$x$, $f(x)$ and $f^2(x)$ in the sequence~$C(x)$.
To complete the proof of the lemma, we show that $x\neq f(x)$ and $x\neq f^2(x)$.
For this we identify a position $i\in[n]$ such that $i\in\ol{\mu}(x)$, $i\in\mu_0(f(x))$ and $i\in\mu_1(f^2(x))$.
By the assumption $n\geq 2k+1$ the set~$\ol{\mu}(x)$ is nonempty, and we let $i\in\ol{\mu}(x)$ be such that $i-1\in\mu(x)$.
Specifically, we have $i-1\in\mu_0(x)$, and consequently $i-1\in\mu_1(f(x))$, which implies that the 0-bit at position~$i$ in~$f(x)$ is matched to the 1-bit to its left, so $i\in\mu_0(f(x))$.
It follows that~$i\in\mu_1(f^2(x))$, as claimed.
\end{proof}

By the definition~\eqref{eq:Cx}, for any two vertices~$x,y\in X_{n,k}$, the cycles~$C(x)$ and~$C(y)$ have either no vertices in common or all of them.
We may thus define a \emph{cycle factor} in~$K(n,k)$ by
\marginpar{$\cC_{n,k}$}
\begin{equation}
\cC_{n,k}:=\{C(x)\mid x\in X_{n,k}\}.
\end{equation}
\end{subequations}

For example, for the Petersen graph $K(5,2)$, for $x:=10100$ and $x':=11000$ we get
\begin{align*}
C(x)&=(1010\hyph,\hyph1010,0\hyph101,10\hyph10,010\hyph1), \\
C(x')&=(1100\hyph,0\hyph110,100\hyph1,\hyph1100,00\hyph11),
\end{align*}
and $\cC_{5,2}=\{C(x),C(x')\}$.
More examples are shown in Figure~\ref{fig:gliders}.

\section{Analysis via gliders: static properties}

In this and the next section we analyze various properties and invariants of the cycles of the factor~$\cC_{n,k}$ defined in the previous section.
This analysis uses a system of interacting gliders, as sketched in Section~\ref{sec:idea-gliders}.
Specifically, we first focus on the `static' properties of gliders for one fixed vertex~$x\in X_{n,k}$, and later on the `dynamic' properties of gliders when going from~$x$ to~$f(x)$ along the cycle~$C(x)$, describing how gliders in~$x$ move to become gliders in~$f(x)$.
All of these results hold for the entire range of values $n\geq 2k+1$.

\subsection{Motzkin paths and Dyck paths}

We identify a bitstring $x\in X_{n,k}$ with a Motzkin path \marginpar{Motzkin path} in the integer lattice~$\mathbb{Z}^2$ in the following way; see Figure~\ref{fig:motzkin}.
We apply parenthesis matching to~$x$ and we read the bits of~$x$ from left to right.
Every matched~1 is drawn as an $\ustep$-step, every matched~0 as a $\dstep$-step, and every unmatched~0 as a $\fstep$-step, and these steps change the current coordinate by $(+1,+1)$, $(+1,-1)$, or $(+1,0)$, respectively.
In other words, the lattice path corresponding to~$x$ has $\ustep$-steps at the positions in~$\mu_1(x)$, $\dstep$-steps at the positions in~$\mu_0(x)$, and $\fstep$-steps at the positions in~$\ol{\mu}(x)$.
To define the absolute position of this lattice path, it suffices to specify the coordinate of one point on it.
Specifically, if $i$ is the position of the first unmatched 0-bit in~$x$, then this $\fstep$-step starts at the coordinate~$(i-1,0)$.

In every substring of~$x$ for which every prefix has at least as many~1s as~0s, every 0-bit is matched to some 1-bit to its left in this substring.
As a consequence, the Motzkin path~$x$ never moves below the abscissa and all $\fstep$-steps of~$x$ lie on the abscissa.
It follows that in the above definition, we can choose~$i$ as the position of \emph{any} unmatched 0-bit in~$x$, not necessarily the first one.

\begin{figure}
\includegraphics{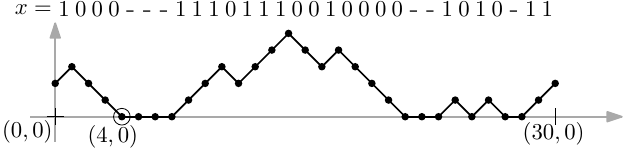}
\caption{The Motzkin path corresponding to the bitstring~$x$ from Figure~\ref{fig:matching}.
The starting point~$(4,0)$ of the first $\fstep$-step of~$x$ is marked.
}
\label{fig:motzkin}
\end{figure}

For any integer $\ell\geq 0$, we write $D_\ell$ for the set of \emph{Dyck words} \marginpar{Dyck word} of length~$2\ell$, i.e., bitstrings with $\ell$ many 1s and~0s that have at least as many~1s as~0s in every prefix.
Note that we have $D_0=\{\varepsilon\}$ and $D_\ell=\{1\,u\,0\,v \mid u\in D_i, v\in D_{\ell-i-1}, \text{ and } i=0,\ldots,\ell-1\}$ for $\ell>0$.
Moreover, we write~$D$ \marginpar{$D,D'$} for the set of Dyck words of arbitrary length, i.e., $D := \bigcup_{\ell\geq 0} D_\ell$.
Lastly, we use $D'\seq D$ to denote the set of bitstrings that have strictly more 1s than~0s in every proper nonempty prefix, i.e., we have $D' = \{1\,u\,0 \mid u \in D\}$.

Observe that a substring~$y$ of~$x$ with $y\in D$ corresponds to a subpath of the Motzkin path~$x$ that has the same number of $\ustep$-steps and $\dstep$-steps, does not contain any $\fstep$-steps, and that never moves below the height of the starting point, i.e., it is a Dyck subpath.

\subsection{The infinite string~\texorpdfstring{$\wh{x}$}{xhat}}
\label{sec:inf-string}

We let~$\wh{x}$ \marginpar{$\wh{x}$} be the bitstring obtained by extending~$x$ by copies of itself infinitely in both directions.
We can think of $\wh{x}$ as the string obtained by unrolling the cyclic string~$x$.
As the parenthesis matching procedure interprets~$x$ as a cyclic string, the notion of matched and unmatched bits extends in the natural way from~$x$ to~$\wh{x}$.
Consequently, the Motzkin path corresponding to~$\wh{x}$ is obtained by concatenating infinitely many copies of the Motzkin path of~$x$.
In this way, the indices of bits in~$\wh{x}$ or of steps on the corresponding Motzkin path continue beyond the indices~$1,\ldots,n$ used in~$x$, i.e., these indices continue with $0,-1,-2,\ldots,$ to the left, and $n+1,n+2,\ldots$ to the right, and in~$\wh{x}$ we see the same bits/steps in each equivalence class of indices modulo~$n$.
We can translate the indices in~$\wh{x}$ back to~$x$ simply by considering them modulo~$n$, with $1,\ldots,n$ as representatives of the equivalence classes.

The motivation for considering the infinite Motzkin path~$\wh{x}$ in addition to the finite path~$x$ is that we aim to partition certain steps of~$\wh{x}$ into groups, which we will call~\emph{gliders}.
We do this for every vertex along the cycle~$C(x)$, with the goal of tracking the movement of gliders along~$C(x)$.
By~\eqref{eq:Cx}, moving one step along the cycle~$C(x)$ corresponds to one application of~$f$, which we interpret as time moving forward by one unit; recall Figure~\ref{fig:gliders}.
As $\wh{x}$ has periodicity~$n$, each glider repeats periodically every $n$ steps along~$\wh{x}$; see Figure~\ref{fig:position}.
However, for formulating continuous equations of motions for the gliders, it is crucial to treat these periodic copies as separate entities that continually move towards $+\infty$ along the cycle~$C(x)$, and not to treat them as a single entity that moves and wraps around the boundary of~$x$.
So in the infinite string~$\wh{x}$, an infinite periodic set of gliders continually moves to the right over time, and we consider them through the finite `window'~$x$, in which they appear to wrap around the boundary.

\subsection{Hills and valleys}

We refer to any Dyck subpath~$y\in D$ of~$\wh{x}$ as a \emph{hill}, \marginpar{hill, valley} and to any complemented Dyck subpath~$y$ with $\ol{y}\in D$ as a \emph{valley}.
If a hill~$y$ of~$\wh{x}$ touches the abscissa, we refer to it as a \emph{base hill}. \marginpar{base hill}
Clearly, a base hill~$y$ starts and ends at the abscissa, but it may touch it more than twice, i.e, $y$ is not necessarily from the set~$D'$.
We define the \emph{height} \marginpar{height} of a hill~$y$ in~$x$ as the difference between the ordinates of its highest point and its starting point.
Similarly, the \emph{depth} \marginpar{depth} of a valley~$y$ in~$x$ is defined as the difference between the ordinates of its starting point and its lowest point.

We consider a hill~$y\in D'$ in~$\wh{x}$ of height~$h$, and we let $p$ denote the leftmost highest point of~$y$ on the Motzkin path.
The hill~$y$ can be decomposed uniquely as
\begin{equation}
\label{eq:hill}
y = \begin{cases}
1\, 0 & \text{if } h=1, \\
1 \,u_1 \, 1 \, u_2 \,\cdots\, 1 \, u_{h-2} \, 1\, 1 \, v_0 \, 0 \, v_1 \, 0 \,\cdots\, 0 \, v_{h-2} \, 0 \, 0, & \\
\hspace{12mm} \text{ with } u_1,\ldots,u_{h-2},\ol{v_0},\ldots,\ol{v_{h-2}} \in D, & \text{if } h\geq 2, \\
\end{cases}
\end{equation}
i.e., the~$u_i$ are inclusion maximal hills in~$y$ to the left of the point~$p$, and the~$v_i$ are inclusion maximal valleys in~$y$ to the right of the point~$p$; see Figure~\ref{fig:hill}.
We refer to the hill~$u_i$ as the \emph{$i$th bulge of~$y$} \marginpar{bulge} and to the valley~$v_i$ as the \emph{$i$th dent of~$y$}. \marginpar{dent}
If $h=1$ then $y=1\,0$ has neither bulges nor dents.
From~\eqref{eq:hill} we obtain the following lemma.

\begin{lemma}
\label{lem:bulge-dent}
For any hill $y\in D'$ of height~$h$ in~$\wh{x}$, the $i$th bulge of~$y$ has height at most~$h-1-i$ for all $i=1,\ldots,h-2$, and the $i$th dent of~$y$ has depth at most~$h-1-i$ for all $i=0,\ldots,h-2$.
\end{lemma}

\begin{figure}[h!]
\includegraphics{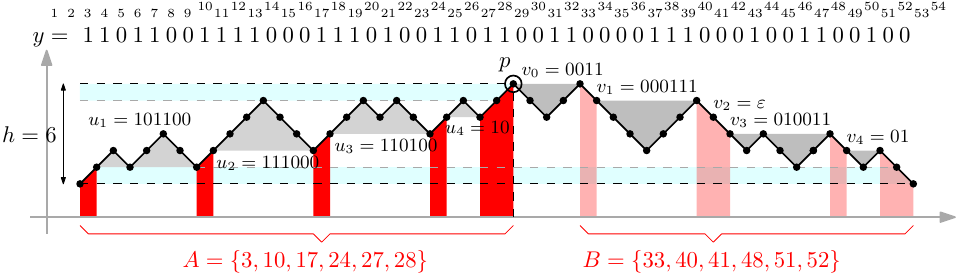}
\caption{Decomposition of a hill $y\in D'$ of height~$h=6$ contained in some larger Motzkin path~$\wh{x}$ into bulges~$u_i$ and dents~$v_i$.
The leftmost highest point~$p$ of~$y$ is marked.
}
\label{fig:hill}
\end{figure}

\subsection{Glider partition}
\label{sec:glider}

In the following we define a function~$\Gamma(x)$ that recursively computes a partition of the set of indices of all $\ustep$-steps and $\dstep$-steps in the infinite Motzkin path~$\wh{x}$.
This corresponds to the set of all indices of matched bits in the bitstring~$\wh{x}$, i.e., to the set~$\mu(\wh{x})\seq \mathbb{Z}$.
Moreover, the sets of the partition are grouped into pairs~$(A,B)$ with $|A|=|B|$ and $\max A<\min B$, and either $A$ contains positions of~$\ustep$-steps of~$\wh{x}$ and $B$ contains positions of~$\dstep$-steps, or vice versa.
We will refer to every such pair~$(A,B)$ as a \emph{glider} (the formal definition is given below), and we think of the steps of the Motzkin path~$\wh{x}$ at the positions in~$A$ and~$B$ as the steps that `belong' to the glider~$(A,B)$.
Every $\ustep$-step and $\dstep$-step of $\wh{x}$ belongs to precisely one glider, whereas $\fstep$-steps do not belong to any glider; see Figure~\ref{fig:Gamma}.

\begin{figure}
\makebox[0cm]{ 
\includegraphics[scale=0.9]{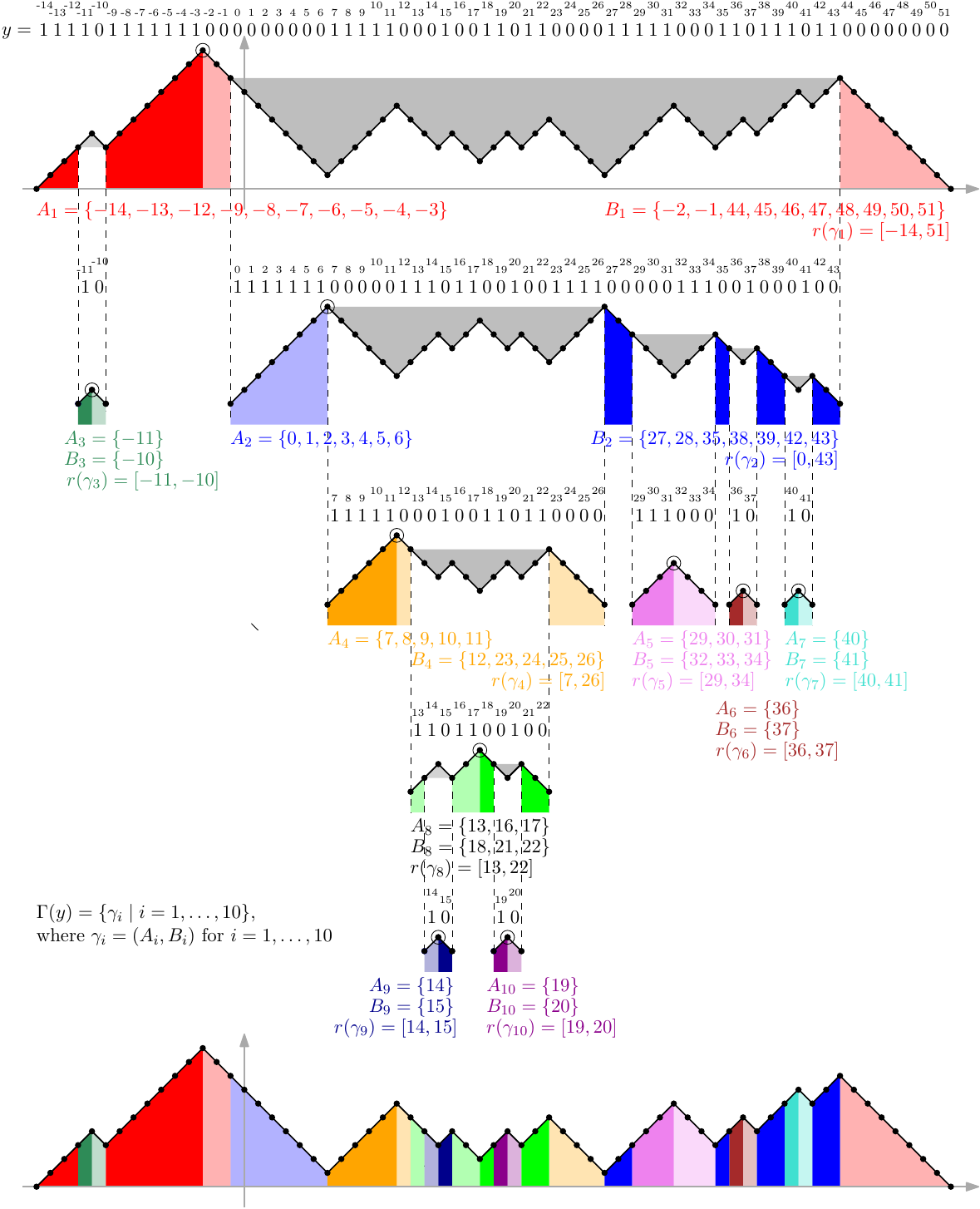}
}
\caption{Glider partitioning in a hill $y\in D'$ contained in some larger Motzkin path~$\wh{x}$.}
\label{fig:Gamma}
\end{figure}

We first define
\begin{subequations}
\label{eq:Gamma}
\marginpar{$\Gamma(x)$}
\begin{equation}
\label{eq:Gamma-base}
\Gamma(x):=\bigcup_{y \text{ base hill in } \wh{x}} \Gamma(y).
\end{equation}

For any hill~$y\in D\setminus D'$ in~$\wh{x}$, if $y=\varepsilon$ we define
\begin{equation}
\Gamma(y):=\emptyset,
\end{equation}
and if $y\neq\varepsilon$ we consider the partition $y=y_1\cdots y_\ell$ with $y_1,\ldots,y_\ell\in D'$ and define
\begin{equation}
\label{eq:Gamma-multi}
\Gamma(y):=\bigcup_{i=1}^\ell \Gamma(y_i).
\end{equation}
The interesting step of the recursion happens for hills $y\in D'$ in~$\wh{x}$, for which we partition~$y$ uniquely as in~\eqref{eq:hill}, and define
\begin{equation}
\label{eq:Gamma-hill}
\Gamma(y):=\Big(\bigcup \limits_{i=1}^{h-2} \Gamma({u_i}) \, \cup \, \bigcup \limits_{i=0}^{h-2} \Gamma(\ol{v_i})\Big) \, \cup \, \{(A, B)\},
\end{equation}
\end{subequations}
where~$A$ and~$B$ are the sets of indices of the~1s and~0s of~$y$, respectively, that do not belong to any of the bulges~$u_i$ or the dents~$v_i$ in the decomposition~\eqref{eq:hill}; see Figure~\ref{fig:hill}.
Recall that if $h=1$, then $y=1\,0$ has no bulges or dents, so in this case the two unions in~\eqref{eq:Gamma-hill} indexed by~$i$ are empty and then $A$ and~$B$ are singleton sets containing the positions of the~1 and the~0 in~$y$, respectively.
In general, the indices in~$A$ and~$B$ are absolute with respect to the Motzkin path~$\wh{x}$ that contains~$y$ as a subpath.
We refer to any pair~$(A,B)\in\Gamma(x)$ computed in~\eqref{eq:Gamma-hill} as a \emph{glider}.  \marginpar{glider}

The complemented strings~$\ol{v_i}$ on the right hand side of~\eqref{eq:Gamma-hill} need further explanation.
By definition we have $\ol{v_i}\in D$, i.e., the subpath~$v_i$ in~$\wh{x}$ is a valley.
Therefore, to compute~$\Gamma(\ol{v_i})$ we apply~$\Gamma$ to the hill~$\ol{v_i}$, obtained by complementing the valley~$v_i$, without changing any other steps of~$\wh{x}$, as they are irrelevant for the computation of~$\Gamma(\ol{v_i})$; see Figure~\ref{fig:Gamma}.
Note that the vertical positions of the steps of~$\wh{x}$ or its subpaths are irrelevant for the definition of~$\Gamma$, but what matters are their indices on the horizontal axis (as they enter the sets~$A$ and~$B$ in~\eqref{eq:Gamma-hill}), which are not modified by the complementation.
However, complementation changes the roles of 0s and 1s, so for any glider $(A,B)\in\Gamma(x)$, $A$ and $B$ are either sets of indices of $\ustep$-steps and $\dstep$-steps, respectively, of some hill in the original Motzkin path~$\wh{x}$, or sets of indices of $\dstep$-steps and $\ustep$-steps, respectively, of some valley in the original path~$\wh{x}$.
What is important is that $\max A<\min B$, i.e., all steps in~$A$ are to the left of all steps in~$B$.

\subsection{Range of gliders and equivalence classes}

We define the \emph{range} of a glider~$\gamma=(A,B)\in\Gamma(x)$ as the interval
\marginpar{range $r(\gamma)$}
\begin{equation*}
r(\gamma):=[\min A, \max B]=\{\min A,\min A+1,\ldots,\max B\}\seq\mathbb{Z}.
\end{equation*}

Moreover, we write $\wh{x}_{r(\gamma)}$ for the subpath of~$\wh{x}$ on the interval~$r(\gamma)$.
This subpath contains all steps from~$A$ and~$B$, plus possibly steps of other gliders, no~$\fstep$-steps at all, and it starts with a step from~$A$ and ends with a step from~$B$.
We refer to a glider $\gamma\in\Gamma(x)$ as \emph{non-inverted} if $\wh{x}_{r(\gamma)}$ is a hill in~$\wh{x}$, and as \emph{inverted} \marginpar{inverted} if $\wh{x}_{r(\gamma)}$ is a valley.
Non-inverted gliders~$\gamma=(A,B)$ have $\ustep$-steps of~$\wh{x}$ at the positions in~$A$ and $\dstep$-steps at the positions in~$B$, and for inverted gliders the situation is reversed.
For example, in Figure~\ref{fig:Gamma}, the gliders $\gamma_1,\gamma_3,\gamma_4,\gamma_5,\gamma_6,\gamma_7,\gamma_{10}$ are non-inverted, whereas $\gamma_2,\gamma_8,\gamma_9$ are inverted.

Clearly, as~$\wh{x}$ is an infinite string containing infinitely many hills, the set~$\Gamma(x)$ is an infinite set.
However, as $\wh{x}$ has periodicity~$n$, we can partition the set~$\Gamma(x)$ into finitely many equivalence classes of gliders.
Specifically, for any two gliders $(A,B),(A',B')\in\Gamma(x)$, we write $(A,B)\sim (A',B')$ if these sets are the same modulo~$n$, i.e., $A'=A+j\cdot n$ and $B'=B+j\cdot n$ for some~$j\in \mathbb{Z}$.
For any glider $\gamma\in\Gamma(x)$, its equivalence class is denoted by~$[\gamma]$ and the set of all equivalence classes by
\marginpar{$\Gamma(x)/{\sim}$}
\begin{equation}
\label{eq:equiv-classes}
\Gamma(x)/{\sim}:=\{[\gamma] \mid \gamma \in \Gamma(x)\};
\end{equation}
see Figure~\ref{fig:position}.
We also define
\marginpar{$\nu(x)$}
\begin{equation}
\label{eq:nu}
\nu(x):=|\Gamma(x)/{\sim}|,
\end{equation}
and we simply refer to this quantity as the \emph{number of gliders}.

\subsection{Position and speed of gliders}
\label{sec:pos-speed}

\begin{figure}[b!]
\makebox[0cm]{ 
\includegraphics{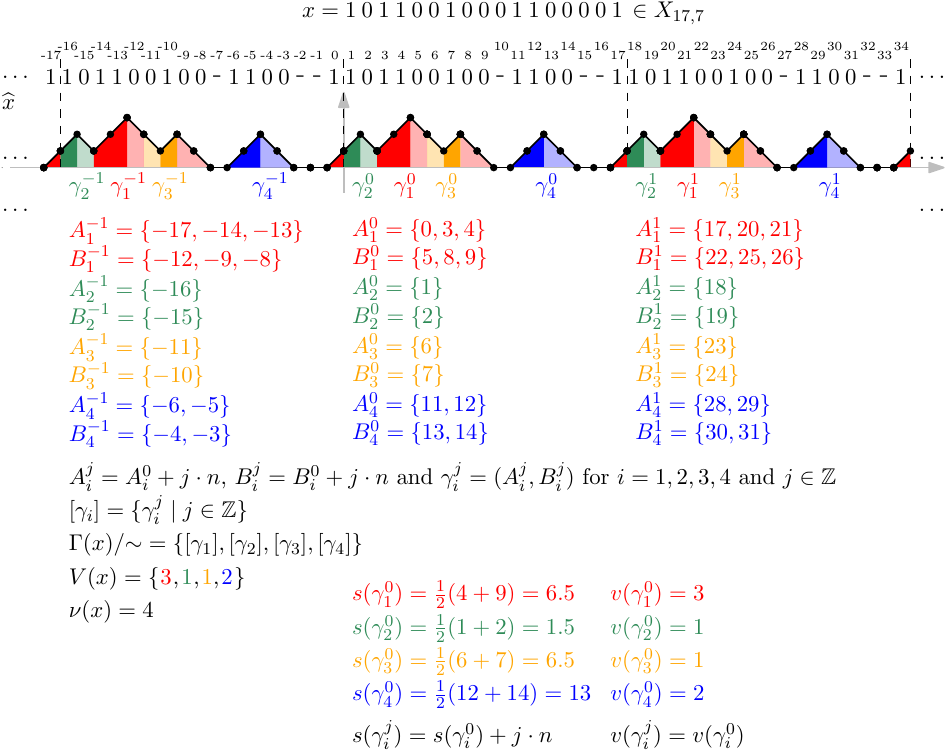}
}
\caption{Equivalence classes of gliders in~$\wh{x}$, and their position and speed.}
\label{fig:position}
\end{figure}

For any glider $\gamma=(A,B)\in\Gamma(x)$ we define
\marginpar{$s_1(\gamma),s_2(\gamma)$}
\begin{subequations}
\begin{equation}
\label{eq:s12}
s_1(\gamma):=\max A \quad \text{and} \quad s_2(\gamma):=\max B.
\end{equation}
Using these, the \emph{position of~$\gamma$} is defined as
\marginpar{position $s(\gamma)$}
\begin{equation}
\label{eq:pos}
s(\gamma):=\frac{1}{2} \big(s_1(\gamma) + s_2(\gamma)\big)=\frac{1}{2} \big({\max A} + \max B\big).
\end{equation}
\end{subequations}

In words, the position~$s(\gamma)$ is the average position of the rightmost $\ustep$-step and the rightmost $\dstep$-step that belong to~$\gamma$, which is true regardless of whether~$\gamma$ is inverted or not.
Clearly, the numbers $s(\gamma)$, $\gamma\in\Gamma(x)$, are either integers or half-integers; see Figure~\ref{fig:position}.
Also note from the example in the figure that two distinct gliders may have the same position~$s(\gamma)$, whereas the integers~$s_1(\gamma)$ and~$s_2(\gamma)$ are unique to the glider~$\gamma$.
For technical reasons we also define
\marginpar{$s_0(\gamma)$}
\begin{equation*}
s_0(\gamma):=\min A,
\end{equation*}
so we have $r(\gamma)=[s_0(\gamma),s_2(\gamma)]$.

The \emph{speed} of $\gamma=(A,B)\in\Gamma(x)$ is defined as
\marginpar{speed $v(\gamma)$}
\begin{equation}
\label{eq:speed}
v(\gamma):=|A|=|B|.
\end{equation}
Note that the speed equals the height of the hill~$\wh{x}_{r(\gamma)}$ if~$\gamma$ is non-inverted, or the depth of the valley~$\wh{x}_{r(\gamma)}$ if~$\gamma$ is inverted.
We also define the \emph{speed set}
\marginpar{speed set}
\marginpar{$V(x)$}
\begin{equation}
\label{eq:Vx}
V(x):=\{v(\gamma) \mid [\gamma] \in \Gamma(x)/{\sim}\},
\end{equation}
which we define as a multiset of size~$\nu(x)$, as gliders from distinct equivalence classes may have the same speed.
Note that
\begin{equation}
\label{eq:sum-speeds}
\sum_{v\in V(x)} v=k,
\end{equation}
as the sum of speeds equals the number of~1s in~$x$.
In other words, the speed set can be seen as a number partition of~$k$.

\subsection{Properties of the speed set}

In the following, we establish a few combinatorial properties of the speed set~$V(x)$.
Specifically, Lemma~\ref{lem:VW} below asserts that the speed set can be computed more directly, without invoking the glider partitioning recursion described in Section~\ref{sec:glider}.
Furthermore, Lemma~\ref{lem:trans} asserts that the number of gliders~$\nu(x)=|V(x)|$ is given by the number of descents in~$x$.

Let $X$ be a nonempty multiset of positive integers.
We write $X\oplus 1$ \marginpar{$X\oplus 1$} for the multiset obtained by incrementing one of the largest elements of~$X$ by~1.
For example, for $X=\{4,4,4,2,1,1\}$ we have $X\oplus 1=\{5,4,4,2,1,1\}$.

\begin{figure}
\makebox[0cm]{ 
\includegraphics[page=1]{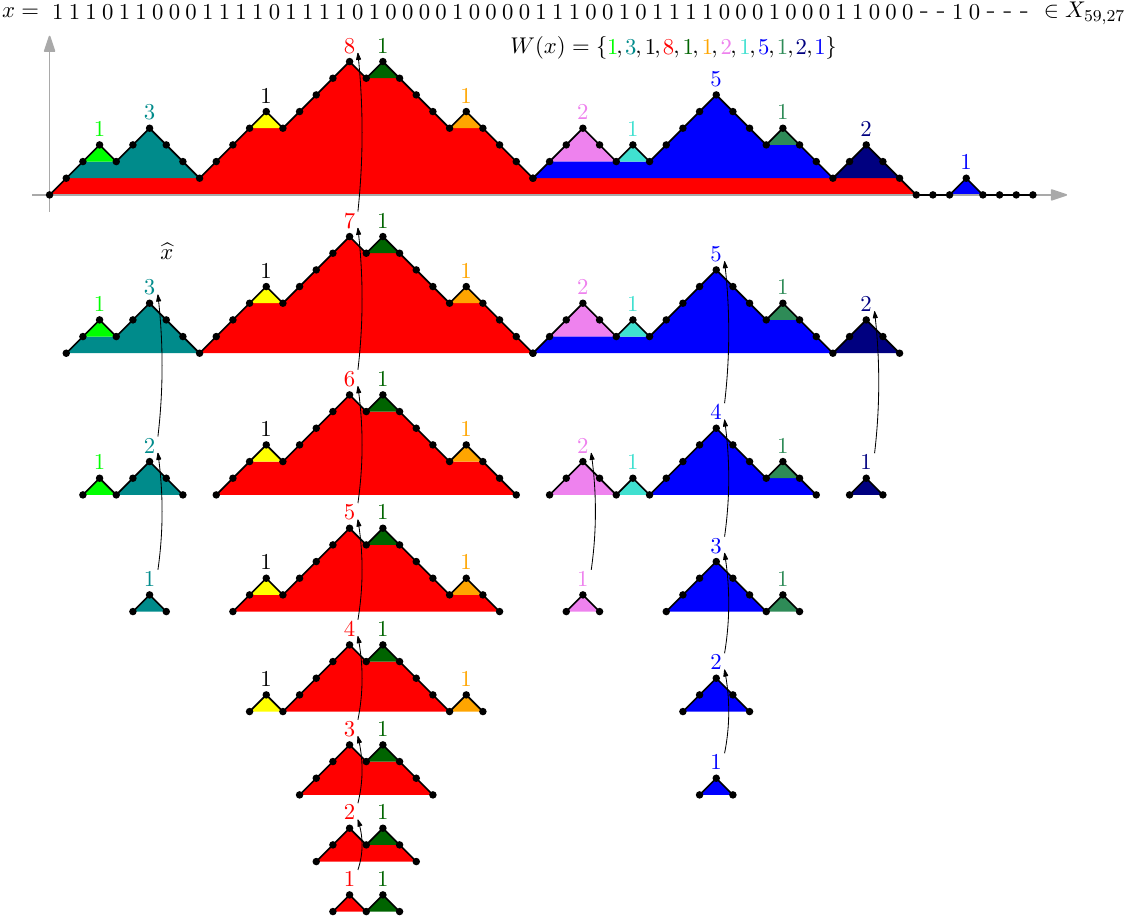}
}
\caption{Illustration of the recursive definition of the multiset~$W(x)$.
The vertical arrows show the values being incremented by the $\oplus 1$ operation.}
\label{fig:speed}
\end{figure}

In a bitstring $x\in X_{n,k}$ under parenthesis matching, we refer to any inclusion maximal (cyclic) substring of matched bits as a \emph{block}. \marginpar{block}
We define a multiset of integers~$W(x)$ as follows; see Figure~\ref{fig:speed}.
We first define
\marginpar{$W(x)$}
\begin{subequations}
\label{eq:Wx}
\begin{equation}
W(x):=\bigcup_{\text{$y$ block in~$x$}} W(y).
\end{equation}
For any $y\in D\setminus D'$, $y\neq \varepsilon$, we consider the partition $y=y_1\cdots y_\ell$ with $y_1,\ldots,y_\ell \in D'$ and define
\begin{equation}
\label{eq:Wx-multi}
W(y):=\bigcup_{i=1}^\ell W(y_i).
\end{equation}
Lastly, for $y\in D'$ we have $y=1\,u\,0$ with $u\in D$ and define
\begin{equation}
\label{eq:Wx-hill}
W(y):=\begin{cases}
\{1\} & \text{if } u=\varepsilon, \\
W(u)\oplus 1 & \text{if } u\neq \varepsilon.
\end{cases}
\end{equation}
\end{subequations}

\begin{lemma}
\label{lem:VW}
For any $x\in X_{n,k}$, the multisets~$V(x)$ and~$W(x)$ defined in~\eqref{eq:Vx} and~\eqref{eq:Wx}, respectively, are the same, i.e., we have $V(x)=W(x)$.
\end{lemma}

The proof of Lemma~\ref{lem:VW} is split into two auxiliary statements, Lemmas~\ref{lem:WT} and~\ref{lem:VT}, below.
To state and prove those, we need some additional preparations.

Consider a lattice path~$y$ consisting of~$\ustep$-steps and~$\dstep$-steps.
We write~$U(y)$ \marginpar{$U(y)$} for the set of positions of all~$\ustep$-steps of~$y$.
For any $j\in U(y)$, the \emph{staircase~$S(j)$} \marginpar{staircase $S(j)$} is the maximum size set $\{j=j_1<\cdots<j_h\}\seq U(y)$ such that the subpath~$u_i$ of~$y$ strictly between positions~$j_i$ and~$j_{i+1}$ is a hill, i.e., $u_i\in D$, of height at most~$h-1-i$, for all $i=1,\ldots,h-1$.
In particular, $u_{h-1}$ has height at most~0, i.e., $u_{h-1}=\varepsilon$.

Note that any two staircases are either disjoint or one is contained in the other as a suffix when ordering the elements increasingly.
A staircase is \emph{full} \marginpar{full} if it is not contained in another staircase.
For example, in Figure~\ref{fig:stair} we have $S(2)\supseteq S(5)\supseteq S(6)$, so $S(5)$ and~$S(6)$ are not full, but $S(2)$ is full.
For any~$j\in U(y)$ we have $j\in S(j)$ by definition, and therefore $j$ is contained in exactly one full staircase~$S(i)$ for some $i\leq j$.
Consequently, the set of full staircases partitions the set~$U(y)$; see Figure~\ref{fig:stair}.
We write $S(x)$ \marginpar{$S(x)$} for the union of all full staircases over all blocks~$y\in D$ in~$x$.
We also define the multiset \marginpar{$T(x)$}
\begin{equation}
\label{eq:Tx}
T(x):=\{|S|\mid S\in S(x)\}.
\end{equation}

\begin{figure}
\makebox[0cm]{ 
\includegraphics[page=2]{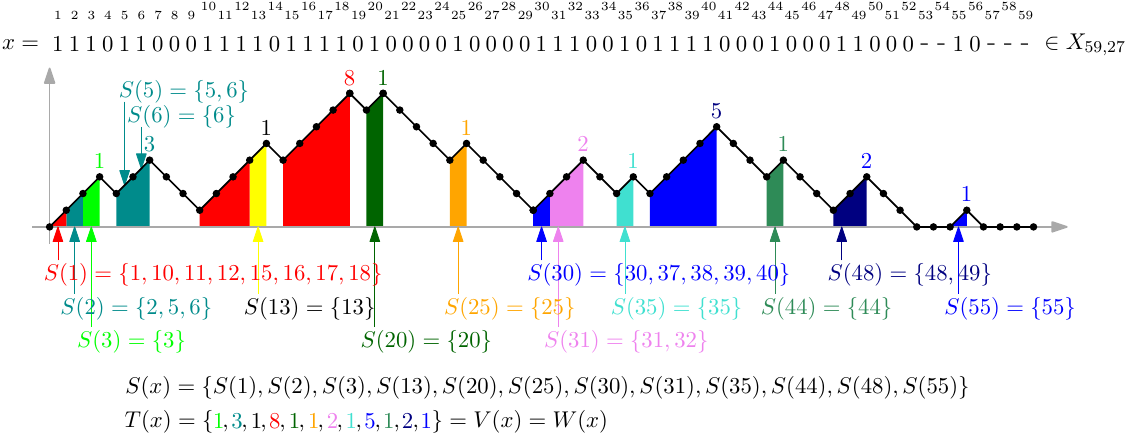}
}
\caption{Illustration of staircases.
All shown staircases are full, except $S(5)$ and~$S(6)$.
This continues the example from Figure~\ref{fig:speed}.}
\label{fig:stair}
\end{figure}

The notations in the following lemma are illustrated in Figure~\ref{fig:stair-proofs}~(a).

\begin{lemma}
\label{lem:staircase}
Let $y\in D$ be a Dyck path. \\
If $y\in D\setminus D'$ consider the partition $y=y_1\cdots y_\ell$ with $y_1,\ldots,y_\ell\in D'$.
Then we have the following:
\begin{enumerate}[label=(\roman*),leftmargin=8mm]
\item If $j,k\in U(y)$ with $j<k$ are from two distinct sets of $U(y_1),\ldots,U(y_\ell)$, then $j$ and~$k$ do not belong to the same staircase of~$y$.
\end{enumerate}
If $y\in D'$ consider the partition of~$y$ defined in~\eqref{eq:hill} and let
\begin{equation*}
A:=\begin{cases} U(y) & \text{if } h=1, \\
U(y)\setminus \Big(\bigcup_{i=1}^{h-2} U(u_i)\cup \bigcup_{i=0}^{h-2} U(v_i)\Big) & \text{if } h\geq 2.
\end{cases}
\end{equation*}
Then we have the following:
\begin{enumerate}[label=(\roman*),leftmargin=8mm]
\setcounter{enumi}{1}
\item If $j,k\in U(y)$ with $j<k$ are from two distinct sets of $A,U(u_1),\ldots,U(u_{h-2}),U(v_0),\ldots,U(v_{h-2})$, then $j$ and $k$ do not belong to the same staircase of~$y$.
\item The set~$A$ is a full staircase of~$y$.
\end{enumerate}
\end{lemma}

\begin{proof}
For any index~$i$ in the support of~$y$, we write~$h(i)$ for the height of the midpoint of the~$i$th step of~$y$.
Furthermore, for $i,j\in U(y)$ we define $h(i,j):=h(j)-h(i)+1$.
For example, for $y=y_1\cdots y_{10}=1011101000$ we have $h(1)=h(2)=h(3)=h(10)=1/2$, $h(4)=h(9)=3/2$, $h(5)=h(6)=h(7)=h(8)=5/2$, and $h(1,7)=3$.

We refer to a pair of indices~$(i,j)$, $i\leq j$, as a \emph{brace} if the subpath~$y_{[i,j]}$ starts and ends with an $\ustep$-step, and has no other steps at height~$h(i)$ or $h(j)$ (other than the first and last one, respectively).
By this definition, $h(i)$ and~$h(j)$ are the unique minimum and maximum of~$\{h(k)\mid k\in[i,j]\}$, respectively.
We observe the following:
(A) If $S(j)=\{j=j_1<\cdots<j_h\}$ is a staircase, then $j_i$ for all $i=1,\ldots,h$ is the rightmost position of an $\ustep$-step of~$y_{[j_1,j_h]}$ with $h(j_i,j_1)=i$.
(B) If $S(j)=\{j=j_1<\cdots<j_h\}$ is a staircase, then $b(S(j)):=(j_1,j_h)$ is a brace, and we have $|S(j)|=h=h(j_1,j_h)$.
(C) If $j,k,\ell,m\in U(y)$, $j<k<\ell<m$, are such that $h(k)\geq h(\ell)$ and~$(j,m)$ is a brace, then $k\notin S(j)$.

To prove~(i), suppose for the sake of contradiction that $j\in U(y_i)$ and $k\in U(y_{i'})$, $i<i'$, belong to the same staircase~$S$, and let $(j',k'):=b(S)$ be its brace.
Note that $j'\leq j$ and $k\leq k'$.
Then, as the first $\ustep$-step of~$y_{i'}$ has the minimum possible height of all steps of~$y$ and lies strictly between positions~$j$ and~$k$, the pair $(j',k')$ is not a brace, contradicting~(B).

To prove~(ii), suppose for the sake of contradiction that $j,k\in U(y)$ with $j<k$ are from two distinct sets of $A,U(u_1),\ldots,U(u_{h-2}),U(v_0),\ldots,U(v_{h-2})$ and belong to the same staircase~$S$, and let $(j',k'):=b(S)$ be its brace.
We also define $p:=\max A$.
We distinguish four cases.

{\bf Case~(a):} $j\in A$ and~$k\in U(u_i)$ for some $i\in\{1,\ldots,h-2\}$.
Then the position~$\ell\in A$ of the first $\ustep$-step after~$u_i$ satisfies $k<\ell<p$ and $h(k)\geq h(\ell)$, and~$(j,p)$ is a brace, so (C) gives $k\notin S(j)$, a contradiction.

{\bf Case (b):} $j\in A$ or $j\in U(u_i)$ for some $i\in\{1,\ldots,h-2\}$, and~$k\in U(v_{i'})$ for some $i'\in\{0,\ldots,h-2\}$.
As the $\ustep$-step at position~$p$ has the maximum possible height of all steps of~$y$ and $j<p<k$, the pair~$(j',k')$ is not a brace, contradicting~(B).

{\bf Case (c):} $j\in U(u_i)$ and~$k\in U(u_{i'})$ for $i,i'\in\{1,\ldots,h-2\}$ with $i<i'$.
Then there is an index~$\ell\in A$ with $j<\ell<k$ and $h(\ell)=h(j)$, contradicting~(A).

{\bf Case (d):} $j\in U(v_i)$ and~$k\in U(v_{i'})$ for $i,i'\in\{0,\ldots,h-2\}$ with $i<i'$.
Then the position~$\ell:=\max U(v_i)$ of the last $\ustep$-step of $v_i$ satisfies $j<\ell<k$ and $h(\ell)>h(k')$, and therefore $(j',k')$ is not a brace, contradicting~(B).

This completes the proof of property~(ii).

Property~(iii) follows directly from~(ii) and the definition of staircases.
\end{proof}

\begin{figure}
\makebox[0cm]{ 
\includegraphics[page=3]{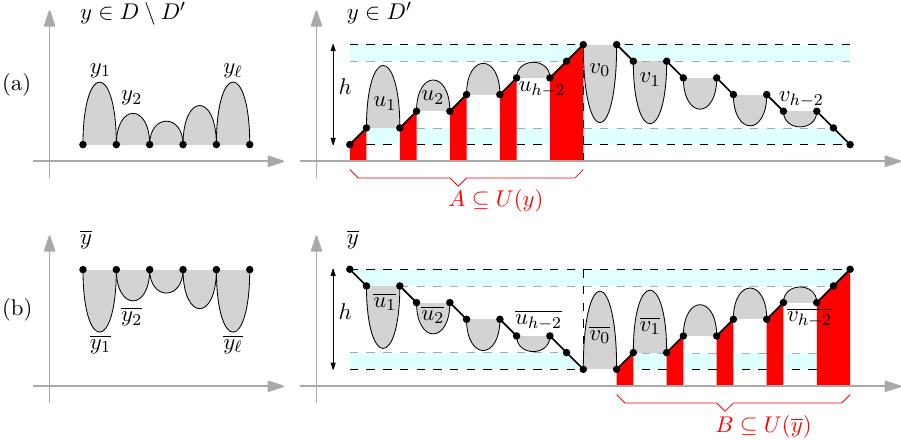}
}
\caption{Illustration of (a) Lemma~\ref{lem:staircase} and (b) Lemma~\ref{lem:staircase-comp}.}
\label{fig:stair-proofs}
\end{figure}

The following lemma is analogous to Lemma~\ref{lem:staircase}, and it considers a complemented Dyck path~$\ol{y}$; see Figure~\ref{fig:stair-proofs}~(b).

\begin{lemma}
\label{lem:staircase-comp}
Let $y\in D$ be a Dyck path. \\
If $y\in D\setminus D'$ consider the partition $y=y_1\cdots y_\ell$ with $y_1,\ldots,y_\ell\in D'$.
Then we have the following:
\begin{enumerate}[label=(\roman*),leftmargin=8mm]
\item If $j,k\in U(\ol{y})$ with $j<k$ are from two distinct sets of $U(\ol{y_1}),\ldots,U(\ol{y_\ell})$, then $j$ and~$k$ do not belong to the same staircase of~$\ol{y}$.
\end{enumerate}
If $y\in D'$ consider the partition of~$y$ defined in~\eqref{eq:hill} and let
\begin{equation*}
B:=\begin{cases} U(\ol{y}) & \text{if } h=1, \\
U(\ol{y})\setminus \Big(\bigcup_{i=1}^{h-2} U(\ol{u_i})\cup \bigcup_{i=0}^{h-2} U(\ol{v_i})\Big) & \text{if } h\geq 2.
\end{cases}
\end{equation*}
Then we have the following:
\begin{enumerate}[label=(\roman*),leftmargin=8mm]
\setcounter{enumi}{1}
\item If $j,k\in U(\ol{y})$ with $j<k$ are from two distinct sets of $B,U(\ol{u_1}),\ldots,U(\ol{u_{h-2}}),U(\ol{v_0}),\ldots,U(\ol{v_{h-2}})$, then $j$ and $k$ do not belong to the same staircase of~$\ol{y}$.
\item The set~$B$ is a full staircase of~$\ol{y}$.
\end{enumerate}
\end{lemma}

The proof of Lemma~\ref{lem:staircase-comp} is analogous to the proof of Lemma~\ref{lem:staircase}, and thus we omit it.

\begin{lemma}
\label{lem:WT}
For any~$x\in X_{n,k}$, the multisets~$W(x)$ and~$T(x)$ defined in~\eqref{eq:Wx} and~\eqref{eq:Tx}, respectively, are the same, i.e., we have $W(x)=T(x)$.
\end{lemma}

\begin{proof}
Consider a block~$y\in D$ in~$x$ and the corresponding Dyck path.
The recursive computation of~$W(y)$ via~\eqref{eq:Wx-multi} and~\eqref{eq:Wx-hill} distinguishes two cases, and the second one has two subcases.

If~$y\in D\setminus D'$, $y\neq \varepsilon$, then we have $y=y_1\cdots y_\ell$ for $y_1,\ldots,y_\ell\in D'$ and by~\eqref{eq:Wx-multi} we have $W(y)=\bigcup_{i=1}^\ell W(y_i)$.
By Lemma~\ref{lem:staircase}~(i), the set of staircases of~$y$ is obtained as the union of the sets of staircases of the substrings~$y_i$, i.e., we have $T(y)=\bigcup_{i=1}^\ell T(y_i)$.
By induction we know $W(y_i)=T(y_i)$ for all $i=1,\ldots,\ell$ and consequently $W(y)=T(y)$.

If~$y=1\,0\in D'$ then by the first case in~\eqref{eq:Wx-hill} we have $W(y)=\{1\}$.
Furthermore, we clearly have $T(y)=\{1\}$, so indeed $W(y)=T(y)$.
If~$y=1\,u\,0\in D'$ with $u\neq \varepsilon$ then by the second case in~\eqref{eq:Wx-hill} we have $W(y)=W(u)\oplus 1$.
Furthermore, one of the largest staircases of~$S(u)$ is extended by the leading~1 in~$y$ to a full staircase in~$S(y)$, implying that $T(y)=T(u)\oplus 1$.
By induction we know $W(u)=T(u)$ and consequently $W(y)=T(y)$.
\end{proof}

\begin{lemma}
\label{lem:VT}
For any~$x\in X_{n,k}$, the multisets~$V(x)$ and~$T(x)$ defined in~\eqref{eq:Vx} and~\eqref{eq:Tx}, respectively, are the same, i.e., we have $V(x)=T(x)$.
\end{lemma}

\begin{proof}
We claim that for any glider~$\gamma=(A,B)\in \Gamma(x)$, the set of positions of its~$\ustep$-steps is a full staircase~$S$ in the base hill of~$\wh{x}$ containing this glider.
Specifically, if $\gamma$ is non-inverted, then $S=A$, whereas if $\gamma$ is inverted, then $S=B$.
In both cases we obtain $v(\gamma)=|A|=|B|=|S|$ (recall~\eqref{eq:speed}).
As every $\ustep$-step of this base hill belongs to exactly one glider and to exactly one full staircase, we thus obtain a bijection between the gliders and full staircases of this base hill, and consequently a bijection between the multisets~$T(x)$ and~$V(x)$, proving the lemma.
The claim follows by induction from~\eqref{eq:Gamma-multi} and~\eqref{eq:Gamma-hill}, using Lemmas~\ref{lem:staircase} and~\ref{lem:staircase-comp}.
\end{proof}

We emphasize that in the proofs of Lemmas~\ref{lem:WT} and~\ref{lem:VT} the glider partitions for $y=1\,u\,0$ and~$u$ differ in a nontrivial way in general, i.e., $\Gamma(y)$ may differ from~$\Gamma(u)$ by more than simply adding the positions of the first and last step of~$y$ to a fastest glider from~$\Gamma(u)$; see Figure~\ref{fig:speed-conn}.
For the base hill~$w$ in this figure, there is no apparent correspondence between gliders in~$\Gamma(1\,w\,0)$ and those in~$\Gamma(w)$, because the added steps in the beginning and in the end substantially change the way that the recursive function~$\Gamma$ operates on and partitions the two Motzkin paths.

With these preparations, the proof of Lemma~\ref{lem:VW} is now straightforward.

\begin{proof}[Proof of Lemma~\ref{lem:VW}]
Combine Lemmas~\ref{lem:WT} and~\ref{lem:VT}.
\end{proof}

We write $d(x)$ \marginpar{descents $d(x)$} for the number of descents in the (cyclic) string~$x$, i.e., the number of occurrences of the substring~$10$.
Clearly, this is equal to the number of ascents, i.e., the number of occurrences of the substring~$01$.
Furthermore, this is the same as the number of inclusion maximal substrings of~1s, or the number of inclusion maximal substrings of~0s.
For example, $x=001100010001$ has three descents, one of them across the boundary, i.e., $d(x)=3$.
As this quantity has nothing to do with parenthesis matching, we did not distinguish matched or unmatched~0s in~$x$ in this example.
Nonetheless, the next lemma asserts that $d(x)$ equals the number of gliders~$\nu(x)$.

\begin{lemma}
\label{lem:trans}
For any~$x\in X_{n,k}$ we have $d(x)=\nu(x)$.
\end{lemma}

\begin{proof}
We prove inductively that $d(x)=|W(x)|$, and then the claim follows with the help of Lemma~\ref{lem:VW} and the definitions~\eqref{eq:nu} and~\eqref{eq:Vx}.
To show that $d(x)=|W(x)|$ we consider the recursive definition of~$W(x)$ in~\eqref{eq:Wx}, and we consider a substring $y\in D$, $y\neq \varepsilon$, of~$x$ that arises in this computation.

If~$y\in D\setminus D'$, $y\neq \varepsilon$, then we have $y=y_1\cdots y_\ell$ for $y_1,\ldots,y_\ell\in D'$ and by~\eqref{eq:Wx-multi} we have $|W(y)|=\sum_{i=1}^\ell |W(y_i)|$.
Furthermore, the number of occurrences of~$10$ in~$y$ is simply the sum of the number of such occurrences in each of the substrings~$y_i$, implying that $d(y)=\sum_{i=1}^\ell d(y_i)$.
By induction we know $d(y_i)=|W(y_i)|$ for all $i=1,\ldots,\ell$ and consequently $d(y)=|W(y)|$.

If~$y=1\,0\in D'$ then by the first case in~\eqref{eq:Wx-hill} we have $|W(y)|=1$.
Furthermore, we clearly have $d(y)=1$, so indeed $d(y)=|W(y)|$.
If~$y=1\,u\,0\in D'$ with $u\neq \varepsilon$ then by the second case in~\eqref{eq:Wx-hill} we have $|W(y)|=|W(u)\oplus 1|=|W(u)|$.
Furthermore, as the first and last bit of~$u$ are~1 and~0, respectively, the first~1 and last~0 in~$y$ do not create additional occurrences of~$10$, implying that $d(y)=d(u)$.
By induction we know $d(u)=|W(u)|$ and consequently $d(y)=|W(y)|$.
\end{proof}

\subsection{Free and trapped gliders}

We associate the recursive computation of $\Gamma(x)$ in~\eqref{eq:Gamma} with a \marginpar{forest $\Lambda(x)$} forest~$\Lambda(x)$ of rooted unordered trees as follows; see Figure~\ref{fig:tree}.
The vertex set of~$\Lambda(x)$ is the set of gliders~$\Gamma(x)$.
Moreover, for any glider~$\gamma\in\Gamma(x)$ we consider the recursion step~\eqref{eq:Gamma-hill} in which $\gamma=(A,B)$ is added to the set~$\Gamma(x)$, and we consider the corresponding hill~$y\in D'$ with bulges~$u_i$ and dents~$v_i$ as defined by~\eqref{eq:hill}.
The set of all descendants (direct or indirect) of~$\gamma$ in the tree of~$\Lambda(x)$ is the set~$\bigcup_{i=1}^{h-2}\Gamma(u_i)\,\cup\,\bigcup_{i=0}^{h-2} \Gamma(\ol{v_i})$, where $h$ is the height of~$y$.
Consequently, the direct descendants, i.e., the children, of~$\gamma$ are exactly the gliders computed in the \emph{next} recursion step for the $u_i$ and~$\ol{v_i}$.
Moreover, every tree in~$\Lambda(x)$ corresponds to a base hill~$y\in D'$ in~$\wh{x}$ (recall \eqref{eq:Gamma-base}), specifically the root~$\gamma$ of this tree satisfies $y=\wh{x}_{r(\gamma)}$.

\begin{figure}[h!]
\includegraphics{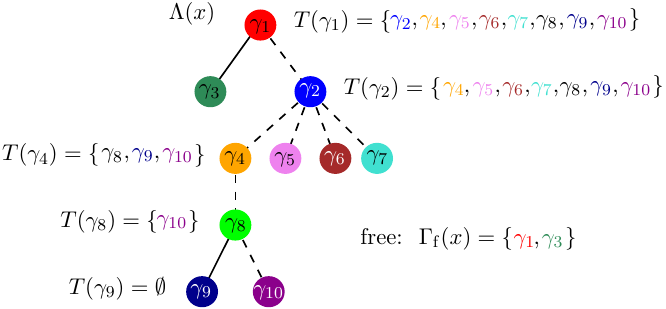}
\caption{Tree of free and trapped gliders present in the hill~$y\in D'$ of Figure~\ref{fig:Gamma}.}
\label{fig:tree}
\end{figure}

\begin{lemma}
\label{lem:dstep}
For any inverted glider~$\gamma\in\Gamma(x)$, the step at position~$s_2(\gamma)+1$ in~$\wh{x}$ is a $\dstep$-step, i.e., the corresponding bit in~$x$ is a matched~0.
\end{lemma}

As $\gamma$ is inverted, its last step at position~$s_2(\gamma)$ is an~$\ustep$-step, and the $\dstep$-step at position~$s_2(\gamma)+1$ is directly to its right.

\begin{proof}
As $\gamma$ is inverted, it cannot be the root of a tree in~$\Lambda(x)$.
Let $\gamma'$ be the parent of~$\gamma$ in the forest~$\Lambda(x)$.
If $\gamma'$ is non-inverted, then $y:=\wh{x}_{r(\gamma)}$ belongs to a dent of~$y':=\wh{x}_{r(\gamma')}$ and $y$ starts and ends on the same height as the starting point and endpoint of this dent.
From~\eqref{eq:hill} it follows that $y$ must be followed by a $\dstep$-step of~$y'$, which either belongs to the same dent or to~$\gamma'$.
If $\gamma'$ is inverted, then $\ol{y}$ belongs to a bulge of~$\ol{y'}$ and $\ol{y}$ starts and ends on the same height as the starting point and endpoint of this bulge.
From~\eqref{eq:hill} it follows that $\ol{y}$ must be followed by a $\ustep$-step of~$\ol{y'}$, which either belongs to the same bulge or to~$\gamma'$ (in the latter case this is a $\dstep$-step of~$\gamma'$, however).
In both cases~$y$ is followed by a $\dstep$-step in~$\wh{x}$.
\end{proof}

For a glider~$\gamma$ as before in the definition of~$\Lambda(x)$, we define a set $T(\gamma)\seq\Gamma(x)$ \marginpar{$T(\gamma)$} as
\begin{equation}
\label{eq:Tgamma}
T(\gamma):=\bigcup_{i=0}^{h-2} \Gamma(\ol{v_i}),
\end{equation}
and we refer to any glider~$\gamma'\in T(\gamma)$ as \emph{trapped by~$\gamma$}. \marginpar{trapped}
In the forest~$\Lambda(x)$, all descendants of~$\gamma$ in the subtrees on the gliders in~$\Gamma(\ol{v_i})$ are trapped by~$\gamma$.
In Figure~\ref{fig:tree}, children that are trapped by their parent are connected by dashed lines.
A glider that is not trapped by any of its predecessors in the forest~$\Lambda(x)$ is called~\emph{free}. \marginpar{free}
In particular, the gliders at the roots of the trees in~$\Lambda(x)$ are all free.
We write $\Gammaf(x)$ for the set of all free gliders. \marginpar{$\Gammaf(x)$}
Note that free gliders are always non-inverted, whereas trapped gliders can be inverted or non-inverted.
Specifically, a trapped glider is inverted if and only if it is trapped by an odd number of predecessors in its tree in~$\Lambda(x)$.

The following lemma describes the location and speed of gliders that are in an ancestor-descendant relationship in the forest~$\Lambda(x)$; cf.\ Figure~\ref{fig:tree}.

\begin{lemma}
\label{lem:child-props}
The set of descendants of~$\gamma \in \Gamma(x)$ in the forest~$\Lambda(x)$ is precisely the set of gliders~$\gamma'\in\Gamma(x)$ for which at least one step of~$\gamma'$ lies between two consecutive steps of~$\gamma$.
Every such glider~$\gamma'$ satisfies~$v(\gamma)>v(\gamma')$, all of its steps lie between two consecutive steps of~$\gamma$, and it is trapped by~$\gamma$ if and only if its steps lie between the last $\ustep$-step and the last $\dstep$-step of~$\gamma$.
\end{lemma}

\begin{proof}
We assume w.l.o.g.\ that $y:=\wh{x}_{r(\gamma)}$ is a hill in~$\wh{x}$.
Let $h$ denote the height of~$y$, and recall that $v(\gamma)=h$.
From Lemma~\ref{lem:bulge-dent} we obtain that the height of the $i$th bulge~$u_i$ of~$y$ is at most~$h-1-i\leq h-2$ and the depth of the $i$th dent~$v_i$ of~$y$ is at most~$h-1-i\leq h-1$.
It follows that $v(\gamma)>v(\gamma')$ for all $\gamma'\in\bigcup_{i=1}^{h-2}\Gamma(u_i)\cup\bigcup_{i=0}^{h-2}\Gamma(\ol{v_i})$, which are precisely the descendants of~$\gamma$ in the forest~$\Lambda(x)$.
By the definition~\eqref{eq:Gamma}, these descendants partition all steps between any two consecutive steps of~$\gamma$.
\end{proof}

\begin{lemma}
\label{lem:outer-steps}
For any base hill~$y$ in~$\wh{x}$ of the form~$y=1\,u\,0$, $u\in D$, the first and last step of~$y$ belong to the glider of the maximum speed in~$\Gamma(y)$.
\end{lemma}

\begin{proof}
From~\eqref{eq:Gamma} we see that the first and last step of~$y$ belong to the same glider~$\gamma\in\Gamma(x)$.
Furthermore, $\gamma$ is the root of a tree in~$\Lambda(x)$ and every other glider in~$\Gamma(y)$ is a descendant of~$\gamma$ in this tree.
The statement hence follows from Lemma~\ref{lem:child-props}.
\end{proof}

We say that a glider~$\gamma\in\Gamma(x)$ is \emph{clean}, \marginpar{clean} if $\wh{x}_{r(\gamma)}=1^{v(\gamma)}0^{v(\gamma)}$ or $\wh{x}_{r(\gamma)}=0^{v(\gamma)}1^{v(\gamma)}$; see Figure~\ref{fig:coupled}.
In other words, in the infinite Motzkin path~$\wh{x}$, the steps belonging to~$\gamma$ are all consecutive, i.e., all steps of~$\wh{x}_{r(\gamma)}$ belong to~$\gamma$.
These are $2v(\gamma)$ steps in total, $v(\gamma)$ many $\ustep$-steps followed by $v(\gamma)$ many $\dstep$-steps if~$\gamma$ is non-inverted, and vice versa if $\gamma$ is inverted.

\begin{lemma}
\label{lem:clean}
If $\gamma\in\Gamma(x)$ has the minimum speed~$v(\gamma)=\min V(x)$, then it is clean.
\end{lemma}

\begin{proof}
As $\gamma$ has the minimum speed, Lemma~\ref{lem:child-props} yields that all of its steps must be consecutive.
\end{proof}

We say that two gliders~$\gamma,\gamma'\in\Gamma(x)$ of the same speed $v(\gamma)=v(\gamma')$ are \emph{coupled}, \marginpar{coupled} if $s(\gamma)<s(\gamma')$ and all steps of~$\wh{x}$ between the last step of~$\gamma$ and the first step of~$\gamma'$ belong to gliders of strictly smaller speed; see Figure~\ref{fig:coupled}.
In particular, no $\fstep$-steps are allowed between~$\gamma$ and~$\gamma'$.
Note that if $v(\gamma)=v(\gamma')=\min V(x)$, then the first step of~$\gamma'$ must directly follow the last step of~$\gamma$.
For the following definition we introduce the abbreviation~$[a,b]:=[s_2(\gamma)+1,s_0(\gamma')-1]$ for the interval between the last step of~$\gamma$ and the first step of~$\gamma'$.
Note that if a step of some glider~$\gammatilde\in\Gamma(x)$ lies in the interval~$[a,b]$, then $v(\gammatilde)<v(\gamma)=v(\gamma')$ by definition, and therefore by Lemma~\ref{lem:child-props} all steps of~$\gammatilde$ lie in~$[a,b]$, i.e., we have $r(\gammatilde)\seq[a,b]$.
Consequently, $\wh{x}_{[a,b]}$ is partitioned into such gliders.
We write $B(\gamma,\gamma')$ \marginpar{$B(\gamma,\gamma')$} for the set of those gliders between~$\gamma$ and~$\gamma'$ with inclusion maximal range, formally,
\begin{multline}
\label{eq:Bgamma}
B(\gamma,\gamma'):=\big\{\gammatilde\in\Gamma(x)\mid r(\gammatilde)\seq [a,b]\text{ and there is no } \gammacheck\in \Gamma(x) \\
\text{ with } r(\gammacheck)\seq [a,b] \text{ and } r(\gammatilde)\subsetneq r(\gammacheck) \Big\}.
\end{multline}
In the forest~$\Lambda(x)$, the gliders~$\gammatilde\in B(\gamma,\gamma')$ are the gliders with~$r(\gammatilde)\seq[a,b]$ that are closest to a root.

\begin{lemma}
\label{lem:coupled-trapped}
Let $\gamma,\gamma',\gamma''\in \Gamma(x)$ such that $\gamma'$ and~$\gamma''$ are coupled.
Then either $\gamma'$, $\gamma''$ and all gliders in the set~$B(\gamma',\gamma'')$ defined in~\eqref{eq:Bgamma} are trapped by~$\gamma$ or none of them.
\end{lemma}

\begin{proof}
As mentioned before, every step in the interval~$[a,b]:=[s_2(\gamma')+1,s_0(\gamma'')-1]$ belongs to a glider that has all of its steps in this interval.
Consequently, by the definition of~$B(\gamma',\gamma'')$, the ranges~$r(\gammatilde)$ of the gliders~$\gammatilde\in\{\gamma',\gamma''\}\cup B(\gamma',\gamma'')$ form consecutive intervals that partition~$[a,b]$, i.e., the last step of one glider is followed directly by the first step of the next glider.
Therefore, if one of the gliders~$\{\gamma',\gamma''\}\cup B(\gamma',\gamma'')$ is trapped by some glider~$\gamma\in\Gamma(x)$, then two consecutive steps of~$\gamma$ must enclose the entire interval~$[a-1,b+1]$, so all gliders in~$\{\gamma',\gamma''\}\cup B(\gamma',\gamma'')$ must be trapped by~$\gamma$ (recall Lemma~\ref{lem:child-props}).
\end{proof}

\begin{figure}
\includegraphics{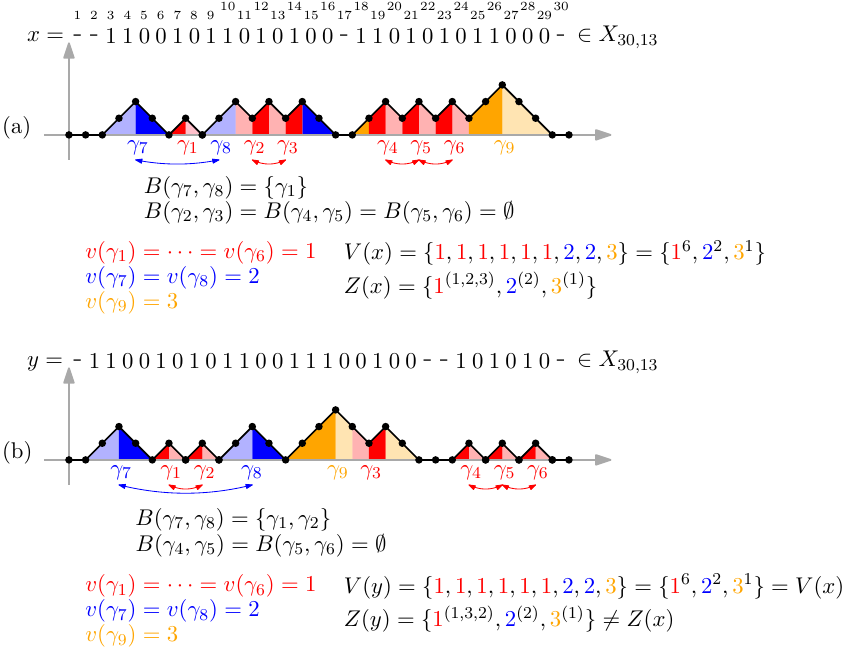}
\caption{Illustration of clean and coupled gliders, and of trains.
In (a), all gliders except~$\gamma_8$ and~$\gamma_9$ are clean, and in (b), all gliders except $\gamma_9$ are clean.
In both figures, pairs of coupled gliders are connected by double arrows.}
\label{fig:coupled}
\end{figure}

\section{Analysis via gliders: dynamic properties}

In the previous section, we have partitioned certain steps of~$\wh{x}$ into gliders, and defined various properties for them, such as position, speed etc.
All of these definitions were `static', i.e., for one fixed vertex~$x\in X_{n,k}$.
In this section we investigate the behavior of gliders as we move from~$x$ to~$f(x)$ along the cycle~$C(x)$, i.e., we aim to understand the `dynamic' behavior of gliders over one time step.

\subsection{Capturing relation}
\label{sec:cap-relation}

As long as a glider does not interact with other gliders, its motion is uniform with the corresponding speed; recall Figure~\ref{fig:gliders}~(a)+(b).
However, two gliders of different speeds interact with each other.
Specifically, they participate in an overtaking, which boosts the faster glider that is overtaking, and delays the slower glider that is being overtaken; see Figure~\ref{fig:gliders}~(c)+(d).
During an overtaking, the glider being overtaken is trapped by the overtaking glider and does not change its position.
However, its bits are complemented, i.e., if it was non-inverted, it becomes inverted, and vice versa.
In the following, we determine which free gliders participate in overtakings in the next time step (trapped gliders do not move, so they are ignored).

\begin{figure}
\makebox[0cm]{ 
\includegraphics[page=1]{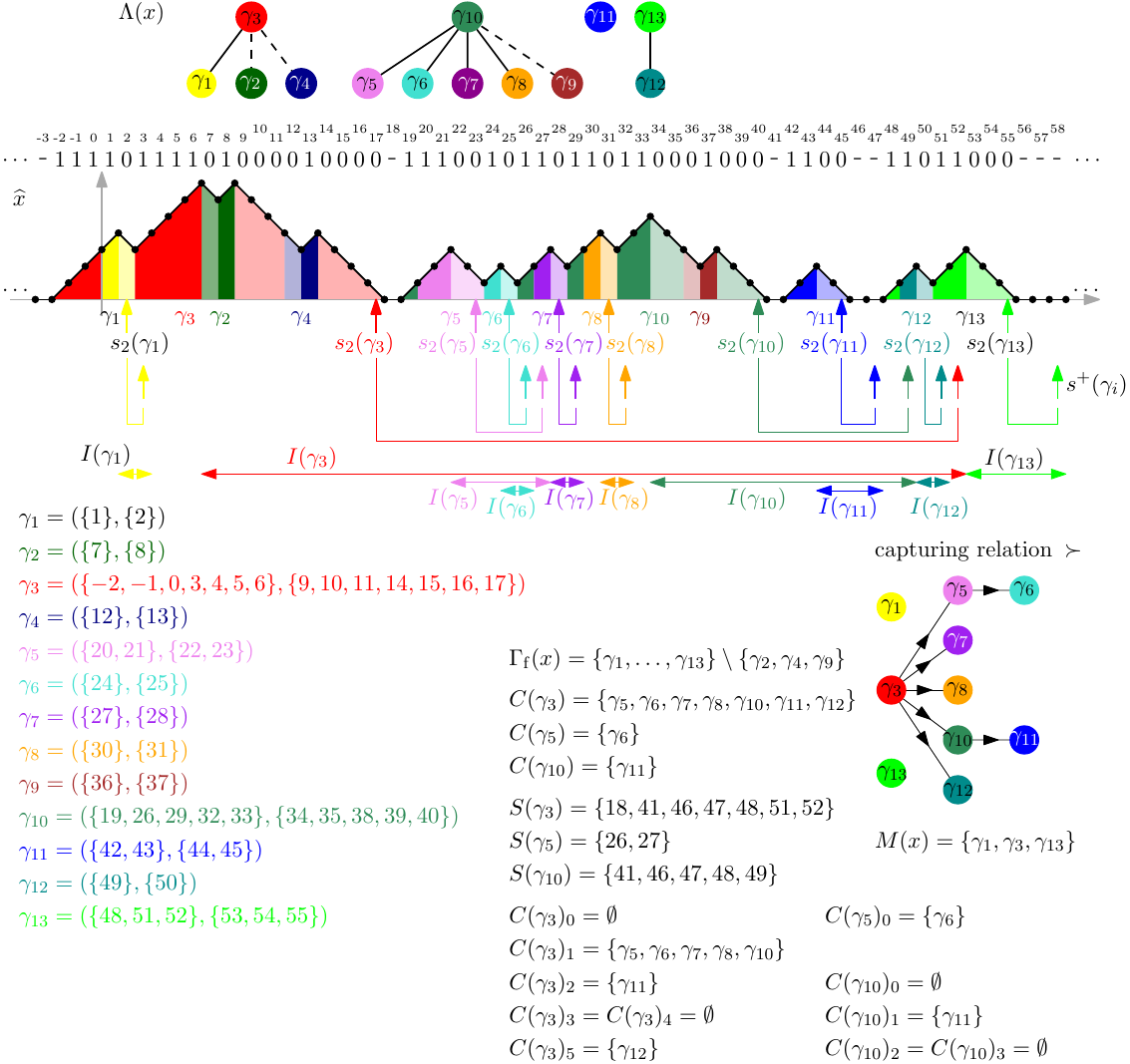}
}
\caption{Illustration of the capturing relation and related definitions.
For all pairs of gliders~$\gamma\succ\gamma'$, the diagram on the bottom right shows an arrow directed from~$\gamma$ to~$\gamma'$.}
\label{fig:capture}
\end{figure}

The following definitions are illustrated in Figure~\ref{fig:capture}.
We first define a binary relation~$\succ$ on the set of free gliders~$\Gammaf(x)$ as follows.
For any $i \in \mathbb{Z}$ we define
\marginpar{$w(i),w(I)$}
\begin{equation*}
w(i) = \begin{cases}
  +1 & \text{if }\wh{x}_i\in\{\fstep,\ustep\}, \\
  -1 & \text{if }\wh{x}_i=\dstep,
\end{cases}
\end{equation*}
and for any interval~$I$ we define
\begin{equation*}
w(I):=\sum_{i\in I} w(i).
\end{equation*}

For a free glider~$\gamma\in\Gammaf(x)$, we consider the subpath of~$\wh{x}$ strictly to the right of position~$s_2(\gamma)$, i.e., starting from $a:=s_2(\gamma)+1$.
We consider the minimum index~$b$ such that $w([a,b])=v(\gamma)$, which exists as $w([a,a-1])=0$, $w([a,p])=w([a,p-1])\pm 1$ for all $p\geq a$, and $w([a,a+n])=n-2k>0$.
Formally, we define
\marginpar{$s^+(\gamma)$}
\begin{equation}
\label{eq:s+gamma}
s^+(\gamma) := \min \{b \mid w([s_2(\gamma) + 1, b]) = v(\gamma) \}.
\end{equation}
From this definition we see that $w([s_2(\gamma) + 1, s^+(\gamma)])=v(\gamma)$.
Intuitively, $s^+(\gamma)$ is the rightmost position that the glider~$\gamma$ would occupy after its movement with speed~$v(\gamma)$ in one time step, assuming that it moves independently of all other free gliders (i.e., if there are no faster gliders that overtake~$\gamma$ in this time step).
For two free gliders $\gamma,\gamma'\in\Gammaf(x)$ such that $s_1(\gamma)<s_1(\gamma')$, we say that \emph{$\gamma$ captures $\gamma'$}, \marginpar{capture $\succ$} which we denote by $\gamma\succ\gamma'$, if $s^+(\gamma)>s^+(\gamma')$; see Figure~\ref{fig:capture}.
It follows directly from this definition that the capturing relation~$\succ$ is transitive, i.e., $\gamma\succ\gamma'$ and $\gamma'\succ\gamma''$ implies that $\gamma\succ\gamma''$.

The next lemma describes a transitivity property for coupled gliders.

\begin{lemma}
\label{lem:coupled-capture}
Let $\gamma,\gamma',\gamma''\in \Gammaf(x)$ such that $\gamma'$ and~$\gamma''$ are coupled.
If $\gamma\succ \gamma'$, then we have $\gamma\succ\gamma''$, and $\gamma\succ\gammatilde$ for all $\gammatilde\in B(\gamma',\gamma'')$ for the set~$B(\gamma',\gamma'')$ defined in~\eqref{eq:Bgamma}.
\end{lemma}

\begin{proof}
This proof is illustrated in Figure~\ref{fig:coupled-capture}.
We define $a:=s_2(\gamma')+1$, $b:=s_0(\gamma'')-1$, $c:=s_1(\gamma'')$, and~$d:=s_2(\gamma'')$.
As~$\gamma'$ and~$\gamma''$ are free, they are non-inverted.
Furthermore, as~$\gamma'$ and~$\gamma''$ are coupled, all steps of $\wh{x}_{[a,b]}$ belong to gliders of strictly smaller speed, and all steps of those gliders are in the interval~$[a,b]$ by Lemma~\ref{lem:child-props}.
By Lemma~\ref{lem:coupled-trapped}, all gliders in~$B(\gamma',\gamma'')$ are also free and hence non-inverted, and consequently~$\wh{x}_{[a,b]}$ is a hill.
Its height is less than~$v(\gamma')=v(\gamma'')$, and as no $\fstep$-steps occur in the interval~$[a,b]$, we have $w([a,p])<v(\gamma')$ for all $p\in[a,c-1]$ and $w([a,c])=v(\gamma')$, and therefore $s^+(\gamma')=c$.
From the assumption~$\gamma\succ\gamma'$ we know that $s^+(\gamma)>s^+(\gamma')=c$ and therefore $v(\gamma)>w([s_2(\gamma)+1,c])$.
As $w([c+1,d])=-v(\gamma'')$, we obtain $v(\gamma)-v(\gamma'')>w([s_2(\gamma)+1,d])$.
Using the definition of~$s^+(\gamma)$ and~$s^+(\gamma'')$ from~\eqref{eq:s+gamma}, this inequality implies~$s^+(\gamma)>s^+(\gamma'')$, and therefore $\gamma\succ\gamma''$, as claimed.
To complete the proof, note that for all $\gammatilde\in B(\gamma',\gamma'')$ we have $v(\gammatilde)<v(\gamma')=v(\gamma'')$ and therefore $s^+(\gammatilde)<c$, which also yields~$\gamma\succ\gammatilde$.
\end{proof}

\begin{figure}[h!]
\includegraphics[page=3]{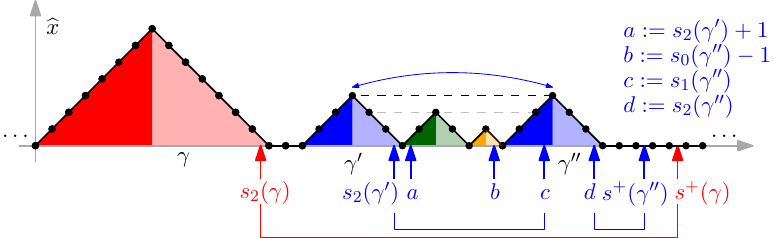}
\caption{Illustration of the proof of Lemma~\ref{lem:coupled-capture}.}
\label{fig:coupled-capture}
\end{figure}

The following lemma describes another transitivity property, namely that when a glider~$\gamma$ captures another glider~$\gamma'$, then $\gamma$ also captures the free descendants of~$\gamma'$ in~$\Lambda(x)$.

\begin{lemma}
\label{lem:child-capture}
Let $\gamma,\gamma',\gamma''\in\Gammaf(x)$.
If $\gamma\succ \gamma'$ and $\gamma''$ is a child~of~$\gamma'$ in~$\Lambda(x)$, then we have $\gamma\succ\gamma''$.
\end{lemma}

\begin{proof}
This proof is illustrated in Figure~\ref{fig:child-capture}.
Consider two gliders $\gamma',\gamma''\in\Gammaf(x)$ such that $\gamma''$ is a child of~$\gamma'$ in~$\Lambda(x)$.
We have that $\wh{x}_{r(\gamma'')}$ is a bulge of the hill~$\wh{x}_{r(\gamma')}$, and hence by Lemma~\ref{lem:bulge-dent} the highest point of the hill~$\wh{x}_{r(\gamma'')}$ is strictly lower than the highest point of the hill~$\wh{x}_{r(\gamma')}$.
From~\eqref{eq:s+gamma} we see that the height of the hill~$\wh{x}_{r(\gamma'')}$ equals $v(\gamma'')=w([s_2(\gamma'')+1, s^+(\gamma'')])$.
Combining these observations shows that $s^+(\gamma'')<s_1(\gamma')$, in particular $\gamma''\not\succ\gamma'$.
More generally, if $\gamma',\gamma''\in\Gammaf(x)$ are such that $\gamma''$ is a descendant of~$\gamma'$ in~$\Lambda(x)$, then for the sequence of free gliders~$\gamma_1=\gamma',\gamma_2,\gamma_3,\ldots,\gamma_\ell=\gamma''$ on the path from~$\gamma'$ to~$\gamma''$ in their common tree of~$\Lambda(x)$ we obtain $s^+(\gamma_i)<s_1(\gamma_{i-1})<s_2(\gamma_{i-1})<s^+(\gamma_{i-1})$ for $i=\ell,\ell-1,\ldots,2$, and consequently $\gamma''\not\succ\gamma'$.

Now let $\gamma,\gamma',\gamma''\in\Gammaf(x)$ be as in the lemma.
We argued before that $s^+(\gamma'')<s_1(\gamma')<s^+(\gamma')$.
As $\gamma\succ\gamma'$, we have $s^+(\gamma)>s^+(\gamma')$ and therefore $s^+(\gamma)>s^+(\gamma'')$.
The arguments before show that $\gamma$ is not a descendant of~$\gamma'$ in~$\Lambda(x)$, so $s_1(\gamma)<s_2(\gamma)<s_0(\gamma')<s_1(\gamma'')$.
Combining these observations shows that $\gamma\succ\gamma''$, as claimed.
\end{proof}

\begin{figure}[h!]
\includegraphics[page=4]{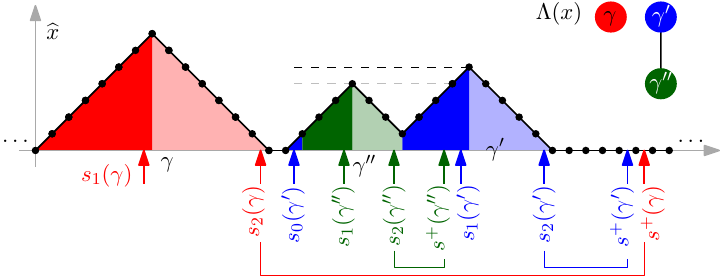}
\caption{Illustration of the proof of Lemma~\ref{lem:child-capture}.}
\label{fig:child-capture}
\end{figure}

We introduce a few more definitions, illustrated in Figure~\ref{fig:capture}.

For any free glider $\gamma \in \Gammaf(x)$, we define a set $C(\gamma)\seq\Gammaf(x)$ as \marginpar{$C(\gamma)$}
\begin{equation}
\label{eq:Cgamma}
C(\gamma) := \big\{\gamma'\in\Gammaf(x) \mid \gamma \succ \gamma' \},
\end{equation}
and we refer to any glider~$\gamma'\in C(\gamma)$ as \emph{captured by~$\gamma$}. \marginpar{captured}
By Lemma~\ref{lem:child-capture}, $C(\gamma)$ is a downset of vertices of the subforest of~$\Lambda(x)$ induced by free gliders.

For any glider~$\gamma\in\Gammaf(x)$, we let~$S(\gamma)$ be the set of positions in the interval~[$s_2(\gamma)+1, s^+(\gamma)$] not belonging to the range of any glider in~$C(\gamma)$.
\marginpar{$S(\gamma)$}
Formally, we define
\begin{equation}
\label{eq:Sgamma}
S(\gamma):=[s_2(\gamma)+1, s^+(\gamma)]\setminus\bigcup_{\gamma'\in C(\gamma)} r(\gamma').
\end{equation}
Furthermore, we define a partition of the set~$C(\gamma)$ of captured gliders as follows:
\marginpar{$C(\gamma)_i$}
We let $C(\gamma)_0$ be the subset of gliders~$\gamma'$ from~$C(\gamma)$ for which the maximum of~$r(\gamma')$ is smaller than the minimum of~$S(\gamma)$.
Furthermore, we let $C(\gamma)_i$, $i\geq 1$, be the subset of gliders~$\gamma'$ from~$C(\gamma)$ for which $r(\gamma')$ is between the $i$th and $(i+1)$th smallest element of~$S(\gamma)$.
Lastly, for any glider $\gamma \in \Gammaf(x)$ we define
\marginpar{$I(\gamma)$}
\begin{equation}
\label{eq:Igamma}
I(\gamma) := [s_1(\gamma) + 1, s^+(\gamma)].
\end{equation}

The next lemma describes a number of crucial properties of the capturing relation and the corresponding concepts defined before.

\begin{lemma}
\label{lem:capture}
Any free glider $\gamma\in\Gammaf(x)$ has the following properties:
\begin{enumerate}[label=(\roman*),leftmargin=8mm]
\item
For any glider~$\gamma'\in C(\gamma)$ for the set~$C(\gamma)$ defined in~\eqref{eq:Cgamma} we have $r(\gamma')\seq [s_2(\gamma)+1,s^+(\gamma)-1]$.

\item
The set~$S(\gamma)$ defined in~\eqref{eq:Sgamma} contains the positions of all $\fstep$-steps of~$\wh{x}$ on the interval~$[s_2(\gamma)+1,s^+(\gamma)]$, followed by positions of 0 or more $\ustep$-steps.
We have $s^+(\gamma)\in S(\gamma)$ and $|S(\gamma)|=v(\gamma)$, i.e., the cardinality of~$S(\gamma)$ equals the speed of~$\gamma$.

\item
For any glider~$\gamma'$ in the set~$C(\gamma)_i$ defined after~\eqref{eq:Sgamma}, we have $v(\gamma')<v(\gamma)-i$.
Consequently, $\wh{x}_{R_i}$ for the interval $R_i:=\bigcup_{\gamma'\in C(\gamma)_i} r(\gamma')$ is a hill of height at most~$v(\gamma)-i-1$.

\item
For any glider $\gamma'\in\Gammaf(x)$ such that $s_1(\gamma)<s_1(\gamma')$, we have $I(\gamma)\supseteq I(\gamma')$ if $\gamma\succ\gamma'$, and $I(\gamma)\cap I(\gamma')=\emptyset$ if $\gamma\not\succ\gamma'$, i.e., any two of the intervals defined in~\eqref{eq:Igamma} are either disjoint or nested.

\item
If for $\gamma,\gamma'\in\Gammaf(x)$ the interval~$I(\gamma')$ is closest to the right of~$I(\gamma)$ and $v(\gamma)<v(\gamma')$, then the intervals are separated by at least one element, i.e., $s^+(\gamma)<s_1(\gamma')$.
\end{enumerate}
\end{lemma}

Item~(iii) shows in particular that $v(\gamma')<v(\gamma)$, i.e., a glider can only capture slower gliders, and a glider can only be captured by faster gliders.

\begin{proof}
Let $c:=s_2(\gamma)+1$ and $d:=s^+(\gamma)$, and consider the subpath of~$\wh{x}$ on the interval~$J:=[c,d]$.

To prove~(i), let $\gamma'\in C(\gamma)$, i.e., we have $\gamma\succ\gamma'$.
The definition of the capturing relation gives $s_1(\gamma)<s_1(\gamma')$.
As $\gamma'$ is free, none of its steps is in the interval~$[s_1(\gamma)+1,s_2(\gamma)]$ (recall Lemma~\ref{lem:child-props}), and therefore $s_2(\gamma)+1\leq s_0(\gamma')$.
As $s^+(\gamma)>s^+(\gamma')>s_2(\gamma')$ we conclude that $r(\gamma')=[s_0(\gamma'),s_2(\gamma')]\seq [c,d-1]$, as claimed.

We now prepare the proofs for (ii)-(v).
From~\eqref{eq:s+gamma} we obtain the following
\textit{Observation~H:} For any hill~$y=\wh{x}_{[a,b]}$ of any height~$h$ with $a\geq c$ and with the $\ustep$-step at position $p\in[a,b]$ leading to the leftmost highest point of~$y$, if $d>p$ then we have $d>b$ and
\begin{equation}
\label{eq:ellcb}
w([c,b])<v(\gamma)-h.
\end{equation}

\begin{figure}
\includegraphics[page=5]{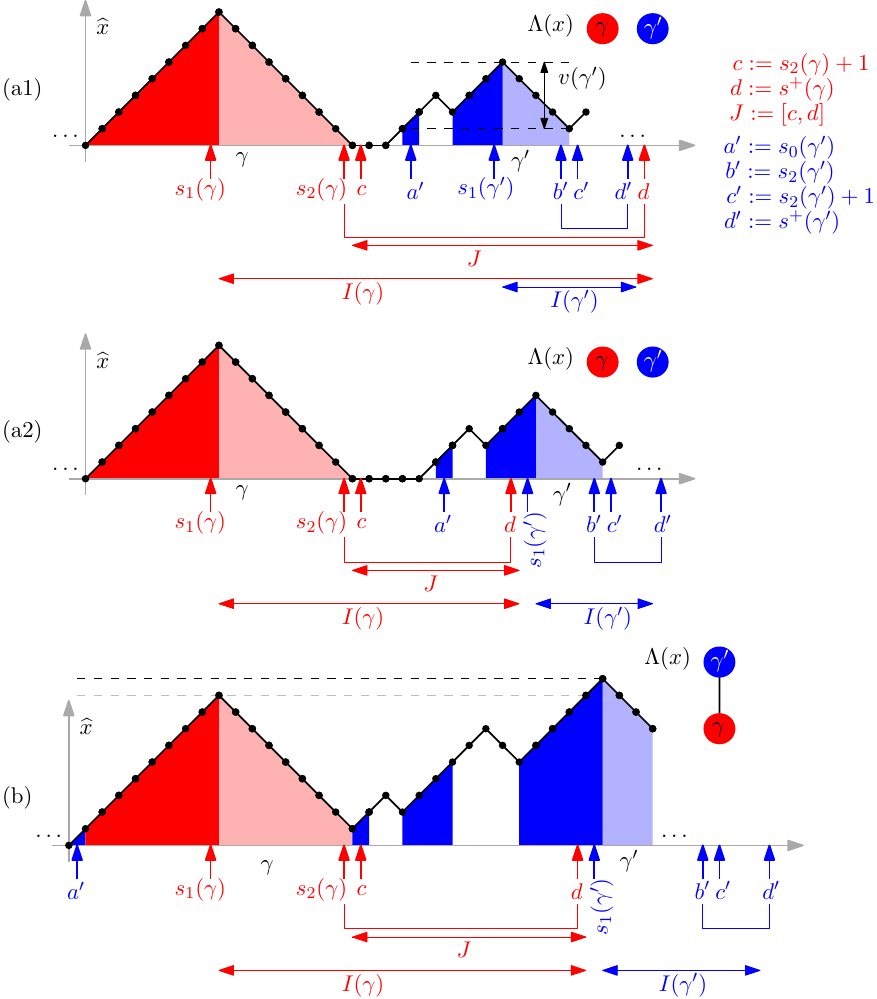}
\caption{Illustration of the three cases in the proof of Lemma~\ref{lem:capture}.}
\label{fig:capture-prop}
\end{figure}

The following arguments are illustrated in Figure~\ref{fig:capture-prop}.
We consider any glider~$\gamma'\in\Gammaf(x)$ with $r(\gamma')\cap J\neq \emptyset$, and define $[a',b']:=r(\gamma')=[s_0(\gamma'),s_2(\gamma')]$, $c':=s_2(\gamma')+1=b'+1$, and $d':=s^+(\gamma')$.
From~\eqref{eq:s+gamma} we know that
\begin{equation}
\label{eq:ellc'd'}
w([c',d'])=v(\gamma').
\end{equation}

Note that $\gamma'$ is not a descendant of~$\gamma$ in~$\Lambda(x)$, otherwise we would have $r(\gamma')\cap J=\emptyset$.
It remains to consider two possible cases:

{\bf Case (a):} $\gamma$ is not a descendant of~$\gamma'$ in~$\Lambda(x)$.
In this case, we have $r(\gamma')\cap r(\gamma)=\emptyset$ and therefore $a'\in J$.
We apply Observation~H to the hill~$\wh{x}_{r(\gamma')}=\wh{x}_{[a',b']}$ of height~$h=v(\gamma')$.

{\bf Case (a1):} If $b'\in J$, then \eqref{eq:ellcb} gives
\begin{equation}
\label{eq:wcb'}
w([c,b'])<v(\gamma)-v(\gamma'),
\end{equation}
which combined with \eqref{eq:ellc'd'} and $c'=b'+1$ yields
\begin{equation}
\label{eq:vgamma}
w([c,d'])=w([c,b'])+w([c',d'])<v(\gamma).
\end{equation}
Combining~\eqref{eq:s+gamma} and~\eqref{eq:vgamma} shows that $d>d'$, i.e., we have $s^+(\gamma)>s^+(\gamma')$, which implies $\gamma\succ\gamma'$.
We also see from~\eqref{eq:Igamma} that $I(\gamma)\supseteq I(\gamma')$.
By the definition~\eqref{eq:Cgamma}, the glider~$\gamma'$ is contained in the set~$C(\gamma)$, and hence none of the positions in~$r(\gamma')=[a',b']$ are contained in the set~$S(\gamma)$ by~\eqref{eq:Sgamma}.

{\bf Case (a2):} If $b'\notin J$, then by Observation~H we have $d\leq s_1(\gamma')$.
It follows that $\gamma\not\succ\gamma'$, and therefore $\gamma'\notin C(\gamma)$.
We also see that $I(\gamma)\cap I(\gamma')=\emptyset$.
Furthermore, as all $\dstep$-steps of~$\gamma'$ are at positions strictly to the right of~$d=s^+(\gamma)$, all steps of~$\gamma'$ at positions contributed to~$S(\gamma)$ are $\ustep$-steps.
As there are no $\fstep$-steps in $\wh{x}_{r(\gamma')}$ and the interval~$r(\gamma')$ ends after the interval~$J$, the positions of all $\ustep$-steps of~$\gamma'$ contributed to~$S(\gamma)$ cannot be interleaved or followed by positions of any~$\fstep$-steps of~$\wh{x}$.

{\bf Case (b):} $\gamma$ is a descendant of~$\gamma'$ in~$\Lambda(x)$.
By Lemma~\ref{lem:bulge-dent} the highest point of the hill~$\wh{x}_{r(\gamma)}$ is strictly lower than the highest point of the hill~$\wh{x}_{r(\gamma')}$.
From~\eqref{eq:s+gamma} we conclude that $d=s^+(\gamma)<s_1(\gamma')$, in particular $\gamma\not\succ\gamma'$ and therefore $\gamma'\notin C(\gamma)$.
We also see that $I(\gamma)\cap I(\gamma')=\emptyset$.
Furthermore, all steps of~$\gamma'$ at positions contributed to~$S(\gamma)$ are $\ustep$-steps, and as there are no $\fstep$-steps in $\wh{x}_{r(\gamma')}$ and the interval~$r(\gamma')$ contains the interval~$J$, the set~$S(\gamma)$ contains no positions of any $\fstep$-steps of~$\wh{x}$ at all.

We are now in position to prove (ii)-(v).

To prove the first part of~(ii), we first argue that no positions of any steps of a trapped glider~$\gamma''\in\Gamma(x)\setminus \Gammaf(x)$ contribute to the set~$S(\gamma)$.
Indeed, if~$i\in S(\gamma)$ and step~$i$ of~$\wh{x}$ belongs to a trapped glider~$\gamma''\in\Gamma(x)\setminus \Gammaf(x)$, then we let $\gamma'\in\Gammaf(x)$ be a free ancestor of~$\gamma''$ in the forest~$\Lambda(x)$, i.e., we have $\gamma''\in T(\gamma')$.
Note that~$r(\gamma')\supseteq r(\gamma'')$, and consequently $\gamma'$ satisfies $r(\gamma')\cap J\neq \emptyset$, showing that~$\gamma'$ is a glider as considered in the case distinction before.
However, in each of the three cases above, no steps of any gliders trapped by~$\gamma'$ contribute to the set~$S(\gamma)$.
In case~(a1) this is because none of the positions in~$r(\gamma')\supseteq r(\gamma'')$ contributes to~$S(\gamma)$.
In cases~(a2) and~(b) this is because all dents of~$\wh{x}_{r(\gamma')}$ lie strictly to the right of position~$s_1(\gamma')\geq d$.
We conclude that the set~$S(\gamma)$ can only contain positions of steps that belong to gliders~$\gamma'\in\Gammaf(x)$ with~$r(\gamma')\cap J\neq \emptyset$ as considered before, and of~$\fstep$-steps of~$\wh{x}$.
In case~(a1) we have $\gamma'\in C(\gamma)$ and hence none of the positions in~$r(\gamma')$ contributes to~$S(\gamma)$.
In cases~(a2) and~(b) we have $\gamma'\notin C(\gamma)$ and hence $\gamma'$ contributes to~$S(\gamma)$ the positions of all of its $\ustep$-steps in~$J$, with the additional property that no $\fstep$-step of~$\wh{x}$ at a position in~$J\supseteq S(\gamma)$ can come after any of these $\ustep$-steps.
Using the fact that $\wh{x}_{r(\gamma')}$ contains no $\fstep$-steps for any $\gamma'\in C(\gamma)$, we see from~\eqref{eq:Sgamma} that $S(\gamma)$ contains the positions of all $\fstep$-steps on the interval~$J$. 
This proves the first part of~(ii).
To prove the second part, first note that $s^+(\gamma)\in S(\gamma)$ is an immediate consequence of~(i).
It remains to argue that~$|S(\gamma)|=v(\gamma)$.
For this consider the subpath~$\wh{x}_J$.
By~\eqref{eq:s+gamma}, we have $w(J)=w([s_2(\gamma)+1,s^+(\gamma)])=v(\gamma)$.
Any glider~$\gamma'\in C(\gamma)$ is free and hence non-inverted, and so the intervals~$r(\gamma')$ with $\gamma'\in C(\gamma)$ that are removed from~$J$ to obtain~$S(\gamma)$ as defined in~\eqref{eq:Sgamma} correspond to hills~$\wh{x}_{r(\gamma')}$ that satisfy
\begin{equation}
\label{eq:wgamma'}
w(r(\gamma'))=0,
\end{equation}
which implies $w(J)=\sum_{i\in S(\gamma)} w(i)=v(\gamma)$.
By the first part of~(ii), for every~$i\in S(\gamma)$ we have $\wh{x}_i\in\{\fstep,\ustep\}$ and therefore $w(i)=1$, implying that $\sum_{i\in S(\gamma)} w(i)=|S(\gamma)|=v(\gamma)$, as claimed.
This completes the proof of~(ii).

To prove~(iii), consider a glider $\gamma'\in C(\gamma)_i$ and define~$[a',b']:=r(\gamma')=[s_0(\gamma'),s_2(\gamma')]$.
From~\eqref{eq:wcb'} and~\eqref{eq:wgamma'} we obtain~$w([c,a'-1])=w([s_2(\gamma)+1,s_0(\gamma')-1])<v(\gamma)-v(\gamma')$.
We write $S(\gamma)_i$ for the set of the $i$ smallest numbers from~$S(\gamma)$.
Using~\eqref{eq:wgamma'} and~(ii) similarly to before we obtain $w([c,a'-1])=\sum_{j\in S(\gamma)_i}w(j)=|S(\gamma)_i|=i$.
Combining these observations yields~$i<v(\gamma)-v(\gamma')$, as desired.
This argument also works in the special case~$i=0$, when the interval~$[c,a'-1]$ is empty and therefore $w([c,a'-1])=0$.
We mentioned before that $\wh{x}_{r(\gamma')}$ is a hill for all $\gamma'\in C(\gamma)_i$ and from~(i) we have $r(\gamma')\seq J$.
As $v(\gamma')<v(\gamma)-i$, the height of the hill is at most~$v(\gamma)-i-1$.
By the definition of~$C(\gamma)_0$, every step of~$\wh{x}$ at positions in~$J$ smaller than the minimum of~$S(\gamma)$ belongs to a glider~$\gamma'\in C(\gamma)_0$.
Similarly, for $i\geq 1$ every step of~$\wh{x}$ between the $i$th and $(i+1)$th smallest element of~$S(\gamma)$ belongs to a glider~$\gamma'\in C(\gamma)_i$.
Consequently, for every $i\geq 0$ the set $R_i:=\bigcup_{\gamma'\in C(\gamma)_i}r(\gamma')$ is an interval inside~$J$.
As $\wh{x}_{r(\gamma')}$ for all $\gamma'\in C(\gamma)_i$ is a hill of height at most~$v(\gamma)-i-1$, we obtain that $\wh{x}_{R_i}$ is a hill of height at most~$v(\gamma)-i-1$.

To prove~(iv), note that any glider~$\gamma'\in\Gammaf(x)$ as in case~(a1) satisfies $\gamma\succ\gamma'$ and $I(\gamma)\supseteq I(\gamma')$, whereas any glider~$\gamma'\in\Gammaf(x)$ as in case~(a2) or case~(b) satisfies $\gamma\not\succ\gamma'$ and $I(\gamma)\cap I(\gamma')=\emptyset$.

To prove~(v), let~$\gamma,\gamma'\in\Gammaf(x)$ such that the interval~$I(\gamma')$ is closest to the right of~$I(\gamma)$ and $v(\gamma)<v(\gamma')$, and suppose for the sake of contradiction that $I(\gamma)$ and~$I(\gamma')$ were not separated by at least one element but $s^+(\gamma)=s_1(\gamma')$.
This means that the rightmost $\ustep$-step of~$\gamma'$ is at position~$s^+(\gamma)$.
As the intervals~$I(\gamma)$ and~$I(\gamma')$ are disjoint we have $\gamma\not\succ \gamma'$ by~(iv) and therefore $\gamma'\notin C(\gamma)$, i.e., every $\ustep$-step of~$\gamma'$ in the interval~$[s_2(\gamma)+1,s^+(\gamma)]$ contributes to the cardinality of~$S(\gamma)$.
As $|S(\gamma)|=v(\gamma)$ by~(ii) and $v(\gamma)<v(\gamma')$, at least one of the $\ustep$-steps of~$\gamma'$ must be to the left of~$r(\gamma)$, i.e., $\gamma$ is a descendant of~$\gamma'$ in~$\Lambda(x)$.
Specifically, let $i\in\{1,2,\ldots,v(\gamma')-2\}$ be such that $\wh{x}_{r(\gamma)}$ belongs to the $i$th bulge of the hill~$y:=\wh{x}_{r(\gamma')}$.
As this bulge has height at least~$v(\gamma)$ and at most $v(\gamma')-1-i$ by Lemma~\ref{lem:bulge-dent}, we obtain $v(\gamma)+1\leq v(\gamma')-i$.
The $i$th bulge of~$y$ is followed by~$v(\gamma')-i\geq v(\gamma)+1$ many $\ustep$-steps of~$\gamma'$, which must all contribute to~$S(\gamma)$, contradicting the fact that $|S(\gamma)|=v(\gamma)$ by~(ii).

This completes the proof of the lemma.
\end{proof}

\subsection{Movement of gliders in one time step}

We now define a bijection~$g$ between gliders in the sets~$\Gamma(x)$ and~$\Gamma(f(x))$, so that repeatedly applying this bijection will enable us to track the movement of each glider over time.
Whereas the mapping~$f$ defined in Section~\ref{sec:factor} describes the changes from~$x$ to~$f(x)$ on the level of~0s and~1s, the bijection~$g$ describes the changes on the level of gliders.

To define the bijection $g:\Gamma(x)\rightarrow \Gamma(f(x))$, we let $M(x)\seq \Gammaf(x)$ be the set of free gliders that are not captured by any other free glider, i.e.,
\marginpar{$M(x)$}
\begin{subequations}
\label{eq:g}
\begin{equation}
\label{eq:Mx}
M(x):=\{\gamma \in \Gammaf(x) \mid \text{there is no $\gamma' \in \Gammaf(x)$ such that $\gamma' \succ \gamma$} \}.
\end{equation}
For any glider $\gamma=:(A,B)\in M(x)$ we consider the set~$S(\gamma)$ defined in~\eqref{eq:Sgamma} and define
\marginpar{$g(\gamma)$}
\begin{equation}
\label{eq:g1}
g(\gamma):=(B,S(\gamma)),
\end{equation}
whereas for any other glider $\gamma\in\Gamma(x)\setminus M(x)$ we define
\begin{equation}
\label{eq:g2}
g(\gamma):=\gamma.
\end{equation}
\end{subequations}
In words, the gliders in~$M(x)$ move forward according to~\eqref{eq:g1}, i.e., the $\dstep$-steps of~$\gamma$ in~$\wh{x}$, which are at the positions in~$B$, become the $\ustep$-steps of~$g(\gamma)$ in~$\wh{f(x)}$, and the $\dstep$-steps of~$g(\gamma)$ in~$\wh{f(x)}$ are at the positions in~$S(\gamma)$.
On the other hand, by~\eqref{eq:g2} none of the gliders in~$\Gamma(x)\setminus M(x)$ changes its position, i.e., we have $g(\gamma)=\gamma$ if and only if $\gamma\notin M(x)$.
However, we will see that gliders that do not move change from being non-inverted to inverted, or vice versa.
From~\eqref{eq:g} and Lemma~\ref{lem:capture}~(ii) we immediately see that if $g(\gamma)=\gamma'$, then $v(\gamma')=v(\gamma)$, i.e., the function~$g$ preserves the speed of gliders.
See Figure~\ref{fig:move} for an example illustrating these definitions.

We need to show that the function~$g$ defined in~\eqref{eq:g} is well-defined, i.e., that~$g$ maps a glider from~$\Gamma(x)$ to a glider in~$\Gamma(f(x))$, and that this mapping is bijective.
This is established in Lemma~\ref{lem:g-welldef} below.

\begin{figure}
\makebox[0cm]{ 
\includegraphics[page=2]{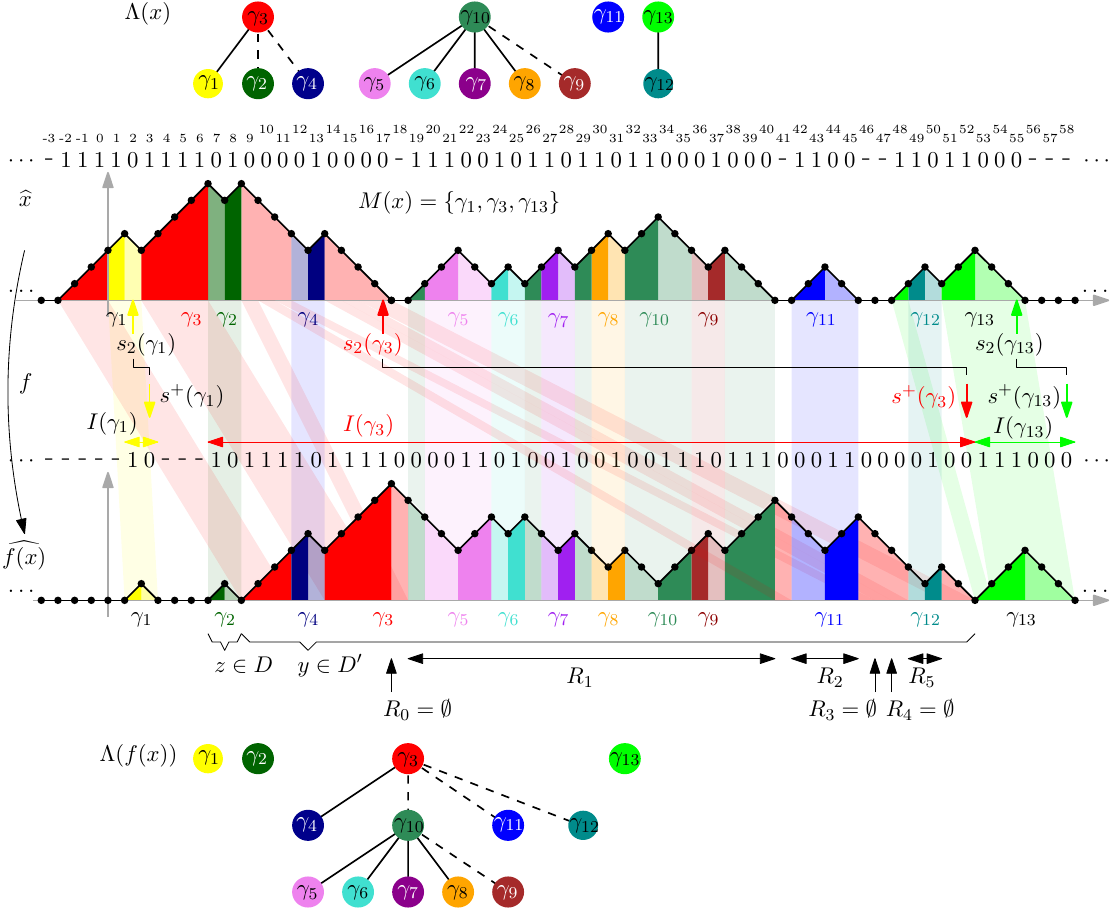}
}
\caption{One application of~$f$ and corresponding movement of gliders, captured by the bijection~$g$.
The pairs of gliders mapped to each other under~$g$ are denoted by the same index~$\gamma_i$ in~$\wh{x}$ and~$\wh{f(x)}$, and connected by shaded trapezoids.
Only the three gliders in~$M(x)=\{\gamma_1,\gamma_3,\gamma_{13}\}$ move forward in this step, whereas all others do not change position.
This figure continues the example from Figure~\ref{fig:capture}.
}
\label{fig:move}
\end{figure}

For the following arguments we rely on the definitions provided in Section~\ref{sec:cap-relation}.
We first establish a simple property about gliders that do not move, i.e., fixed points of the function~$g$.

\begin{lemma}
\label{lem:g-fixed}
For any $\gamma\in\Gamma(x)\setminus M(x)$ and any $\gamma'\in\Gamma(x)$ with $r(\gamma')\seq r(\gamma)$ we have~$g(\gamma')=\gamma'$.
\end{lemma}

\begin{proof}
Let $\gamma\in\Gamma(x)\setminus M(x)$.
The definition~\eqref{eq:g2} gives $g(\gamma)=\gamma$.
If $\gamma\notin\Gammaf(x)$, then clearly all descendants~$\gamma'$ of~$\gamma$ in the forest~$\Lambda(x)$ satisfy~$\gamma'\notin\Gammaf(x)$ and therefore $\gamma'\notin M(x)$ and $g(\gamma')=\gamma'$.
It remains to consider the case $\gamma\in\Gammaf(x)$ and $\gammatilde\succ\gamma$ for some $\gammatilde\in M(x)$.
Consider any descendant~$\gamma'$ of~$\gamma$ in the forest~$\Lambda(x)$.
If $\gamma'\in T(\gamma)$, then $\gamma'\notin \Gammaf(x)$ and hence $g(\gamma')=\gamma'$.
On the other hand, if $\gamma'\in\Gammaf(x)$ then by Lemma~\ref{lem:child-capture} we have~$\gammatilde\succ\gamma'$ and hence~$\gamma'\notin M(x)$, which implies $g(\gamma')=\gamma'$.
\end{proof}

We also need the following auxiliary construction and lemma.
Specifically, given the Motzkin path~$\wh{x}$, we first define another infinite Motzkin path~$\varphi(\wh{x})$.
Lemma~\ref{lem:move} below shows that actually $\varphi(\wh{x})=\wh{f(x)}$, and the lemma establishes additional properties about gliders in the Motzkin paths~$\wh{x}$ and $\wh{f(x)}$; see Figure~\ref{fig:move}.

The Motzkin path~$\varphi(\wh{x})$ is defined as follows:
For all steps $i\in\bigcup_{\gamma \in M(x)} I(\gamma)$ we define
\marginpar{$\varphi(\wh{x})$}
\begin{subequations}
\label{eq:phix}
\begin{equation}
\label{eq:phixI}
\varphi(\wh{x})_i := \begin{cases}
  \ustep & \text{if } \wh{x}_i = {}\dstep, \\
  \dstep & \text{if } \wh{x}_i \in \{\fstep, \ustep\},
\end{cases}
\end{equation}
and otherwise we define
\begin{equation}
\label{eq:phixnotI}
\varphi(\wh{x})_i := {}\fstep.
\end{equation}
\end{subequations}

Recall from Lemma~\ref{lem:capture}~(iv) that for any two gliders $\gamma,\gamma' \in M(x)$, the intervals~$I(\gamma)$ and~$I(\gamma')$ are disjoint.

\begin{lemma}
\label{lem:move}
For any $x \in X_{n,k}$, the Motzkin path~$\varphi(\wh{x})$ defined in~\eqref{eq:phix} satisfies $\varphi(\wh{x})=\wh{f(x)}$.
Moreover, for any glider~$\gamma=:(A,B)\in M(x)$, we have $\varphi(\wh{x})_{I(\gamma)}=z\,y$, where $z\in D$ and $y\in D'$ are base hills in~$\varphi(\wh{x})$ and $y$ has height~$h:=v(\gamma)$, satisfying the following conditions:
\begin{enumerate}[label=(\roman*),leftmargin=8mm]
\item $z$ is the complement of the 0th dent of $\wh{x}_{r(\gamma)}$;
\item the $i$th bulge of~$y$ is the complement of the $i$th dent of~$\wh{x}_{r(\gamma)}$, for all $i=1,\ldots,h-2$;
\item the $i$th dent of~$y$ is the complement of $\wh{x}_{R_i}$, where $R_i:=\bigcup_{\gamma'\in C(\gamma)_i}r(\gamma')$, for all $i=0,\ldots,h-2$;
\item the $\ustep$-steps and $\dstep$-steps of~$y$ not belonging to any of its bulges or dents are at the positions in~$B$ or~$S(\gamma)$, respectively.
\end{enumerate}
\end{lemma}

In the proof of Lemma~\ref{lem:move} we denote steps of the Motzkin path~$\varphi(\wh{x})$ by~$\ustep$, $\dstep$, or $\fstep$, as given by the definition~\eqref{eq:phix}.
Once we have established that $\varphi(\wh{x})=\wh{f(x)}$, it is clear that these correspond to matched~1s, matched~0s, or unmatched~0s ($\hyph$) in~$f(x)$, respectively.

\begin{proof}
Let $\gamma=:(A,B)\in M(x)$ and define~$h:=v(\gamma)$.
We split the interval $I(\gamma)=[s_1(\gamma)+1,s^+(\gamma)]$ into the two smaller intervals $[a,b]:=[s_1(\gamma)+1,s_2(\gamma)]$ and $[c,d]:=[s_2(\gamma)+1,s^+(\gamma)]$.

We first consider the steps of~$\varphi(\wh{x})$ in the interval~$[a,b]$.
If $h=1$, then we have $\gamma=(A,B)=(\{a-1\},\{a\})$, $[a,b]=\{a\}$, and $\wh{x}_{[a,b]}={}\dstep$, and consequently
\begin{equation}
\label{eq:whab1}
\varphi(\wh{x})_{[a,b]}={}\ustep
\end{equation}
by~\eqref{eq:phixI}.
If $h\geq 2$, then let $v_i'$, $i=0,\ldots,h-2$, be the $i$th dent of~$\wh{x}_{r(\gamma)}$, and let $I_i$ be its support.
As $v_0'$ is a valley in~$\wh{x}$, we obtain from~\eqref{eq:phixI} that $z:=\varphi(\wh{x})_{I_0}=\ol{v_0'}$ is a hill, i.e., $z\in D$.
Similarly, for $i=1,\ldots,h-2$ we have that $v_i'$ is a valley of depth at most~$h-1-i$ in~$\wh{x}$ by Lemma~\ref{lem:bulge-dent}, and consequently $u_i:=\varphi(\wh{x})_{I_i}=\ol{v_i'}$ is a hill of height at most~$h-1-i$ in~$\varphi(\wh{x})$.
The remaining steps of~$\wh{x}$ in~$[a,b]$ not belonging to any of the intervals~$I_i$, $i=0,\ldots,h-2$, are at the positions in~$B$ (recall~\eqref{eq:Gamma-hill}), and these are $v(\gamma)=h$ many $\dstep$-steps (recall~\eqref{eq:speed}), so in~$\varphi(\wh{x})$ there are $h$ many $\ustep$-steps at the positions in~$B$.
Combining these observations shows that
\begin{equation}
\label{eq:whab2}
\varphi(\wh{x})_{[a,b]}=z\,\ustep u_1\,\ustep u_2\,\cdots \,\ustep u_{h-2}\,\ustep\ustep
\end{equation}
with $z\in D$ and hills~$u_i\in D$ of height at most~$h-1-i$ for all $i=1,\ldots,h-2$.

We now consider the steps of~$\varphi(\wh{x})$ in the interval~$[c,d]$.
If $h=1$, then we have $[c,d]=S(\gamma)=\{a+1\}$ and $\wh{x}_{[c,d]}\in\{\fstep,\ustep\}$ by Lemma~\ref{lem:capture}~(ii), and consequently
\begin{equation}
\label{eq:whcd1}
\varphi(\wh{x})_{[c,d]}={}\dstep
\end{equation}
by~\eqref{eq:phixI}.
If $h\geq 2$, then let $R_i:=\bigcup_{\gamma'\in C(\gamma)_i}r(\gamma')$ for $i=0,\ldots,h-2$.
By Lemma~\ref{lem:capture}~(i)+(iii), $\wh{x}_{R_i}$ is a hill of height at most~$v(\gamma)-i-1=h-i-1$ with support in the interval~$[c,d]$, and therefore $v_i:=\varphi(\wh{x}_{R_i})=\ol{\wh{x}_{R_i}}$ is a valley of depth at most~$h-i-1$ in~$\varphi(\wh{x})$.
The remaining steps of~$\wh{x}$ in~$[c,d]$ not belonging to any of the intervals~$R_i$, $i=0,\ldots,h-2$, are at the positions in~$S(\gamma)$, and these are $|S(\gamma)|=v(\gamma)=h$ many $\fstep$-steps or $\ustep$-steps by Lemma~\ref{lem:capture}~(ii), so in~$\varphi(\wh{x})$ there are $h$ many $\dstep$-steps at these positions from~$S(\gamma)$.
Combining these observations shows that
\begin{equation}
\label{eq:whcd2}
\varphi(\wh{x})_{[c,d]}=v_0\dstep\,v_1\dstep\,\cdots\, v_{h-2}\dstep\,\dstep
\end{equation}
with valleys~$v_i$, i.e., $\ol{v_i}\in D$, of depth at most~$h-1-i$ for all $i=0,\ldots,h-2$.

If $h=1$, then from~\eqref{eq:whab1} and~\eqref{eq:whcd1} we obtain
\begin{equation}
\label{eq:phihx1}
\varphi(\wh{x})_{I(\gamma)}={}\ustep\,\dstep{}=\underbrace{1\,0}_{=:y},
\end{equation}
i.e., the statements in the lemma are trivially satisfied with $z:=\varepsilon$.
In particular, the $\ustep$-step and the $\dstep$-step of $y=10$ are at the positions~$a\in\{a\}=B$ and $a+1\in\{a+1\}=S(\gamma)$.
If $h\geq 2$, then from~\eqref{eq:whab2} and~\eqref{eq:whcd2}, using the stated constraints on the height of the hills~$u_i$ and the depths of the valleys~$v_i$ we obtain
\begin{equation}
\label{eq:phihx2}
\varphi(\wh{x})_{I(\gamma)}=z\,\underbrace{1\,u_1\,1\,u_2\,\cdots\, 1\,u_{h-2}\,1\,1\,v_0\,0\,v_1\,0\,\cdots\,0\,v_{h-2}\,0\,0}_{=:y},
\end{equation}
i.e., we indeed have $z\in D$ and $y\in D'$ (recall~\eqref{eq:hill} and Lemma~\ref{lem:bulge-dent}).
Note that the 0s and~1s in~\eqref{eq:phihx1} and~\eqref{eq:phihx2} denote bits that are matched to another bit inside~$y$.
Moreover, $z=\ol{v_0'}$ is the complement of the 0th dent of~$\wh{x}_{r(\gamma)}$, the $i$th bulge of~$y$ is~$u_i=\ol{v_i'}$, which is the complement of the $i$th dent of~$\wh{x}_{r(\gamma)}$, and the $i$th dent of~$y$ is $v_i=\ol{\wh{x}_{R_i}}$.
Furthermore, the $\ustep$-steps or $\dstep$-steps of~$y$ not belonging to any of its bulges or dents are at the positions in~$B$ or~$S(\gamma)$, respectively.
This proves statements~(i)--(iv) in the lemma.

As every $\gamma\in M(x)$ satisfies~\eqref{eq:phihx1} or~\eqref{eq:phihx2}, and the steps of~$\varphi(\wh{x})$ not belonging to any of the intervals~$I(\gamma)$, $\gamma\in M(x)$, are $\fstep$-steps by~\eqref{eq:phixnotI}, we see that $\varphi(\wh{x})$ is a Motzkin path that never moves below the abscissa and all of whose $\fstep$-steps are on the abscissa.
We also see that $\varphi(\wh{x})_{I(\gamma)}$ is a base hill in~$\varphi(\wh{x})$ for all~$\gamma\in M(x)$.
Clearly, every subpath of~$\wh{x}$ on some interval~$I$ of length~$n$ has exactly $k$ many $\dstep$-steps, and therefore~\eqref{eq:phixI} yields that the subpath of~$\varphi(\wh{x})$ on~$I$ has exactly $k$ many $\ustep$-steps.
Using this and the observation that $\varphi(\wh{x})$ has periodicity~$n$, it follows that $\varphi(\wh{x})=\wh{x'}$ for some $x'\in X_{n,k}$.
As the 1s in~$x'$ (corresponding to $\ustep$-steps on the Motzkin path $\varphi(\wh{x})$) are precisely at the positions of the matched~0s in~$x$ (corresponding to $\dstep$-steps on the Motzkin path~$\wh{x}$), we indeed have $x'=f(x)$, i.e., $\varphi(\wh{x})=\wh{f(x)}$, as claimed.
\end{proof}

For the next lemma, recall the definition of the equivalence relation~$\sim$ between gliders stated before~\eqref{eq:equiv-classes}.

\begin{lemma}
\label{lem:g-welldef}
For any $x\in X_{n,k}$, the mapping $g:\Gamma(x)\rightarrow \Gamma(f(x))$ defined in~\eqref{eq:g} is a bijection with the following properties:
\begin{enumerate}[label=(\roman*),leftmargin=8mm]
\item
For two gliders $\gamma,\gamma'\in\Gamma(x)$ with $\gamma\sim\gamma'$ we have $s(\gamma)-s(\gamma')=s(g(\gamma))-s(g(\gamma'))$.
\item
For two gliders $\gamma,\gamma'\in\Gamma(x)$ with $v(\gamma)=v(\gamma')$ and $s(\gamma)<s(\gamma')$ we have $s(g(\gamma))<s(g(\gamma'))$.
\item
For two gliders $\gamma,\gamma'\in \Gamma(x)$ with $s_2(\gamma)<s_2(\gamma')$ (in particular, $\gamma'$ is not trapped by~$\gamma$) we have $s_2(g(\gamma))>s_2(g(\gamma'))$ if and only if $g(\gamma')$ is trapped by~$g(\gamma)$.
\end{enumerate}
\end{lemma}

In words, property~(i) asserts that the mapping~$g$ is $n$-periodic, i.e., for gliders whose positions differ in a multiple of~$n$ their images under~$g$ also have the same distance.
Property~(ii) asserts that the relative positions of gliders of the same speed are the same in~$\wh{x}$ and~$\wh{f(x)}$.
Property~(iii) describes how the positions and trappedness relations of two gliders change during an overtaking.

\begin{proof}
Let $\gamma=:(A,B)\in M(x)$ be as in Lemma~\ref{lem:move}.

By Lemma~\ref{lem:move}~(iv), the Motzkin path~$\wh{f(x)}$ has a base hill~$y\in D'$, and the $\ustep$-steps and $\dstep$-steps not belonging to any of its bulges or dents are at the positions in~$B$ or~$S(\gamma)$, respectively.
Consequently, from~\eqref{eq:Gamma-hill} we see that~$(B,S(\gamma))$ is indeed a glider of~$\Gamma(f(x))$ and hence the definition~\eqref{eq:g1} is well-defined.

Now consider the 0th dent of~$\wh{x}_{r(\gamma)}$, and let $I_0$ be its support.
By Lemma~\ref{lem:move}~(i), $z:=\wh{f(x)}_{I_0}=\ol{\wh{x}_{I_0}}\in D$ is a base hill of the Motzkin path~$\wh{f(x)}$, so by~\eqref{eq:Gamma} the recursion~$\Gamma(z)$ yields the same results for~$\wh{x}$ and~$\wh{f(x)}$.
It follows that for gliders~$\gamma'\in\Gamma(x)\setminus\Gammaf(x)$ with $r(\gamma')\seq I_0$ the definition~$g(\gamma')=\gamma'$ given in~\eqref{eq:g2} is well-defined, and it is a bijection between those sets of gliders.

Now consider the $i$th dent of~$\wh{x}_{r(\gamma)}$ for some $i=1,\ldots,h-2$, and let $I_i$ be its support.
By Lemma~\ref{lem:move}~(ii), $u_i:=\wh{f(x)}_{I_i}=\ol{\wh{x}_{I_i}}\in D$ is the $i$th bulge of the base hill~$y\in D'$ of the Motzkin path~$\wh{f(x)}$, so by~\eqref{eq:Gamma} the recursion~$\Gamma(u_i)$ yields the same results for~$\wh{x}$ and~$\wh{f(x)}$.
It follows that for gliders~$\gamma'\in\Gamma(x)\setminus\Gammaf(x)$ with $r(\gamma')\seq I_i$ the definition~$g(\gamma')=\gamma'$ given in~\eqref{eq:g2} is a well-defined bijection between those sets of gliders.

Now consider the interval~$R_i:=\bigcup_{\gamma'\in C(\gamma)_i}r(\gamma')$ for some $i=0,\ldots,h-2$.
By Lemma~\ref{lem:capture}~(iii), $\wh{x}_{R_i}\in D$ is a hill in~$\wh{x}$, either a base hill or a bulge of some base hill.
By Lemma~\ref{lem:move}~(iii), $\wh{f(x)}_{R_i}=\ol{\wh{x}_{R_i}}$ is the $i$th dent of the base hill~$y\in D'$ of the Motzkin path~$\wh{f(x)}$, so by~\eqref{eq:Gamma} the recursion~$\Gamma(\wh{x}_{R_i})$ yields the same results for~$\wh{x}$ and~$\wh{f(x)}$.
It follows that for gliders~$\gamma'\in\Gamma(x)\setminus\Gammaf(x)$ with $r(\gamma')\seq R_i$ the definition~$g(\gamma')=\gamma'$ given in~\eqref{eq:g2} is a well-defined bijection between those sets of gliders.

We have described how the function~$g$ maps all gliders in~$M(x)\seq \Gammaf(x)$, the gliders in~$\Gammaf(x)\setminus M(x)$ (they are in the set~$C(\gamma)$ for some~$\gamma\in M(x)$, so their range lies in~$I(\gamma)$), and those in~$\Gamma(x)\setminus\Gammaf(x)$ (their range lies in~$I(\gamma)$ for some~$\gamma\in\Gammaf(x)$).
This proves that $g$ is a bijection.

To verify property~(i), consider two gliders $\gamma,\gamma'\in\Gamma(x)$ with $\gamma\sim\gamma'$.
As $\wh{x}$ has periodicity~$n$, the set~$S(\gamma)$ is obtained from~$S(\gamma')$ by adding $s(\gamma)-s(\gamma')$ to all elements, i.e., $S(\gamma)=S(\gamma')+s(\gamma)-s(\gamma')$.
As $S(\gamma)$ and $S(\gamma')$ determine the images~$g(\gamma)$ and~$g(\gamma')$, we obtain the desired statement.

We now prove~(iii).
If $g(\gamma')$ is trapped by~$g(\gamma)$, then Lemma~\ref{lem:child-props} directly gives that $s_2(g(\gamma))>s_2(g(\gamma'))$.
To prove the converse, we assume that $s_2(g(\gamma))>s_2(g(\gamma'))$.
Clearly we have $s_2(g(\gamma))\geq s_2(\gamma)$ and $s_2(g(\gamma'))\geq s_2(\gamma')$ with equality if and only if $\gamma\notin M(x)$ and $\gamma'\notin M(x)$, respectively.
From $s_2(g(\gamma))>s_2(g(\gamma'))\geq s_2(\gamma')>s_2(\gamma)$ we thus obtain $\gamma\in M(x)$.
As the position~$s_2(g(\gamma'))$ lies strictly between $s_1(g(\gamma))=s_2(\gamma)$ and $s_2(g(\gamma))$, Lemma~\ref{lem:child-props} yields that $g(\gamma')$ is trapped by~$g(\gamma)$.

It remains to prove~(ii).
For two gliders~$\gamma,\gamma'\in\Gamma(x)$ with $v(\gamma)=v(\gamma')$, Lemma~\ref{lem:child-props} gives that $r(\gamma)\cap r(\gamma')=\emptyset$, which implies that the three conditions $s_1(\gamma)<s_1(\gamma')$, $s_2(\gamma)<s_2(\gamma')$ and $s(\gamma)<s(\gamma')$ are all equivalent, and also neither of the two gliders~$\gamma$ and~$\gamma'$ is trapped by the other.
The same statements hold for $g(\gamma)$ and~$g(\gamma')$, as $v(g(\gamma))=v(\gamma)=v(\gamma')=v(g(\gamma'))$, and so the claim follows from~(iii).

This completes the proof of the lemma.
\end{proof}

The next lemma describes how the bijection~$g$ affects the trappedness relations between gliders from~$\Gamma(x)$ and~$\Gamma(f(x))$; see Figure~\ref{fig:move}.

\begin{lemma}
\label{lem:move-trapped}
Let $T_x(\gamma)$ and $C_x(\gamma)$ \marginpar{$T_x(\gamma)$ $C_x(\gamma)$} be the sets~$T(\gamma)$ and~$C(\gamma)$ of gliders trapped and captured by~$\gamma$, respectively, defined in~\eqref{eq:Tgamma} and~\eqref{eq:Cgamma} for a given bitstring~$x\in X_{n,k}$.
Then we have:
\begin{enumerate}[label=(\roman*),leftmargin=8mm]
\item
For any glider $\gamma\in M(x)$, we have $g(\gamma)\in\Gammaf(f(x))$ and $T_{f(x)}(g(\gamma))=C_x(\gamma)\cup \bigcup_{\gamma'\in C_x(\gamma)} T_x(\gamma')$.
In words, the glider~$g(\gamma)$ is free and the gliders trapped by~$g(\gamma)$ are precisely the free gliders captured by~$\gamma$, and the gliders trapped by them.
In particular, $T_x(\gamma)\cap T_{f(x)}(g(\gamma))=\emptyset$, i.e., none of the gliders trapped by~$\gamma$ is trapped by~$g(\gamma)$.
\item
For any glider $\gamma\in\Gamma(x)\setminus M(x)$ we have $T_{f(x)}(g(\gamma))=T_x(\gamma)$.
In words, the gliders trapped by~$\gamma$ and~$g(\gamma)$ are the same.
\end{enumerate}
\end{lemma}

\begin{proof}
Let $\gamma\in\Gamma(x)\setminus M(x)$ and define~$\Gamma(\gamma):=\{\gamma'\in\Gamma(x)\mid r(\gamma')\seq r(\gamma)\}$ as the set of all descendants of~$\gamma$ in the forest~$\Lambda(x)$, including~$\gamma$ itself.
By Lemma~\ref{lem:g-fixed} we have $g(\gamma')=\gamma'$ for all $\gamma'\in \Gamma(\gamma)$, in particular $g(\gamma)=\gamma$.
This proves that the set~$\Gamma(\gamma)$ has the same ancestor-descendant relations in the forests~$\Lambda(x)$ and~$\Lambda(f(x))$, respectively, and the same trappedness relations.

This observation directly gives~(ii).

To prove~(i), let $\gamma\in M(x)$.
By~\eqref{eq:g1} and Lemma~\ref{lem:move}~(iv), $y:=\wh{f(x)}_{r(g(\gamma))}\in D'$ is a base hill in~$\wh{f(x)}$, and thus $g(\gamma)$ is the root of the tree in the forest~$\Lambda(f(x))$ that arises from the recursive computation of~$\Gamma(y)$.
Consequently, $g(\gamma)$ is not trapped by any other glider and therefore~$g(\gamma)\in\Gammaf(f(x))$.
Furthermore, by Lemma~\ref{lem:move}~(iii), for every $i=0,\ldots,h-2$ the support of the $i$th dent of~$y$ is the interval~$R_i:=\bigcup_{\gamma'\in C_x(\gamma)_i} r(\gamma')$.
Every such glider~$\gamma'\in C_x(\gamma)$ satisfies~$\gamma'\notin M(x)$ and therefore the observation from the beginning applies to the set~$\Gamma(\gamma')$.
The gliders from~$\Gamma(f(x))$ trapped by~$g(\gamma)$ are exactly the gliders whose range lies in one of the dents of~$y$, so this is exactly the set~$C_x(\gamma)\cup \bigcup_{\gamma'\in C_x(\gamma)}T_x(\gamma')$, as claimed in~(i).
\end{proof}

We say that a glider~$\gamma\in\Gamma(x)$ is \emph{open}, \marginpar{open} if in the Motzkin path~$\wh{x}$, the step to the right of~$\gamma$ at position~$s_2(\gamma)+1$ is a $\fstep$-step, i.e., the corresponding bit in~$x$ is an unmatched~0.
Note that if~$\gamma$ is open, then it must be free and therefore non-inverted.
The next lemma asserts that if a glider is open, then it has moved in the preceding time step.

\begin{lemma}
\label{lem:open-move}
Let $\gamma\in\Gamma(x)$ be such that $g(\gamma)$ is open in~$\wh{f(x)}$.
Then we have $\gamma\in M(x)$.
\end{lemma}

\begin{proof}
For the sake of contradiction suppose that $\gamma\notin M(x)$, i.e., $g(\gamma)=\gamma$.
In this case $\gamma$ must be inverted, and by Lemma~\ref{lem:dstep} the step of~$\wh{x}$ at position~$s_2(\gamma)+1$ is a $\dstep$-step, i.e., the corresponding bit in~$x$ is a matched~0.
This however implies that the step of~$\wh{f(x)}$ at this position is an $\ustep$-step (matched~1 in~$x$), contradicting the assumption that $g(\gamma)$ is open.
\end{proof}

\subsection{Cycle invariants}

Clearly, the bijection~$g$ defined in~\eqref{eq:g} preserves the speeds of all gliders.
Together with the `uniformity' property of~$g$ stated in Lemma~\ref{lem:g-welldef}~(i), we obtain that the speed set~$V(x)$ defined in~\eqref{eq:Vx} is invariant along the cycle~$C(x)$.

\begin{lemma}
\label{lem:Vx-invariant}
For any $x\in X_{n,k}$ we have $V(x)=V(f(x))$.
Consequently, for any vertex~$y$ on the cycle~$C(x)$, we have $V(x)=V(y)$.
In words, the speed set is invariant along each cycle.
\end{lemma}

Lemma~\ref{lem:Vx-invariant} shows that if $V(x)\neq V(y)$, then $x$ and $y$ are not on the same cycle.
However, if $V(x)=V(y)$ then $x$ and $y$ may still be on different cycles.
For example, for $x=1010\hyph\hyph\in X_{6,2}$ and $y=10\hyph10\hyph\in X_{6,2}$ we have $V(x)=V(y)=\{1,1\}$, but in both cycles~$C(x)$ and~$C(y)$ the two gliders of speed~1 are in different relative distances to each other.
Furthermore, these distances do not change along the cycles, as all gliders have the same speed.

We can derive the following somewhat refined condition from Lemma~\ref{lem:Vx-invariant}, stated in Lemma~\ref{lem:Zx-invariant} below.
For this we write $V^-(x)$ \marginpar{$V^-(x)$} for the set obtained from the multiset~$V(x)$ by removing duplicates, and for each $v\in V^-(x)$ we write $\nu_v(x)$ \marginpar{$\nu_v(x)$} for the number of occurrences of~$v$ in~$V(x)$, i.e., $\nu_v(x)$ is the number of gliders of speed~$v$.
We can then write $V(x)$ compactly as
\begin{equation}
\label{eq:Vx-compact}
V(x)=\{v^{\nu_v(x)}\mid v\in V^-(x)\},
\end{equation}
i.e., we write the elements of~$V^-(x)$ with their multiplicities~$\nu_v(x)$ as exponents to indicate repetition.
We refer to an inclusion maximal sequence of coupled gliders as a \emph{train}, and its \emph{size} is the number of gliders belonging to this train. \marginpar{train}
For example, in Figure~\ref{fig:coupled}~(a), the gliders $\gamma_2,\gamma_3$ form a train of size~2, the gliders $\gamma_4,\gamma_5,\gamma_6$ form a train of size~3, and the gliders $\gamma_7,\gamma_8$ form a train of size~2, whereas $\gamma_1$ and~$\gamma_9$ are both their own train of size~1 each.
We consider the train sizes formed by all equivalence classes of gliders of speed~$v$ from left to right, and we write $z_v(x)$ \marginpar{$z_v(x)$} for the corresponding cyclic composition of~$\nu_v(x)$, i.e., for the partition of~$\nu_v(x)$ that is ordered up to cyclic shifts (orderings that differ only by cyclic shift are considered equivalent).
In the example shown in Figure~\ref{fig:coupled}, we have $z_1(x)=(1,2,3)=(2,3,1)=(3,1,2)$, $z_2(x)=(2)$, $z_3(x)=(1)$, $z_1(y)=(2,1,3)=(1,3,2)=(3,2,1)$, $z_2(y)=(2)$, and $z_3(y)=(1)$.
We then define
\marginpar{train \\ compo\-sition}
\marginpar{$Z(x)$}
\begin{equation*}
Z(x):=\{v^{z_v(x)}\mid v\in V^-(x)\},
\end{equation*}
which we refer to as \emph{train composition}.
Note that $Z(x)$ is obtained from~\eqref{eq:Vx-compact} by replacing each exponent~$\nu_v(x)$ by a cyclic composition of it, i.e., by specifying how the gliders of speed~$v$ are grouped into trains.

\begin{lemma}
\label{lem:Zx-invariant}
For any $x\in X_{n,k}$, the gliders $\gamma,\gamma'\in \Gamma(x)$ are coupled if and only if $g(\gamma),g(\gamma')\in\Gamma(f(x))$ are coupled.
Furthermore, for any vertex~$y$ on the cycle~$C(x)$, we have $Z(x)=Z(y)$.
In words, the relative positions of gliders in their trains and the cyclic order of trains are invariant along each cycle.
\end{lemma}

\begin{proof}
Let $\gamma=:(A,B)\in\Gamma(x)$ and~$\gamma'=:(A',B')\in\Gamma(x)$ be two coupled gliders.
By Lemma~\ref{lem:coupled-trapped}, $\gamma$ and~$\gamma'$ together with all gliders in the set~$B(\gamma,\gamma')$ defined in~\eqref{eq:Bgamma}, are trapped by the same gliders.
In particular, they are either all free or all trapped.
Furthermore, by Lemma~\ref{lem:coupled-capture} and the fact that all steps of~$\wh{x}$ on the interval~$[s_0(\gamma),s_2(\gamma')]$ belong to gliders of speed at most~$v(\gamma)=v(\gamma')$, if one of~$\gamma$ or~$\gamma'$ is captured by some glider (of larger speed), then both~$\gamma$, $\gamma'$ and all gliders in~$B(\gamma,\gamma')$ are captured by the same glider.
We conclude that if $\gamma,\gamma'\in\Gammaf(x)\setminus M(x)$, then all $\gammatilde\in B(\gamma,\gamma')$ satisfy $\gammatilde\notin M(x)$ and $g(\gammatilde)=\gammatilde$, and therefore the steps of~$\wh{f(x)}$ between the last step of~$g(\gamma)=\gamma$ and the first step of~$g(\gamma')=\gamma'$ belong to gliders of strictly smaller speed (recall Lemma~\ref{lem:g-fixed}).
On the other hand, if $\gamma,\gamma'\in M(x)$, then $\wh{x}_{[s_2(\gamma)+1,s_0(\gamma')-1]}$ is a hill of height less than~$v(\gamma)=v(\gamma')$ and hence we have~$s^+(\gamma)=s_1(\gamma')$, $C(\gamma)\supseteq B(\gamma,\gamma')$ and~$S(\gamma)=A'$, which implies that~$g(\gamma)=(B,A')$ and~$g(\gamma')=(B',S(\gamma'))$.
Consequently, the steps of~$\wh{f(x)}$ between the last step of~$g(\gamma)$ and the first step of~$g(\gamma')$ are obtained by complementation from the 0th dent of~$\wh{x}_{r(\gamma')}$ (Lemma~\ref{lem:move}~(i)), i.e., these steps belong to gliders $\gammatilde\in\Gamma(x)\setminus M(x)$ with $g(\gammatilde)=\gammatilde$ that have strictly smaller speed in~$\wh{x}$ and~$\wh{f(x)}$.
We have argued that in both cases $g(\gamma)$ and~$g(\gamma')$ are coupled.
To complete the proof of the first part of the lemma, it remains to show the reverse implication.
This follows by observing that~$f^i(f(x))=x$ for some sufficiently large integer~$i>0$, so the number of trains of some fixed speed~$v\in V^-(x)$ cannot decrease from~$\wh{x}$ to~$\wh{f(x)}$, i.e., $z_v(x)$ and~$z_v(f(x))$ have the same length and the same entries, but possibly in a different cyclic order.
By Lemma~\ref{lem:g-welldef}~(ii) the cyclic order of trains must in fact be the same, i.e., we have $z_v(x)=z_v(f(x))$, which proves the second part of the lemma.
\end{proof}

For the bitstrings~$x$ and~$y$ shown in Figure~\ref{fig:coupled}~(a) and~(b), respectively, we have $V(x)=V(y)$, but $Z(x)\neq Z(y)$, as the train sizes of the equivalence classes of gliders of speed~1 appear in different cyclic order, so $x$ and~$y$ lie on different cycles.

It seems that one should be able to give a complete combinatorial interpretation of the set of cycles of the factor~$\cC_{n,k}$ defined in~\eqref{eq:factor} via the speed sets of gliders, their train composition, and the relative distances of the trains, but unfortunately we do not have such an interpretation.
Specifically, given two vertices~$x$ and~$y$ with the same train compositions $Z(x)=Z(y)$, in general we do not have an efficient way to decide whether $x$ and~$y$ lie on the same cycle, other than computing $f^i(x)$ for $i=0,1,\ldots$ and checking if $y=f^i(x)$ holds for one of them.

We can also derive the following somewhat coarser condition from Lemma~\ref{lem:Vx-invariant}.
It uses the number of gliders~$\nu(x)$ defined in~\eqref{eq:nu}, which satisfies $\nu(x)=|V(x)|$, and which can be computed very easily as the number~$d(x)$ of descents in~$x$ by Lemma~\ref{lem:trans}, without invoking the somewhat intricate glider partition introduced in Section~\ref{sec:glider}.

\begin{lemma}
\label{lem:trans-invariant}
For any $x\in X_{n,k}$ and any vertex~$y$ on the cycle~$C(x)$, we have~$\nu(x)=\nu(y)$ and consequently $d(x)=d(y)$.
In words, the number of descents is invariant along each cycle.
\end{lemma}

For example, the number of descents in all bitstrings in the four cycles shown in Figure~\ref{fig:gliders}~(a)--(d) is 1, 1, 2, 3 respectively.
As mentioned before, the number of descents equals the number of inclusion maximal substrings of~1s (or 0s), so this quantity is also a cycle invariant.
As Figure~\ref{fig:gliders}~(d) shows, the lengths of those inclusion maximal substrings of 1s are not invariant, however (neither for 0s).

\subsection{Movement of slowest gliders}
\label{sec:slow-move}

Gliders with the minimum speed play an important role in our later arguments, and we will now analyze their movement.

The next lemma can be seen as the converse of Lemma~\ref{lem:open-move} for the case of gliders of the minimum speed.
Specifically, it asserts that if certain gliders of the minimum speed move in some step, then after this step they are open.

\begin{lemma}
\label{lem:move-open}
Let $\gamma\in M(x)$ be the rightmost glider in a train of the minimum speed~$v(\gamma)=\min V(x)$.
Then $g(\gamma)$ is open in~$\wh{f(x)}$.
\end{lemma}

Lemma~\ref{lem:move-open} is illustrated in Figure~\ref{fig:open}.
Intuitively, in the step from~$x$ to~$f(x)$, the glider~$\gamma$ moves at the minimum speed, whereas all other moving gliders either move at a strictly faster speed, or they belong to another train of the same speed, i.e., a train that is separated from~$\gamma$'s train by a positive number of steps.
This creates a `gap' to the right of~$g(\gamma)$, i.e., a step that does not belong to any other glider.

\begin{proof}
Let $\gamma \in M(x)$ be a glider of the minimum speed~$v(\gamma)=\min V(x)$.
As $V(x)=V(f(x))$ by Lemma~\ref{lem:Vx-invariant}, $g(\gamma)$ also has the minimum speed among gliders from~$\Gamma(f(x))$.
Consequently, $\gamma$ and~$g(\gamma)$ are both clean by Lemma~\ref{lem:clean} and therefore we have $I(\gamma)=r(g(\gamma))=r(\gamma)+v(\gamma)$ and $\wh{x}_{r(\gamma)}=\wh{f(x)}_{r(g(\gamma))}=1^{v(\gamma)}0^{v(\gamma)}$.

Now let $\gamma \in M(x)$ be the rightmost glider in a train of the minimum speed, and let $I(\gamma')$, $\gamma'\in M(x)$, be the closest interval to the right of~$I(\gamma)$.
We claim that $I(\gamma')$ is separated from~$I(\gamma)$ by at least one element.
Indeed, if $v(\gamma')=v(\gamma)=\min V(x)$, this is true because otherwise $g(\gamma)$ and~$g(\gamma')$ were coupled in~$\wh{f(x)}$ but $\gamma$ and~$\gamma'$ were not coupled in~$\wh{x}$, violating Lemma~\ref{lem:Zx-invariant}.
On the other hand, if $v(\gamma')>v(\gamma)$, then the claim is Lemma~\ref{lem:capture}~(v).
Consequently, there is a step of~$\wh{x}$ at a position between~$I(\gamma)$ and~$I(\gamma')$, and by~\eqref{eq:phixnotI} every step of~$\wh{f(x)}$ at such a position is an $\fstep$-step, which proves the lemma.
\end{proof}

Given a vertex~$x\in X_{n,k}$ and a glider~$\gamma\in\Gamma(x)$ of the minimum speed, we now introduce a slightly modified vertex~$h_\gamma(x)$ that differs from~$x$ only in `pushing' the glider~$\gamma$ one step to the right, together with all gliders in its equivalence class~$[\gamma]$.
Formally, for any $x\in X_{n,k}$ and the rightmost glider~$\gamma\in\Gamma(x)$ in a train of the minimum speed~$v(\gamma)=\min V(x)$, we let $g':\Gamma(f^{-1}(x))\rightarrow\Gamma(x)$ be the bijection defined in~\eqref{eq:g} for the string~$f^{-1}(x)$, and we let $\gamma'\in\Gamma(f^{-1}(x))$ be such that $g'(\gamma')=\gamma$.
By Lemma~\ref{lem:Zx-invariant}, $\gamma'$ is the rightmost glider in a train of the minimum speed.
We define $h_\gamma(x)\in X_{n,k}$ \marginpar{$h_\gamma(x)$} as the bitstring obtained from~$x$ by transposing the two bits at positions~$s_1(\gamma')+1$ and~$s_2(\gamma')+1$; see Figures~\ref{fig:open} and~\ref{fig:push}.
We emphasize that the transposed positions are computed based on~$f^{-1}(x)$ (specifically, based on $\gamma'\in\Gamma(f^{-1}(x))$), but the bits are transposed in~$x$.
In fact, we will see that the transposed bits in~$x$ are always a (matched)~1 and a matched~0.
However, when applying parenthesis matching to the modified string~$h_\gamma(x)$, the transposed~0 may be unmatched, and another 0-bit may instead be matched.

Observe that if~$\gamma'\in M(f^{-1}(x))$, then by Lemma~\ref{lem:move-open} the glider~$\gamma$ is open in~$\wh{x}$, i.e., the subpaths of~$\wh{x}$ and~$\wh{h_\gamma(x)}$ on the interval~$[s_0(\gamma),s_2(\gamma)+1]$ are $y\hyph$ and $\hyph y$ for $y:=1^{v(\gamma)}0^{v(\gamma)}$, respectively.
The $\hyph$ indicates an $\fstep$-step, i.e., $y$ is a base hill in both cases, which implies that $\Gamma(h_\gamma(x))=(\Gamma(x)\setminus[\gamma])\cup ([\gamma]+1)$ where $[\gamma]+1:=\big\{(A+1,B+1)\mid (A,B)\in[\gamma]\big\}$, i.e., the gliders in~$\Gamma(h_\gamma(x))$ are obtained from the gliders of~$\Gamma(x)$ by pushing the gliders in the equivalence class~$[\gamma]$ one step to the right, while leaving all other gliders unchanged.
In Figure~\ref{fig:open}~(b), these steps are marked by little arrows on the right.

\begin{figure}
\includegraphics[page=1]{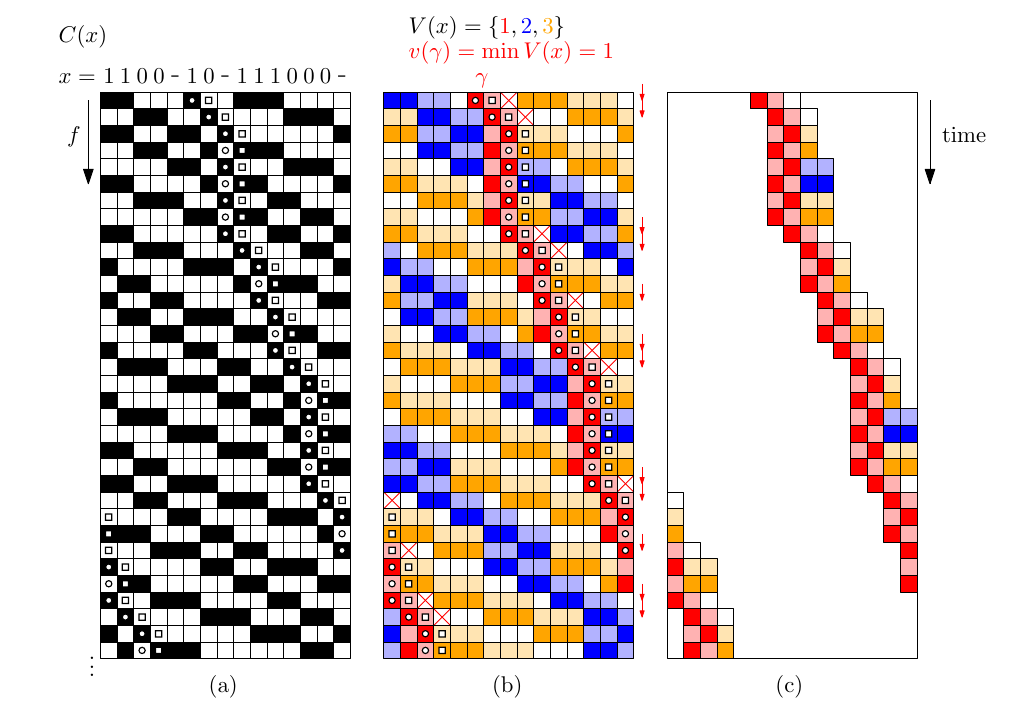}
\caption{Illustration of Lemmas~\ref{lem:move-open} and~\ref{lem:flip-vertex}.
The glider~$\gamma$ has the minimum speed~$v(\gamma)=\min V(x)=1$.
In~(b), the steps in which~$\gamma$ moves are marked by little arrows on the right, and the corresponding unmatched 0-bit guaranteed by the lemmas is marked by a cross.
Furthermore, transposing the two marked bits yields the cycle $C(h_\gamma(x))$ shown in Figure~\ref{fig:push} for the bitstring~$h_\gamma(x)$ that is obtained from~$x$ by pushing the glider~$\gamma$ one position to the right.
Part~(c) shows only the steps in~(b) that are different from the ones in Figure~\ref{fig:push}~(b).
}
\label{fig:open}
\end{figure}

\begin{figure}
\includegraphics[page=2]{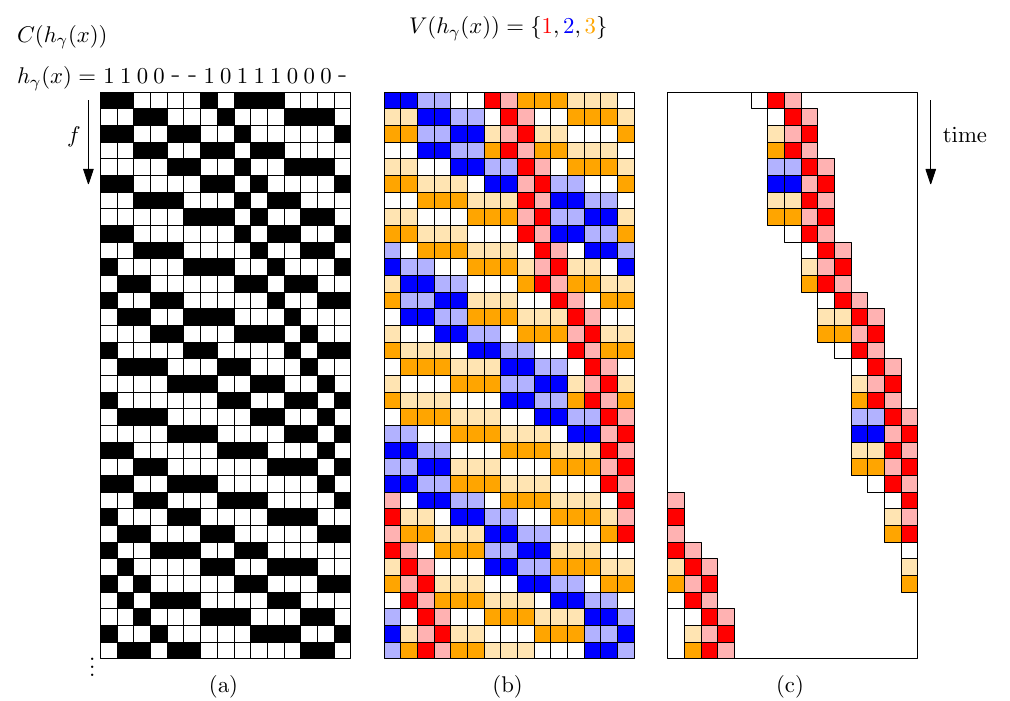}
\caption{Illustration of Lemma~\ref{lem:push}.
This figure continues Figure~\ref{fig:open}.
Part~(c) shows only the steps in~(b) that are different from the ones in Figure~\ref{fig:open}~(b).
To fully appreciate the two figures, the reader is encouraged to print both pages in color and alternately look at them like in a flip-book.
Alternatively, the reader may want to flip back and forth between both pages on a screen reader, with fixed alignment of the page boundaries.}
\label{fig:push}
\end{figure}

On the other hand, if~$\gamma'\notin M(f^{-1}(x))$, then steps of more than one glider may change in~$\Gamma(h_\gamma(x))$ compared to~$\Gamma(x)$, and in particular the position of~$\gamma$ may change by more than~$+1$; see Figures~\ref{fig:open}~(c) and~\ref{fig:push}~(c).
Luckily, in this case we do not need to understand the more complicated effect of applying~$h_\gamma$ on the glider decomposition.

The next lemma shows that the cycle~$C(h_\gamma(x))$ is obtained from the cycle~$C(x)$ by pushing the glider~$\gamma$ and its images under repeated applications of~$g$ one step to the right.
As mentioned before, this intuition is literally true only for steps in~$C(x)$ in which the glider has moved.
Nonetheless, it will be the case that $V(x)=V(h_\gamma(x))$, i.e., the speed sets of gliders in the cycles~$C(x)$ and~$C(h_\gamma(x))$ are identical; recall Lemma~\ref{lem:Vx-invariant}.

\begin{lemma}
\label{lem:push}
Let $\gamma\in\Gamma(x)$ be the rightmost glider in a train of the minimum speed~$v(\gamma)=\min V(x)$.
Then we have $f(h_\gamma(x))=h_{g(\gamma)}(f(x))$.
\end{lemma}

By Lemma~\ref{lem:Zx-invariant}, the relative positions of gliders in their trains is preserved when applying~$g$, so Lemma~\ref{lem:push} can be applied repeatedly to the entire cycle~$C(x)$.

\begin{proof}
In the proof we also consider the glider~$\gamma'\in\Gamma(f^{-1}(x))$ defined before the lemma.
Throughout this proof, we will repeatedly use that $\gamma$ and its images under~$g$ are clean by Lemma~\ref{lem:clean}.
Let $(A,B):=\gamma$, and define the abbreviations $a:=s_0(\gamma)=\min A$, $b:=s_1(\gamma)+1=\max A+1=\min B=a+v(\gamma)$, and $c:=s_2(\gamma)+1=\max B+1=b+v(\gamma)=a+2v(\gamma)$.

The idea is to start with the relation
\begin{equation}
\label{eq:mu1mu0}
\mu_1(f(x))=\mu_0(x)
\end{equation}
and to argue that the sets on the left and right hand side change in the same way, namely by removing one element and adding another, when replacing~$f(x)$ by~$h_{g(\gamma)}(f(x))$ and~$x$ by~$h_\gamma(x)$, respectively, which then proves the lemma.

{\bf Case (a):}
We first consider the case that $s(\gamma)>s(\gamma')$; see Figure~\ref{fig:push-proof}~(a).
Clearly, we have
\begin{equation*}
s_1(\gamma')+1=a \quad \text{and} \quad s_2(\gamma')+1=b.
\end{equation*}
Applying the definition of~$h_\gamma$, we see that $h_\gamma(x)$ is obtained from~$x$ by transposing the bits at positions~$a$ and~$b$.

In~$x$, the 1-bits and 0-bits at the positions in~$A$ and~$B$, respectively, are matched to each other, i.e., we have $A\seq \mu_1(x)$ and $B\seq \mu_0(x)$.
Furthermore, by Lemma~\ref{lem:move-open} the bit at position~$c$ is an unmatched~0, i.e., we have~$c\notin \mu_0(x)$.
Combining these two observations we obtain
\begin{equation}
\label{eq:mu0xa}
\mu_0(h_\gamma(x))=\big(\mu_0(x)\setminus \{b\}\big)\cup \{c\}.
\end{equation}
As $b\in\mu_0(x)$, there is a 1-bit at position~$b$ in~$f(x)$.
Furthermore, from~$c\notin\mu_0(x)$ we see that there is a 0-bit at position~$c$ in~$f(x)$.
These two bits are swapped in~$f(x)$ to obtain~$h_{g(\gamma)}(f(x))$, and consequently we have
\begin{equation}
\label{eq:mu1fxa}
\mu_1\big(h_{g(\gamma)}(f(x))\big)=\big(\mu_1(f(x))\setminus \{b\}\big)\cup\{c\}.
\end{equation}
Combining~\eqref{eq:mu1mu0}, \eqref{eq:mu0xa} and~\eqref{eq:mu1fxa} proves that $f(h_\gamma(x))=h_{g(\gamma)}(f(x))$.

{\bf Case (b):}
We now consider the case that $s(\gamma)=s(\gamma')$.
Clearly, we have
\begin{equation*}
s_1(\gamma')+1=s_1(\gamma)+1=b \quad \text{and} \quad s_2(\gamma')+1=s_2(\gamma)+1=c,
\end{equation*}
i.e., $h_\gamma(x)$ is obtained from~$x$ by transposing the bits at positions~$b$ and~$c$.

\begin{figure}
\includegraphics[page=3]{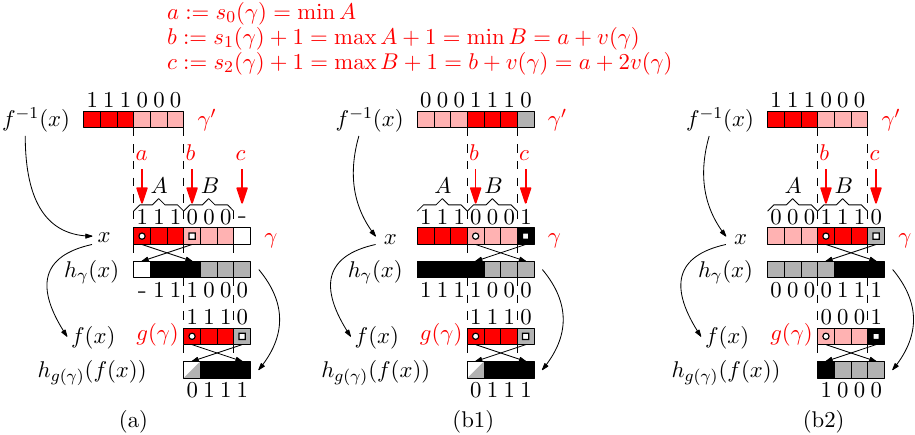}
\caption{Illustration of the proof of Lemma~\ref{lem:push}.
For the bits in~$h_\gamma(x)$ and~$h_{g(\gamma)}(f(x))$, no membership to certain gliders is specified, so they are drawn black (1-bits) and gray (matched 0-bits).
The 0-bits corresponding to half-white/half-gray cells may be matched or unmatched.
}
\label{fig:push-proof}
\end{figure}

We distinguish the subcases whether $\gamma$ is inverted or non-inverted.

{\bf Case (b1):} $\gamma$ is non-inverted; see Figure~\ref{fig:push-proof}~(b1).
In this case~$\gamma'$ is inverted, and by Lemma~\ref{lem:dstep} we have $c\in\mu_0(f^{-1}(x))$ and therefore $c\in\mu_1(x)$.
Furthermore, in~$x$ the 1-bits and 0-bits at the positions in~$A$ and~$B$, respectively, are matched to each other, i.e., we have $A\seq \mu_1(x)$ and $B\seq \mu_0(x)$.
Combining these two observations we obtain
\begin{equation}
\label{eq:mu0xb2}
\mu_0(h_\gamma(x))=\big(\mu_0(x)\setminus\{b\}\big)\cup\{c\}.
\end{equation}
As $b\in\mu_0(x)$, there is a 1-bit at position~$b$ in~$f(x)$.
Furthermore, as $c\in\mu_1(x)$ there is a 0-bit at position~$c$ in~$f(x)$.
These two bits are swapped in~$f(x)$ to obtain~$h_{g(\gamma)}(f(x))$, and consequently we have
\begin{equation}
\label{eq:mu1fxb2}
\mu_1\big(h_{g(\gamma)}(f(x))\big)=\big(\mu_1(f(x))\setminus \{b\}\big)\cup \{c\}.
\end{equation}
Combining~\eqref{eq:mu1mu0}, \eqref{eq:mu0xb2} and~\eqref{eq:mu1fxb2} proves that $f(h_\gamma(x))=h_{g(\gamma)}(f(x))$.

{\bf Case (b2):} $\gamma$ is inverted; see Figure~\ref{fig:push-proof}~(b2).
By Lemma~\ref{lem:dstep} we have $c\in\mu_0(x)$.
Note also that the valley~$\wh{x}_{r(\gamma)}$ does not touch the abscissa, as otherwise $A$ and~$B$ would not belong to the same glider.
Combining these two observations we obtain
\begin{equation}
\label{eq:mu0xb1}
\mu_0(h_\gamma(x))=\big(\mu_0(x)\setminus\{c\}\big)\cup \{b\}.
\end{equation}
As $b\in\mu_1(x)$, there is a 0-bit at position~$b$ in~$f(x)$.
Furthermore, as $c\in\mu_0(x)$ there is a 1-bit at position~$c$ in~$f(x)$.
These two bits are swapped in~$f(x)$ to obtain~$h_{g(\gamma)}(f(x))$, and consequently we have
\begin{equation}
\label{eq:mu1fxb1}
\mu_1\big(h_{g(\gamma)}(f(x))\big)=\big(\mu_1(f(x))\setminus\{c\}\big)\cup\{b\}.
\end{equation}
Combining~\eqref{eq:mu1mu0}, \eqref{eq:mu0xb1} and~\eqref{eq:mu1fxb1} proves that $f(h_\gamma(x))=h_{g(\gamma)}(f(x))$.

This completes the proof of the lemma.
\end{proof}

\subsection{Equations of motion}

The main result of this section, Proposition~\ref{prop:motion} below, describes equations of motion that capture the movement of gliders over time, including their interactions.

For any $x\in X_{n,k}$ we define $x^t:=f^t(x)$  \marginpar{$x^t$} for all $t\geq 0$, and we refer to the parameter~$t$ as \emph{time}.
We let $g^t$ be the bijection defined in~\eqref{eq:g} between the sets~$\Gamma(x^{t-1})$ and~$\Gamma(x^t)$.
Furthermore, for any $\gamma\in\Gamma(x)=\Gamma(x^0)$ we define $\gamma^0:=\gamma$ and \marginpar{$\gamma^t$} $\gamma^t:=g^t(\gamma^{t-1})$ for~$t\geq 1$.

Recall that~$g^t$ preserves the speeds of all gliders, so $v(\gamma^t)=v(\gamma^0)=v(\gamma)$ for all $\gamma\in\Gamma(x)$ and $t\geq 0$.
However, to describe the movement of gliders, i.e., the change of positions~$s(\gamma^t)$, over time correctly, we need to take into account that faster gliders may overtake slower gliders; see Figure~\ref{fig:ct}.
An overtaking happens in two steps, namely first the slower glider is trapped by the faster glider (see the step $t-1\rightarrow t$ in the figure), and later the slower glider is released again, i.e., it is not trapped anymore by the faster glider (see the step $t\rightarrow t+1$ in the figure).
Note that these two steps need not happen directly consecutively, as in between both gliders may get trapped temporarily by an even faster glider.

Formally, we consider the equivalence classes of gliders defined in~\eqref{eq:equiv-classes}, namely gliders that differ by multiples of~$n$.
\marginpar{get trapped/}
\marginpar{released}
We say that \emph{$[\gamma']$ gets trapped by~$[\gamma]$ in step~$t$}, if there are representatives $\gammatilde'\in[\gamma']$ and $\gammatilde\in[\gamma]$ such that $\gammatilde'^{t-1}$ is not trapped by $\gammatilde^{t-1}$ in $x^{t-1}$ and $\gammatilde'^t$ is trapped by $\gammatilde^t$ in~$x^t$.
Similarly, we say that \emph{$[\gamma']$ gets released by~$[\gamma]$ in step~$t$}, if there are representatives $\gammatilde'\in[\gamma']$ and $\gammatilde\in[\gamma]$ such that $\gammatilde'^{t-1}$ is trapped by $\gammatilde^{t-1}$ in $x^{t-1}$ and $\gammatilde'^t$ is not trapped by $\gammatilde^t$ in~$x^t$.

\begin{figure}[b!]
\makebox[0cm]{ 
\includegraphics[page=1]{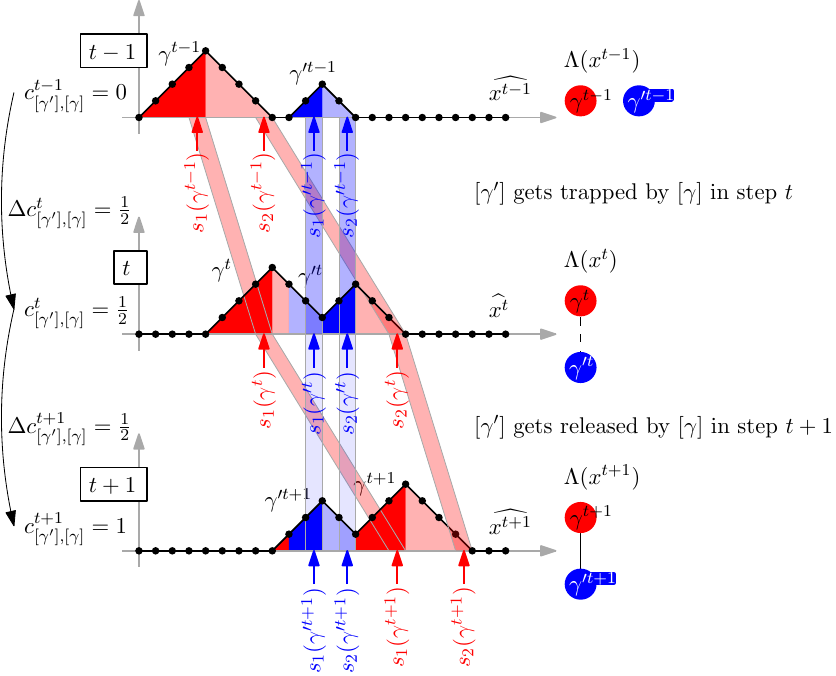}
}
\caption{Illustration of the overtaking counter~$c_{[\gamma'],[\gamma]}^t$.}
\label{fig:ct}
\end{figure}

To track the trapped/released events over time we introduce half-integral parameters~$\Delta c_{[\gamma'],[\gamma]}^t$ and~$c_{[\gamma'],[\gamma]}^t$, defined for any two equivalence classes~$[\gamma],[\gamma']\in\Gamma(x)/{\sim}$.
Specifically, we define
\marginpar{$c^t_{[\gamma'],[\gamma]}$}
\begin{subequations}
\label{eq:ct}
\begin{equation}
\label{eq:dct}
\Delta c_{[\gamma'],[\gamma]}^t:=
\begin{cases}
\frac{1}{2} & \text{$[\gamma']$ gets trapped by~$[\gamma]$ in step~$t$,} \\
\frac{1}{2} & \text{$[\gamma']$ gets released by~$[\gamma]$ in step~$t$,} \\
0 & \text{otherwise,}
\end{cases}
\end{equation}
for $t\geq 1$.
Based on this we define
\begin{equation}
c_{[\gamma'],[\gamma]}^t := \begin{cases}
0 & \text{if $t=0$}, \\
c_{[\gamma'],[\gamma]}^{t-1} + \Delta c_{[\gamma'],[\gamma]}^t & \text{if $t\geq 1$}.
\end{cases}
\end{equation}
\end{subequations}
Figure~\ref{fig:time} shows the evolution of these counters over several time steps for an example with four gliders.

\begin{figure}
\makebox[0cm]{ 
\includegraphics[page=2]{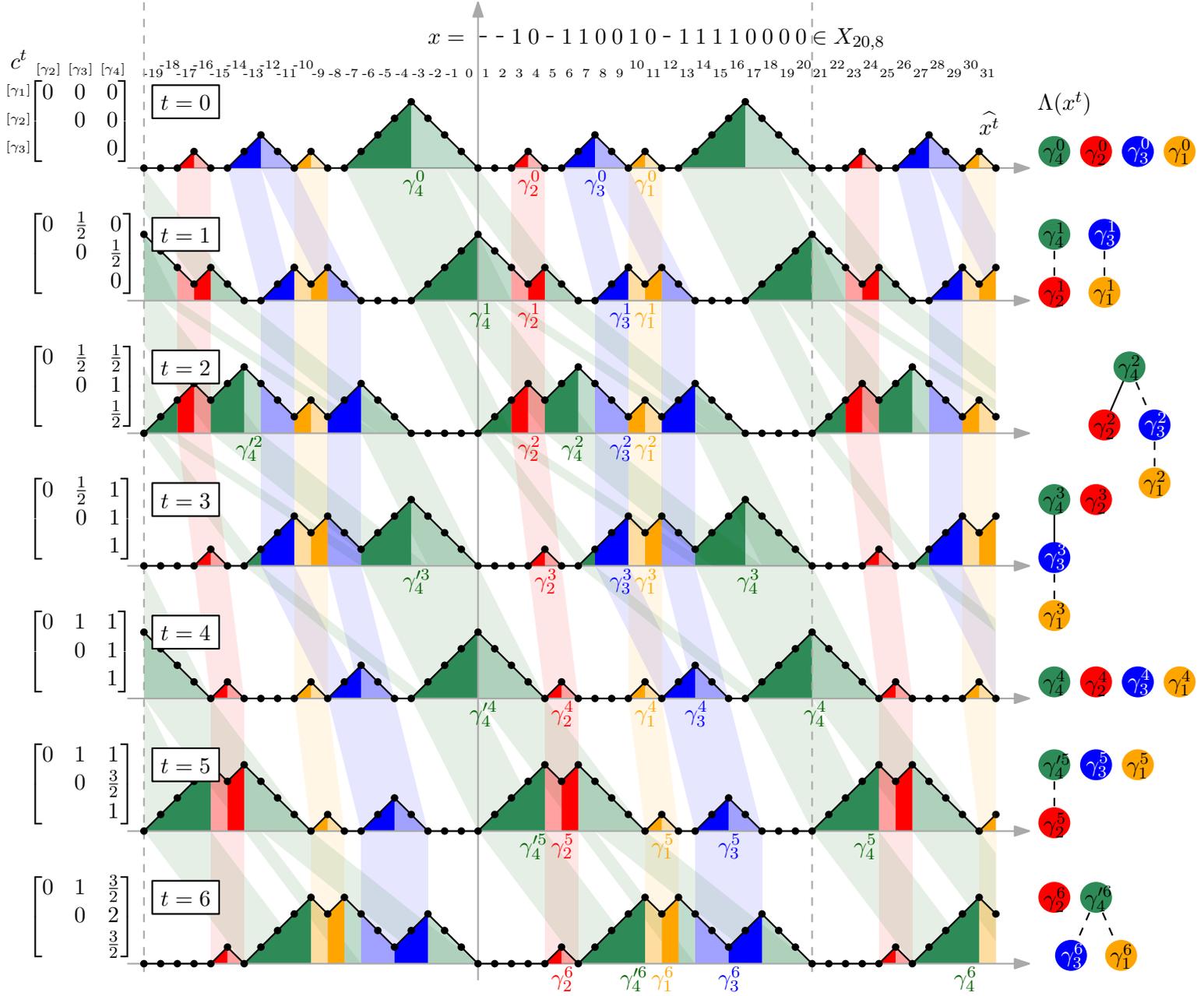}
}
\caption{Illustration of movement of gliders over time.}
\label{fig:time}
\end{figure}

We are now in position to formulate equations of motion that describe the movement of each glider as a function of time~$t$.
These equations involve the position~$s(\gamma)$ and the speed~$v(\gamma)$ of a glider~$\gamma$ introduced in Section~\ref{sec:pos-speed}, and the counters defined in~\eqref{eq:ct}.

\begin{proposition}
\label{prop:motion}
For any $x\in X_{n,k}$ and $\gamma\in\Gamma(x)$, the position of the glider~$\gamma$ at time~$t$ is given by
\begin{equation}
\label{eq:svt}
s(\gamma^t) = s(\gamma^0) + v(\gamma)\cdot t+\sum_{[\gamma']\in \Gamma(x)/{\sim}} 2v(\gamma') c_{[\gamma'],[\gamma]}^t-\sum_{[\gamma']\in \Gamma(x)/{\sim}} 2v(\gamma) c_{[\gamma],[\gamma']}^t.
\end{equation}
\end{proposition}

Note that if $\gamma'$ is trapped by~$\gamma$, then $v(\gamma')<v(\gamma)$ by Lemma~\ref{lem:child-props}.
Consequently, if $v(\gamma')\geq v(\gamma)$ then~$[\gamma']$ does not get trapped and does not get released by~$[\gamma]$ in any time step~$t\geq 1$, so~\eqref{eq:ct} directly yields that $c_{[\gamma'],[\gamma]}^t=0$ for all $t\geq 0$.
It follows that the only nonzero contributions to the first and second summation in~\eqref{eq:svt} come from gliders with $v(\gamma')<v(\gamma)$ or with $v(\gamma')>v(\gamma)$, respectively.

The equation~\eqref{eq:svt} has the following physical interpretation:
The first part of the equation $s(\gamma^t)=s(\gamma^0)+v(\gamma)\cdot t$ captures that the motion of~$\gamma$ is uniform with speed~$v(\gamma)$.
This is all that happens if~$\gamma$ never interacts with any other gliders, which occurs precisely if all gliders have the same speed.
Now let us discuss the additional terms in~\eqref{eq:svt} that make the motion non-uniform, in case gliders of different speeds are present.
The first summation in~\eqref{eq:svt}, which has a positive sign, corresponds to a boost whenever a slower glider~$[\gamma']$, i.e., one with $v(\gamma')<v(\gamma)$, gets trapped by~$[\gamma]$.
The second summation in~\eqref{eq:svt}, which has a negative sign, corresponds to a delay whenever~$[\gamma]$ gets trapped by a faster glider~$[\gamma']$, i.e., one with $v(\gamma')>v(\gamma)$.
Note that when $[\gamma']$ gets trapped or released by~$[\gamma]$ in step~$t$, then we see a change of $+2v(\gamma')\Delta c_{[\gamma'],[\gamma]}^t=+2v(\gamma')\cdot\frac{1}{2}=+v(\gamma')$ (recall~\eqref{eq:dct}) in the equation for~$\gamma$, and a change of $-2v(\gamma')\Delta c_{[\gamma'],[\gamma]}^t=-v(\gamma')$ in the equation for~$\gamma'$, i.e., two terms with the same absolute value but opposite signs.
This can be seen as `energy conservation' in the system of gliders.
For the slower glider~$\gamma'$, the uniform change in position by~$v(\gamma')$ and the delay term $-v(\gamma')$ cancel each other out, so effectively it does not change position in the two steps where it gets trapped and released; see Figure~\ref{fig:ct}.
On the other hand, the total boost received by the faster glider~$\gamma$ in these two steps is $2v(\gamma')$, which is twice the speed of the slower glider, or equivalently, the total number of its steps in the Motzkin path (number of $\ustep$-steps plus $\dstep$-steps; recall~\eqref{eq:speed}).

\begin{proof}
We calculate the change in position of a glider in the step from~$t-1$ to~$t$.
To prove~\eqref{eq:svt} we will show that this change equals
\begin{align}
\Delta s(\gamma^t)&:=s(\gamma^t)-s(\gamma^{t-1}) \notag \\
&=v(\gamma)+\sum_{[\gamma']\in \Gamma(x)/{\sim}} 2v(\gamma') \Delta c_{[\gamma'],[\gamma]}^t-\sum_{[\gamma']\in \Gamma(x)/{\sim}} 2v(\gamma) \Delta c_{[\gamma],[\gamma']}^t. \label{eq:Delta-s1}
\end{align}

We first consider a glider~$\gamma^{t-1}\in M(x^{t-1})$.
Note that $\gamma^{t-1}$ is free by assumption, and $\gamma^t$ is free by Lemma~\ref{lem:move-trapped}~(i), and therefore the second summation on the right hand side of~\eqref{eq:Delta-s1} equals~0.

Recall from~\eqref{eq:speed} that for any glider~$\gamma'=(A,B)\in \Gamma(x)$, we have $v(\gamma')=|A|=|B|$, and therefore the total number of its steps on the Motzkin path is $2v(\gamma')=2|A|=2|B|$.

Using this simple observation, the definitions~\eqref{eq:s12}, as well as~\eqref{eq:g1}, Lemma~\ref{lem:capture}~(ii) and Lemma~\ref{lem:move-trapped}~(i), we get
\begin{align*}
\Delta s_2(\gamma^t)&:=s_2(\gamma^t)-s_2(\gamma^{t-1})=v(\gamma)+\sum_{\gamma'\in T_{x^t}(\gamma^t)} 2v(\gamma'), \\
\Delta s_1(\gamma^t)&:=s_1(\gamma^t)-s_1(\gamma^{t-1})=v(\gamma)+\sum_{\gamma'\in T_{x^{t-1}}(\gamma^{t-1})} 2v(\gamma');
\end{align*}
see also Figure~\ref{fig:ct}.
Combining these two equations via~\eqref{eq:pos} we obtain
\begin{equation}
\label{eq:Delta-s2}
\Delta s(\gamma^t)=\frac{1}{2}\big(\Delta s_1(\gamma^t)+\Delta s_2(\gamma^t)\big)=v(\gamma)+\sum_{\gamma'\in T_{x^t}(\gamma^t)} v(\gamma') + \sum_{\gamma'\in T_{x^{t-1}}(\gamma^{t-1})} v(\gamma').
\end{equation}
We consider the first summation on the right hand side of~\eqref{eq:Delta-s2}.
By Lemma~\ref{lem:move-trapped}~(i), $T_{x^t}(\gamma^t)=C_{x^{t-1}}(\gamma^{t-1})\cup \bigcup_{\gamma'\in C_{x^{t-1}}(\gamma^{t-1})} T_{x^{t-1}}(\gamma')$, and these are precisely the gliders~$\gamma'\in\Gamma(x)$ for which $[\gamma']$ gets trapped by~$[\gamma]$ in step~$t$, i.e., we have $\Delta c_{[\gamma'],[\gamma]}^t=\frac{1}{2}$ by the first case in~\eqref{eq:dct} and therefore $v(\gamma')=2v(\gamma')\Delta c_{[\gamma'],[\gamma]}^t$ for such gliders.
Now consider the second summation on the right hand side of~\eqref{eq:Delta-s2}.
By Lemma~\ref{lem:move-trapped}~(i), none of the gliders in~$T_{x^{t-1}}(\gamma^{t-1})$ are trapped by~$\gamma^t$ in~$x^t$, so these are precisely the gliders~$\gamma'\in\Gamma(x)$ for which $[\gamma']$ gets released by~$[\gamma]$ in step~$t$, i.e., we have $\Delta c_{[\gamma'],[\gamma]}^t=\frac{1}{2}$ by the second case in~\eqref{eq:dct} and therefore $v(\gamma')=2v(\gamma')\Delta c_{[\gamma'],[\gamma]}^t$.

This shows that the right hand side of~\eqref{eq:Delta-s2} equals \eqref{eq:Delta-s1} (recall that the second summation in~\eqref{eq:Delta-s1} equals~0), which completes the proof for gliders~$\gamma^{t-1}\in M(x^{t-1})$.

It remains to consider a glider~$\gamma^{t-1}\notin M(x^{t-1})$, i.e., we have $\gamma^t=\gamma^{t-1}$, so $\Delta(s(\gamma^t))=0$.
By Lemma~\ref{lem:move-trapped}~(ii) the first summation on the right hand side of~\eqref{eq:Delta-s1} equals~0.
Furthermore, by Lemma~\ref{lem:move-trapped} there is a unique glider~$\gamma'\in \Gamma(x)$ such that $\gamma^{t-1}$ is trapped by $\gamma'^{t-1}$ and $\gamma^t=\gamma^{t-1}$ is not trapped by~$\gamma'^t$, or $\gamma^{t-1}$ is not trapped by~$\gamma'^{t-1}$ and $\gamma^t$ is trapped by~$\gamma'^t$.
It follows that we have $\Delta c_{[\gamma],[\gamma']}^t=\frac{1}{2}$ for this unique glider~$\gamma'$ and therefore \eqref{eq:Delta-s1} evaluates to $\Delta s(\gamma^t)=v(\gamma)-2v(\gamma)\Delta c_{[\gamma],[\gamma']}^t=v(\gamma)-2v(\gamma)\frac{1}{2}=0$, as it should.
\end{proof}

\subsection{Analyzing the equations}
\label{sec:analysis}

In this section, we analyze the equations of motion derived in the previous section, by showing that the corresponding coefficient matrix is non-singular.

For any bitstring~$x\in X_{n,k}$, we consider the cycle~$C(x)$ defined in~\eqref{eq:Cx}.
As the length of the cycle is finite, there is a minimum number~$T>0$, such that $x^T=x^0$ and within $T$ time steps from~$x^0$ to~$x^T$ every glider~$\gamma\in\Gamma(x)$ has traveled an integral multiple of~$n$ many steps. \marginpar{$T$}
In other words, after $T$ steps, in~$x^T$ a glider $\gamma'$ from the equivalence class~$[\gamma^T]$ satisfies~$\gamma'=\gamma^0$.
Note that~$T$ is not necessarily the length of the cycle, but a multiple of it.
For example, for $x=100100\in X_{6,2}$ we have $C(x)=(100100,010010,001001)$, i.e., the cycle has length 3, but we have $T=6$, as it takes 6 time steps until the gliders of speed~1 have traveled $n=6$ steps.

We consider the equivalence classes of gliders~$\Gamma(x)/{\sim}$, and we label one representative from each class by~$\gamma_i$, $i=1,\ldots,\nu$, \marginpar{$\gamma_i,v_i$} $\nu=\nu(x)$, in non-decreasing order of their speeds, i.e., for $v_i:=v(\gamma_i)$ we have $v_1\leq v_2\leq \cdots\leq v_\nu$. 

The definition of~$T$ gives rise to the system of equations
\begin{equation}
\label{eq:Delta-s-T}
s(\gamma_i^T)-s(\gamma_i^0)=c_i n, \quad \text{for } i=1,\ldots,\nu,
\end{equation}
where the non-negative integers~$c_i$, $i=1,\ldots,\nu$, \marginpar{$c_i$} specify how many multiples of~$n$ the glider~$\gamma^i$ travels between time~$0$ to time~$T$.
We define \marginpar{$c_{i,j}$}
\begin{equation}
\label{eq:cij}
c_{i,j}:=c_j-c_i
\end{equation}
for indices $1\leq i\leq j\leq \nu$.

\begin{lemma}
We have
\begin{equation}
\label{eq:c1cnu}
c_1\leq c_2\leq \cdots\leq c_\nu,
\end{equation}
with equalities $c_i=c_{i+1}$ if $v_i=v_{i+1}$, and therefore $c_{i,j}\geq 0$.
Furthermore, we have
\begin{equation}
\label{eq:cijgamma}
c_{i,j}=c_{[\gamma_i],[\gamma_j]}^T.
\end{equation}
\end{lemma}

\begin{proof}
For any $\gamma\in\Gamma(x)$ let $\thetatilde_\gamma$ be the function on the set of integers~$[0,T]$ defined by $t\mapsto s_2(\gamma^t)$ for all $t\in[0,T]$, and let $\theta_\gamma$ be the function on the set of real numbers from~$0$ to~$T$ obtained from $\thetatilde_\gamma$ by piecewise linear interpolation.
Furthermore, we let $\Theta_{[\gamma]}$ be the union of the graphs of~$\theta_{\gamma'}$ for all $\gamma'\in[\gamma]$.
An intersection point of the graphs of~$\theta_\gamma$ and~$\theta_{\gamma'}$ between two arguments~$t-1,t\in[0,T]$ is called \emph{positive} if the function~$\theta_\gamma-\theta_{\gamma'}$ changes its sign from negative to positive, formally $\theta_\gamma(t-1)=s_2(\gamma^{t-1})<\theta_{\gamma'}(t-1)=s_2(\gamma'^{t-1})$ and $\theta_\gamma(t)=s_2(\gamma^t)>\theta_{\gamma'}(t)=s_2(\gamma'^t)$.
By Lemma~\ref{lem:g-welldef}~(iii) the quantity $c_{i,j}=c_j-c_i$ for indices $1\leq i\leq j\leq \nu$ equals the number of intersection points of $\theta_{\gamma_j}$ and $\Theta_{[\gamma_i]}$, and as $v_i\leq v_j$ all intersection points are positive, implying that $c_i\leq c_j$, which proves~\eqref{eq:c1cnu}.
Lemma~\ref{lem:g-welldef}~(ii) shows that if $v_i=v_j$ then there are no intersection points of $\theta_{\gamma_j}$ and $\Theta_{[\gamma_i]}$ and therefore $c_i=c_j$.
This completes the proof of the first part of the lemma.
Furthermore, using~\eqref{eq:ct} and the second part of Lemma~\ref{lem:move-trapped}~(i), the number of intersection points of $\theta_{\gamma_j}$ and $\Theta_{[\gamma_i]}$ equals~$c_{[\gamma_i],[\gamma_j]}^T$, which proves~\eqref{eq:cijgamma}.
\end{proof}

The definition~\eqref{eq:cij} directly gives the two relations
\begin{equation}
\label{eq:c1ij}
\begin{split}
c_i&=c_1+c_{1,i}, \\
c_{i,j}&=c_{1,j}-c_{1,i},
\end{split}
\end{equation}
valid for all $1\leq i\leq j\leq \nu$.
The equation of motion~\eqref{eq:svt} guaranteed by Proposition~\ref{prop:motion} for the glider~$\gamma_i$ at time~$t=T$ combined with~\eqref{eq:Delta-s-T} and~\eqref{eq:cijgamma} reads
\begin{equation}
\label{eq:cij-T}
v_iT+\sum_{1\leq j<i} 2v_j c_{j,i}-\sum_{i<j\leq \nu} 2v_i c_{i,j}=c_i n, \quad \text{for } i=1,\ldots,\nu.
\end{equation}
Indeed, the first summation in~\eqref{eq:svt} only gives nonzero contributions for gliders~$[\gamma']=[\gamma_j]\in\Gamma(x)/{\sim}$ with $v(\gamma_j)<v(\gamma_i)$, which translates into the summation range $1\leq j<i$, and the second summation only gives nonzero contributions for gliders~$[\gamma']=[\gamma_j]\in\Gamma(x)/{\sim}$ with $v(\gamma_j)>v(\gamma_i)$, which translates into the summation range $i<j\leq \nu$.
Eliminating from~\eqref{eq:cij-T} all coefficients~$c_i$ and~$c_{i,j}$ with $i>1$ with the help of~\eqref{eq:c1ij} we obtain the linear system
\begin{equation}
\label{eq:M}
\begin{bmatrix}
1 \\ 1 \\ 1 \\ 1 \\ \vdots \\ 1
\end{bmatrix}
c_1n
=
\underbrace{
\begin{bmatrix}
v_1 & -2v_1 & -2v_1 & -2v_1 & \cdots & -2v_1 \\
v_2 & \scriptstyle V_2-2v_2-n & -2v_2 & -2v_2 & \cdots & -2v_2 \\
v_3 & -2v_2 & \scriptstyle V_3-2v_3-n & -2v_3 & \cdots & -2v_3 \\
v_4 & -2v_2 & -2v_3 & \scriptstyle V_4-2v_4-n & \cdots & -2v_4 \\
\vdots & \vdots & \vdots & \vdots & \ddots & \vdots \\
v_\nu & -2v_2 & -2v_3 & -2v_4 & \cdots & \scriptstyle V_\nu-2v_\nu-n
\end{bmatrix}
}_{=:M}
\begin{bmatrix}
T \\ c_{1,2} \\ c_{1,3} \\ c_{1,4} \\ \vdots \\ c_{1,\nu}
\end{bmatrix},
\end{equation}
where
\begin{equation*}
V_i:=\sum_{j=1}^\nu 2v_{\min\{i,j\}}, \quad \text{for } i=2,\ldots,\nu.
\end{equation*}
These are $\nu$ homogeneous equations in the $\nu+1$ integral non-negative unknowns $T,c_1,c_{1,2},\ldots,c_{1,\nu}$.

\begin{lemma}
\label{lem:detM}
The coefficient matrix~$M$ of the linear system~\eqref{eq:M} satisfies
\begin{equation}
\label{eq:detM}
\det M=(-1)^{\nu-1}v_1\prod_{i=2}^\nu (n-V_i)\neq 0.
\end{equation}
\end{lemma}

\begin{proof}
Adding twice the first column of~$M$ to every other column creates a matrix that has 0s above the diagonal, and with $i$th diagonal entry equal to $-(n-V_i)$ for all $i=2,\ldots,\nu$.
Multiplying those values on the diagonal yields the claimed determinant.
Note that $V_i\leq \sum_{j=1}^\nu 2v_i=2k$ (recall~\eqref{eq:sum-speeds}), and therefore $n-V_i\geq n-2k>0$, implying that the product in~\eqref{eq:detM} is nonzero.
\end{proof}

With similar tricks we could evaluate all cofactors of~$M$, and thus obtain an explicit solution of the system~\eqref{eq:M}.
However, this is not needed for the rest of this paper, and so we omit these calculations.

\subsection{Every glider moves eventually}

As mentioned before, gliders with the minimum speed play an important role in our arguments.
Clearly, a glider of the maximum speed moves in every time step, as it cannot be trapped by any other glider (recall Lemma~\ref{lem:child-props}).
On the other hand, a glider of the minimum speed may be trapped, and hence does not move, for several consecutive time steps by various faster gliders; see Figures~\ref{fig:open} and~\ref{fig:time}.
In fact, it is easy to construct examples where a glider is trapped for an arbitrarily long interval of time (by adding gliders of increasing speed on the left, which increases $n$ and~$k$).
Nevertheless, the next lemma asserts that gliders of the minimum speed cannot be trapped indefinitely, but eventually will move forward.
This property is true more generally for gliders of any speed.

\begin{lemma}
\label{lem:forward}
For any glider~$\gamma\in\Gamma(x)$, there is a $t>0$ such that $s(\gamma^t)-s(\gamma^0)>0$.
\end{lemma}

Maybe surprisingly, we do not have a purely combinatorial proof for this lemma.
Instead, our proof relies on the earlier determinant computation.

\begin{proof}
If $s(\gamma^t)-s(\gamma^0)=0$ for all $t>0$, then in particular $s(\gamma^T)-s(\gamma^0)=0$, i.e., by~\eqref{eq:Delta-s-T} and the monotonicity~\eqref{eq:c1cnu} we would have $c_1=0$.
However, if $c_1=0$, then the left-hand side of~\eqref{eq:M} is the 0-vector, so by Lemma~\ref{lem:detM} the linear system~\eqref{eq:M} only has the trivial solution $T=0$ and $c_{1,2}=c_{1,3}=\cdots=c_{1,\nu}=0$.
\end{proof}

We consider a bitstring~$x\in X_{n,k}$ under parenthesis matching, and we say that a pair of matched bits is \emph{visible}, \marginpar{visible} if there is no pair of matched bits around it.
For example, the visible pairs of matched bits in the bitstring from Figure~\ref{fig:matching} are highlighted by shading in Figure~\ref{fig:visible}.

\begin{figure}[h!]
\includegraphics[page=2]{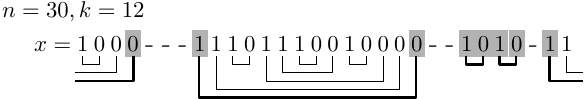}
\caption{Visible pairs of matched bits highlighted in the bitstring~$x$ from Figure~\ref{fig:matching}.}
\label{fig:visible}
\end{figure}

The next lemma asserts that for any glider~$\gamma\in\Gamma(x)$ of the minimum speed that is rightmost in its train, while moving along the cycle~$C(x)$ this glider visits every possible position~$i\in[n]$, both with its 0-bits and with its 1-bits, and moreover it is open in those steps, which implies that its outermost pair of matched bits is visible.
Recall that the positions of steps of~$\gamma$ in the infinite Motzkin path~$\wh{x}$ translate to positions of bits in the finite string~$x$ by considering them modulo~$n$, with $1,\ldots,n$ as representatives of the equivalence classes.
In this sense each step of~$\gamma$ corresponds to a bit in~$x$, specifically every $\ustep$-step corresponds to a matched~1 and every $\dstep$-step corresponds to a matched~0.
This lemma is illustrated in Figure~\ref{fig:open}.

\begin{lemma}
\label{lem:flip-vertex}
Let~$x\in X_{n,k}$ and let $\gamma\in\Gamma(x)$ be the rightmost glider in a train of the minimum speed~$v(\gamma)=\min V(x)$.
Then for $b\in\{0,1\}$ and for any position~$i\in [n]$ there is a $t>0$ such that~$\gamma^t$ is open and one of its $b$-bits is at position~$i$ in~$x^t$.
\end{lemma}

\begin{proof}
By Lemma~\ref{lem:clean}, we know that $\gamma^t=:(A^t,B^t)$ is clean for all $t\geq 0$, i.e., we have $\wh{x^t}_{r(\gamma^t)}=1^{v(\gamma)}0^{v(\gamma)}$ if $\gamma^t$ is non-inverted and $\wh{x^t}_{r(\gamma^t)}=0^{v(\gamma)}1^{v(\gamma)}$ if $\gamma^t$ is inverted.

We consider two cases.
Case~(a): If $s(\gamma^t)=s(\gamma^{t-1})$, then we have $\gamma^t=\gamma^{t-1}$ and exactly one of $\gamma^{t-1}$ and~$\gamma^t$ is inverted and the other is non-inverted.
Case~(b): If $s(\gamma^t)>s(\gamma^{t-1})$, then we have $\gamma^{t-1}\in M(x^{t-1})$ and therefore $\gamma^{t-1}\in\Gammaf(x^{t-1})$, and also $\gamma^t\in\Gammaf(x^t)$ by Lemma~\ref{lem:move-trapped}~(i), i.e., $\gamma^{t-1}$ and~$\gamma^t$ are both free and hence non-inverted, which means that the $2v(\gamma)$ bits of~$\gamma^t$ are obtained by cyclically shifting the $2v(\gamma)$ bits of~$\gamma^{t-1}$ to the right by $v(\gamma)$ positions (from~\eqref{eq:g1} we see that $A^t=B^{t-1}=A^{t-1}+v(\gamma)$).
Also, $\gamma^t$ is open by Lemma~\ref{lem:move-open}.

By Lemma~\ref{lem:forward}, any time step~$t-1\rightarrow t$ as in case~(a) must eventually be followed by a step as in case~(b), and this proves the lemma.
\end{proof}

\section{Gluing the cycles together}
\label{sec:gluing}

In this section we glue the cycles from the factor~$\cC_{n,k}$ together via 4-cycles, as outlined in Section~\ref{sec:idea-gluing}.
The main result of this section, Theorem~\ref{thm:connect} below, asserts that this is possible without any conflicts between 4-cycles, and such that the gluing results in a single Hamilton cycle.
The proof of this theorem is the first and only time in this paper that the assumption $n\geq 2k+3$ is used.
With Theorem~\ref{thm:connect} in hand, we prove Theorem~\ref{thm:Knk} at the end of this section.

\subsection{Connectors}

For a bitstring~$x\in X_{n,k}$ and integer~$i\geq 0$, we write $\sigma^i(x)$ \marginpar{$\sigma^i(x)$} for the string obtained from~$x$ by cyclic right-shift by~$i$ positions.

If two bitstrings~$x,y\in X_{n,k}$ differ in an exchange of one visible pair of matched bits, then we refer to~$\{x,y\}$ as a \emph{connector}. \marginpar{connector $\{x,y\}$}
In other words, $x$ and~$y$ have the form
\begin{equation}
\label{eq:conn}
\begin{bmatrix}
x \\ y
\end{bmatrix}
=\sigma^i\left(
\begin{bmatrix}
\cdots\vis{1}\,u\,\vis{0}\cdots \,\hyph\,w\,\hyph\cdots \\
\cdots\,\hyph\,u\,\hyph\cdots \vis{1}\,w\,\vis{0}\cdots
\end{bmatrix}\right)
\end{equation}
for some $i\geq 0$, where $u,w\in D$ and the shaded pair of matched bits is visible, i.e., $x$ and~$y$ differ by transposing a visible pair of matched bits with two unmatched~0s that have no other unmatched~0s nor the visible pair in between them.
In other words, the transposed unmatched~0s are adjacent to a block other than the one containing the transposed pair of matched bits.
We write~$\cX_{n,k}$ \marginpar{$\cX_{n,k}$} for the set of all connectors.

\begin{lemma}
\label{lem:connector}
For any connector~$\{x,y\}\in \cX_{n,k}$, the sequence $C_4(x,y):=(x,f(x),y,f(y))$ is a 4-cycle in the Kneser graph~$K(n,k)$.
\end{lemma}
\marginpar{$C_4(x,y)$}

\begin{proof}
Recall from Section~\ref{sec:factor} that $f$ complements all matched bits.
We already know that $(x,f(x))$ and $(y,f(y))$ are edges in the Kneser graph~$K(n,k)$.
Therefore, it remains to show that $(x,f(y))$ and $(y,f(x))$ are also edges.
By symmetry, it suffices to argue that $(x,f(y))$ is an edge in~$K(n,k)$.
From the definition of~$f$ we know that the positions of~1s in~$f(y)$ are $\mu_1(f(y))=\mu_0(y)$.
Let $x$ and~$y$ be as in~\eqref{eq:conn}, and let $a$ and $b$ be the positions of the bits directly to the right of~$u$ and~$w$, respectively.
Clearly, we have $\mu_0(y)=\mu_0(x)\setminus\{a\}\cup\{b\}$, i.e., we have $\mu_1(f(y))=\mu_0(x)\setminus\{a\}\cup\{b\}$.
As $x_b=\hyph$ is an unmatched~0, we see that the positions of 1s in~$f(y)$ are where $x$ has either matched or unmatched~0s, which proves that $(x,f(y))$ is indeed an edge in~$K(n,k)$.
\end{proof}

Observe that if $\{x,y\}\in \cX_{n,k}$ and $C(x)$ and~$C(y)$ are two distinct cycles in the factor~$\cC_{n,k}$ defined in~\eqref{eq:factor}, then the symmetric difference of the edge sets~$\big(C(x)\cup C(y)\big)\Delta C_4(x,y)$ is a single cycle in~$K(n,k)$ on the same vertex set as~$C(x)\cup C(y)$, i.e., the 4-cycle `glues' two cycles from the factor together to a single cycle.

\begin{figure}
\makebox[0cm]{ 
\includegraphics[page=4]{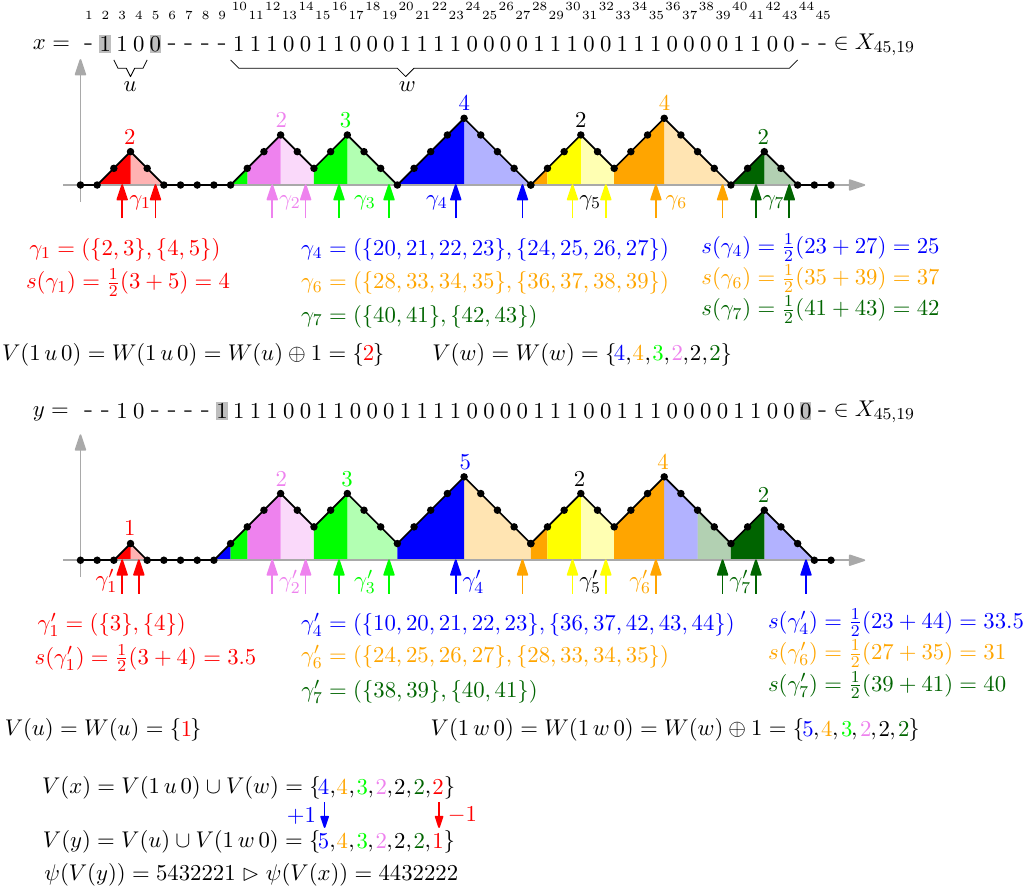}
}
\caption{Relation between gliders and speed sets~$V(x)$ and~$V(y)$ for a connector~$\{x,y\}$.
Note that the steps of the base hill~$w$ in~$x$ that belong to the gliders~$\gamma_4,\gamma_6,\gamma_7$ get redistributed into the gliders~$\gamma_4',\gamma_6',\gamma_7'$ in the base hill~$1\,w\,0$ in~$y$, and so there is no apparent correspondence between gliders in~$\Gamma(w)$ and gliders in~$\Gamma(1\,w\,0)$.
Nonetheless, the speed sets satisfy~$V(1\,w\,0)=V(w)\oplus 1$, i.e., a largest speed value is incremented by~1.}
\label{fig:speed-conn}
\end{figure}

Given a connector~$\{x,y\}$ as in~\eqref{eq:conn}, we now aim to understand the relation between the corresponding speed sets~$V(x)$ and~$V(y)$; see Figure~\ref{fig:speed-conn}.
From~\eqref{eq:conn} we see that
\begin{equation*}
V(x)=V(1\,u\,0)\cup V(w)\cup S
\end{equation*}
for some set~$S$ (the set~$S$ contains the speeds of gliders outside of~$1\,u\,0$ and~$w$).
From the definition~\eqref{eq:Wx} we obtain directly that $W(1\,u\,0)=W(u)\oplus 1$.
Using Lemma~\ref{lem:VW} we therefore have
\begin{subequations}
\label{eq:Vxy}
\begin{equation}
V(x)=\big(V(u)\oplus 1\big)\cup V(w)\cup S.
\end{equation}
With analogous arguments we obtain
\begin{equation}
V(y)=V(u)\cup \big(V(w)\oplus 1\big)\cup S.
\end{equation}
\end{subequations}
In words, $V(y)$ is obtained from~$V(x)$ by decrementing the speed of the fastest glider whose bits belong to the substring~$1\,u\,0$ by~1, and incrementing the speed of a fastest glider whose bits belong to the substring~$w$ by~1, whereas all other glider speeds remain unchanged.
We emphasize that~\eqref{eq:Vxy} is only a statement about the speed sets, not about the glider partition and the resulting glider positions in~$x$ and~$y$, which may in fact change considerably; see Figure~\ref{fig:speed-conn}.

In the following, we write~$\psi(V(x))$ \marginpar{$\psi(V(x))$} for the sequence obtained from the multiset~$V(x)$ by sorting its elements in non-increasing order.
As in Section~\ref{sec:analysis} before, we label the elements of~$V(x)$ in non-decreasing order as $v_1\leq v_2\leq \cdots\leq v_\nu$, $\nu=\nu(x)$, so we have $\psi(V(x))=(v_\nu,v_{\nu-1},\ldots,v_2,v_1)$.
In this way we can think of the sequence~$\psi(V(x))$ as a number partition of~$k$ (recall~\eqref{eq:sum-speeds}).
For two number partitions~$p$ and~$q$ of~$k$, we write $p\trr q$ \marginpar{$\trr$} if $p$ is lexicographically strictly larger than~$q$.
This defines a total order on all number partitions of~$k$.
Observe that even if $p$ and~$q$ do not have the same number of parts (i.e., different length), they will differ in a prefix of the same length in both, as they are partitions of the same number~$k$.

Going back to~\eqref{eq:Vxy}, note that if~$1\,u\,0=1^{v(\gamma)}0^{v(\gamma)}$ are the bits belonging to a glider~$\gamma\in\Gamma(x)$ of the minimum speed~$v(\gamma)=\min V(x)$, then the multiset~$V(y)$ is obtained from~$V(x)$ by decrementing an element of the minimum value~$v(\gamma)$ by~1, and incrementing another element by~1.
When considering the corresponding Young diagrams of the associated number partitions, this corresponds to one square moving from the smallest (rightmost) stack to a stack further to the left.
As a consequence, we have $\psi(V(y))\trr \psi(V(x))$.
We will use this observation to control the process of gluing together cycles from our factor via connectors.
Specifically, by using connectors that lexicographically increase the number partitions associated with the speed set of gliders, we ensure that every cycle is joined to a cycle~$C(x)$ that has the lexicographically largest partition~$\psi(V(x))=k$, i.e., a single glider of speed~$k$.

Furthermore, for any number partition~$p$ of~$k$ and for any integer~$i\geq 0$ we write $p\boxplus i$ for the $i$th smallest number partition larger than~$p$ in the total lexicographic order~$\trr$ of all number partitions of~$k$.
Similarly, we write $p\boxminus i$ for the $i$th largest number partition smaller than~$p$.
For example, all number partitions of~$k=5$ in lexicographic order are $11111,2111,221,311,32,41,5$, so we have $221\boxplus 1=311$, and $32\boxminus 3=2111$, $2111\boxplus 5\trr 41$, and $11111\boxplus 2\trr 32\boxminus 3$.
In this and the following examples of single-digit number partitions, we omit brackets and commas for brevity.

We state the following simple observations for further reference.

\begin{lemma}
\label{lem:partition}
Let $v_\nu\geq \cdots\geq v_1\geq 1$ be a number partition of~$k$ with $\nu\geq 2$ parts.
Then the multiset $V:=\{v_\nu,\ldots,v_1\}$ satisfies the following properties:
\begin{enumerate}[label=(\roman*),leftmargin=8mm]
\item
For any $\nu\geq j>i\geq 1$ we have $\psi\big(V\setminus\{v_j,v_i\}\cup\{v_j+1\}\cup(\{v_i-1\}\setminus\{0\})\big)\trr \psi(V)$.
\item
If $\nu\geq 3$ and $v_3>v_2$ then for any $\nu\geq j\geq 3$ and $i\in\{2,1\}$ we have $\psi\big(V\setminus\{v_j,v_i\}\cup \{v_j+1\}\cup(\{v_i-1\}\setminus\{0\})\big)\trr \psi(V)\boxplus 1$.
\item
If $v_1\geq 2$, then $\psi(V\setminus\{v_1\}\cup\{v_1-1,1\})=\psi(V)\boxminus 1$.
\end{enumerate}
\end{lemma}

The expression $\{v_i-1\}\setminus\{0\}$ evaluates to~$\{v_i-1\}$ if $v_i>1$ and to~$\emptyset$ if $v_i=1$, and therefore the set $V\setminus\{v_j,v_i\}\cup\{v_j+1\}\cup(\{v_i-1\}\setminus\{0\})$ is obtained from~$V$ by decrementing~$v_i$ by~1 and removing it if~$v_i-1=0$, and incrementing~$v_j$ by~1.

\subsection{Gluing the single glider cycles}

Our proof will join most cycles from the factor~$\cC_{n,k}$ via the aforementioned connectors.
However, the cycles which contain the $n$ vertices
\marginpar{$s_i$}
\begin{equation}
\label{eq:si}
s_i:=\sigma^i(1^k\,0^k\,\hyph^\ell),
\end{equation}
$\ell:=n-2k$, for $i=0,\ldots,n-1$ will first be joined in a slightly different way.
These are precisely the vertices with $V(s_i)=\{k\}$, i.e., they have only a single glider of the maximum speed~$k$.
Specifically, defining $g:=\gcd(n,k)$ \marginpar{$g$} these cycles from~$\cC_{n,k}$ are $C(s_i)=(s_{i+kj})_{j=0,\ldots,n/g-1}$ for $i=0,\ldots,g-1$.
Consequently, the $n$ vertices $s_i$, $i=0,\ldots,n-1$, are partitioned into $g$ cycles, each containing $n/g$ of those vertices.
In particular, if $g=1$, then this is only a single cycle, and no joining is needed.
We define
\marginpar{$\cD$}
\begin{equation}
\label{eq:D}
\cD:=\{C(s_i)\mid i=0,\ldots,g-1\}.
\end{equation}

Note that the subgraph of~$K(n,k)$ induced by the vertices~$s_i$, $i=0,\ldots,n-1$, is isomorphic to the Cayley graph of~$\mathbb{Z}/n\mathbb{Z}$ with generators $\{k+i\mid i=0,\ldots,\ell\}$.
The following lemma is a simple consequence of that.

\begin{lemma}
\label{lem:join}
Let $k\geq 1$ and $n\geq 2k+1$.
For any $r\in\{0,\ldots,n-1\}$ the symmetric difference of the edge sets of the cycles~$\cD$ with the union of the 4-cycles $C_4'(s_{r+j},s_{r+j+k+1}):=(s_{r+j},s_{r+j+k},s_{r+j+2k+1},s_{r+j+k+1})$ for $j=0,\ldots,g-2$ is a single cycle in~$K(n,k)$ on the same vertex set.
\end{lemma}
\marginpar{$C_4'(x,y)$}

As expected, Lemma~\ref{lem:join} is trivial for~$g=1$, as in this case $\cD$ only contains a single cycle and so no 4-cycles are needed for the joining.
For $g\geq 2$, the edges that the 4-cycles~$C_4'(s_{r+j},s_{r+j+k+1})$ defined in Lemma~\ref{lem:join} share with the cycles from~$\cD$ are precisely the edges $(s_i,f(s_i))$ with $s_i$ in
\marginpar{$\cS_r$}
\begin{equation}
\label{eq:Sr}
\cS_r:=\big\{\{s_{r+j},s_{r+j+k+1}\}\mid j=0,\ldots,g-2\big\}.
\end{equation}

The pairs in~$\cS_r$ are not connectors in the sense of our definition~\eqref{eq:conn}.
However, they serve the same purpose of describing pairs of edges on the cycle factor~$\cC_{n,k}$ that lie on a common 4-cycle (specifically, each edge is described by one endpoint~$x$, and the other endpoint is~$f(x)$).
Note that the 4-cycle $C_4(x,y)$ defined in Lemma~\ref{lem:connector} for a connector~$\{x,y\}\in\cX_{n,k}$ has the form~$C_4(x,y)=(x,f(x),y,f(y))$, whereas for a pair~$\{x,y\}\in\cS_r$ the 4-cycle $C_4'(x,y)$ defined in Lemma~\ref{lem:join} has the form~$C_4'(x,y)=(x,f(x),f(y),y)$.

\subsection{Auxiliary graph}

Our strategy is to join the cycles of the factor~$\cC_{n,k}$ by repeatedly gluing pairs of them together via connectors, as described before.
This is modeled by the following auxiliary graph.
For any set of connectors~$\cU\seq\cX_{n,k}$ we define a graph~$\cH_{n,k}[\cU]$ \marginpar{$\cH_{n,k}[\cU]$} as follows:
The nodes of~$\cH_{n,k}[\cU]$ are the cycles of the factor~$\cC_{n,k}\setminus\cD$, plus the set~$\cD$ defined in~\eqref{eq:D}, which forms its own single node.
For any connector~$\{x,y\}\in\cU$ such that $C(x),C(y)\in \cC_{n,k}\setminus\cD$ are distinct cycles we add an edge that connects~$C(x)$ and~$C(y)$ to the graph~$\cH_{n,k}[\cU]$.
Furthermore, for any connector~$\{x,y\}\in\cU$ such that $C(x)\in\cC_{n,k}\setminus \cD$ and~$C(y)\in \cD$ we add an edge that connects~$C(x)$ and~$\cD$ to the graph~$\cH_{n,k}[\cU]$.
The graph~$\cH_{n,k}[\cU]$ does not have multiple edges, even though there may be several connectors for the same two cycles~$C(x)$ and~$C(y)$.

Note that every bitstring~$x\in\cX_{n,k}$ is contained in many different connectors.
Specifically, let $p\geq 1$ denote the number of visible pairs of matched bits in~$x$.
There are $\ell:=n-2k$ unmatched~0s in~$x$, i.e., $\ell$ pairs of substrings of the form~$\hyph w\hyph$ with $w\in D$, possibly $w=\varepsilon$.
Each visible pair is contained in precisely one of those substrings, and not contained in the remaining $\ell-1$ of them.
Consequently, $x$ is contained in exactly $p(\ell-1)$ connectors.
For example, the bitstring in Figure~\ref{fig:visible} has $p=4$ visible pairs and $\ell=30-2\cdot 12=6$ unmatched~0s, i.e., it is contained in exactly~20 connectors.

Exploiting this observation, it is relatively easy to show that the graph~$\cH_{n,k}[\cX_{n,k}]$ is connected, i.e., that it has a spanning tree, and each of the connectors corresponding to one edge of this spanning tree glues together two cycles from the factor~$\cC_{n,k}$.
However, to obtain a Hamilton cycle in~$K(n,k)$, it is crucial that no two gluing operations interfere with each other, i.e., that any two of the 4-cycles used for the gluing are edge-disjoint.
For this it is enough to guarantee that $C_4(x,y)$ and~$C_4(x',y')$ do not have an edge in common that lies on one of the cycles of the factor~$\cC_{n,k}$, which is equivalent to $\{x,y\}\cap \{x',y'\}=\emptyset$.
We therefore require a set of \emph{pairwise disjoint} connectors~$\cU\seq\cX_{n,k}$ such that~$\cH_{n,k}[\cU]$ is a connected graph.
We remark that the condition $\{x,y\}\cap \{x',y'\}=\emptyset$ does not rule out that $C_4(x,y)$ and~$C_4(x',y')$ have an edge in common that goes between the same two cycles $\{C(x),C(y)\}=\{C(x'),C(y')\}$ of the factor, but then at most one of the two 4-cycles will be used in the gluing process (there is no need to glue the same two cycles of the factor together more than once), so this is not an issue.

\begin{theorem}
\label{thm:connect}
For any~$k\geq 1$ and~$n\geq 2k+3$, there is a set~$\cU\seq\cX_{n,k}$ of connectors and an~$r\in \{0,\ldots,n-1\}$ such that~$\cU$ and~$\cS_r$ defined in~\eqref{eq:Sr} satisfy the following: the sets in~$\cU$ are pairwise disjoint, the sets in~$\cU$ and~$\cS_r$ are pairwise disjoint, and $\cH_{n,k}[\cU]$ is a connected graph.
\end{theorem}

Before providing the formal proof of Theorem~\ref{thm:connect}, we first give an informal sketch that highlights the crucial ingredients and obstacles, following the outline given in Section~\ref{sec:idea-gluing}.

\textit{Proof sketch.}
We start by fixing a position~$p\in[n]$ arbitrarily.
For any vertex~$x\in X_{n,k}$ with at least two gliders, i.e., $\nu(x)\geq 2$, we consider its speed set~$V(x)$ and the rightmost glider~$\gamma\in\Gamma(x)$ in a train of the minimum speed~$v(\gamma)=\min V(x)=v_1$.
We follow the vertices along the cycle~$C(x)$, tracking the movement of the glider~$\gamma$, until we reach a vertex~$\chi(x,\gamma)$ in which the glider~$\gamma$ has moved to the special position~$p$ and is open, which is possible by Lemma~\ref{lem:flip-vertex}.
Specifically, if $v_1$ is even/odd then we let $\chi(x,\gamma)$ be such that $\gamma$ has one of its 0-bits/1-bits at position~$p$, respectively, and we write $X\seq X_{n,k}$ for the set of all such vertices~$\chi(x,\gamma)$; see Figure~\ref{fig:connect-sketch}~(a).
For any $x\in X$ we add a connector $\{x,\alpha(x)\}$ to~$\cU$, where the mapping~$\alpha(x)$ is defined as follows (recall \eqref{eq:conn}):
As the glider~$\gamma$ at position~$p$ in~$x$ is open, the bit to the right of~$\gamma$ is an unmatched~0, and furthermore the outermost pair of matched bits of~$\gamma$ is visible.
We let $w$ be the first block in~$x$ to the right of~$\gamma$, and we let $\alpha(x)$ be the vertex obtained from~$x$ by removing the visible pair of matched bits of~$\gamma$ and adding a new pair of matched bits around~$w$.
We write $Y:=\alpha(X)$ for the image of~$X$ under the mapping~$\alpha$.

To check that the connectors in~$\cU:=\{\{x,\alpha(x)\}\mid x\in X\}$ are pairwise disjoint, we argue as follows:
Given any vertex~$x\in X\cup Y$, we can uniquely reconstruct to which of the two sets~$X$ or~$Y$ it belongs and the corresponding image or preimage under~$\alpha$, based on the speed of the glider at position~$p$ (even or odd) and which bits it has at position~$p$ (0 or 1).
For example, if the glider~$\gamma$ at position~$p$ has even speed and a 1-bit at this position, then we have $x\in Y$, and the corresponding preimage under~$\alpha$ is obtained by removing the first visible pair of matched bits to the right of~$\gamma$ and adding a new pair of matched bits around~$\gamma$.

\begin{figure}
\includegraphics[page=2]{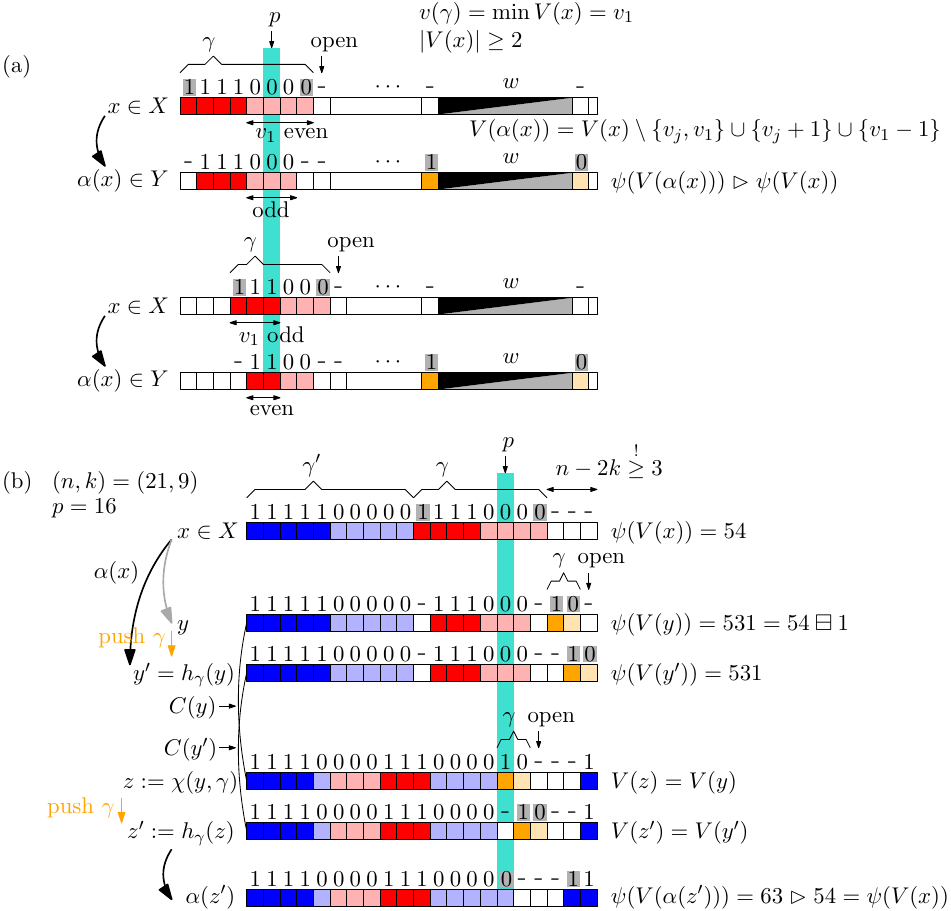}
\caption{Illustration of the proof sketch of Theorem~\ref{thm:connect}.}
\label{fig:connect-sketch}
\end{figure}

To show that the graph~$\cH_{n,k}[\cU]$ is connected, we consider the speed sets~$V(x)$ of vertices~$x\in X_{n,k}$, and the corresponding number partitions~$\psi(V(x))$ of~$k$ obtained by sorting the speeds in~$V(x)$ non-increasingly.
Note that the speed sets and hence the number partitions are invariant along each cycle of the factor~$\cC_{n,k}$ by Lemma~\ref{lem:Vx-invariant}, but they can only change along connectors.
Thus, we consider the number partitions~$\psi(V(x))$ and~$\psi(V(\alpha(x)))$ for connectors~$\{x,\alpha(x)\}$, $x\in X$, and we argue that they increase lexicographically, i.e., $\psi(V(\alpha(x)))\trr \psi(V(x))$.
This ensures that every cycle is joined, via a sequence of connectors, to a cycle in~$\cD$, which has the lexicographically largest number partition of~$k$, namely the number~$k$ itself (a single part).
By our definition of~$\alpha$ that involves a minimum speed glider~$\gamma$ at position~$p$ in~$x$ and the first block~$w$ to the right of it, we obtain from~\eqref{eq:Vxy} that the speed set~$V(\alpha(x))$ is obtained from the speed set~$V(x)$ by decrementing the smallest speed value~$v_1$ by~1, and incrementing one of the largest speed values from the gliders in~$\Gamma(w)$ by~1.
Lemma~\ref{lem:partition}~(i) thus gives $\psi(V(\alpha(x)))\trr \psi(V(x))$, as desired.
An example of such a sequence of lexicographically increasing number partitions via connectors as described before for $k=9$ is $4221\trl 432\trl 531\trl 54\trl 63\trl 72\trl 81\trl 9$.

The proof sketch presented so far has sneakily glossed over two problems with the definition of connectors~$\{x,\alpha(x)\}$ that we discuss below.
Fixing these problems requires careful adjustment of the definition of the mapping~$\alpha$, so that the two opposing objectives disjointness and connectedness are met simultaneously.
This was achieved with the help of a computer program that finds pairs of connectors~$\{x,\alpha(x)\}$ and~$\{x',\alpha(x')\}$ with $\{x,\alpha(x)\}\cap\{x',\alpha(x')\}\neq \emptyset$ arising from the current definition of~$\alpha$ for small values of~$n$ and~$k$, and then successively changing and refining the definition, so as to avoid these particular conflicts, sometimes creating new and unforeseen conflicts.
This process started with the two parity-based rules captured in Figure~\ref{fig:connect-sketch}~(a) and terminated with the regular expressions stated in~\eqref{eq:alphai} that split the definition of~$\alpha$ into nine cases~$\alpha_i$, $i\in\{1,\ldots,9\}$.
One can see that $\alpha_3$ and~$\alpha_5$ defined in~\eqref{eq:alpha3} and~\eqref{eq:alpha5} are refinements of the parity-based rules in Figure~\ref{fig:connect-sketch}~(a).

The first problem is that the aforementioned disjointness argument fails when~$\gamma$ has its first or last bit at position~$p$.
For example, if $(n,k)=(13,5)$, $p=4$, and $x=\hyph10\underline{1}0\hyph111000\hyph$ and $x'=110\underline{0}\hyph\hyph111000\hyph$, then we have a conflict $\alpha(x)=\alpha(x')=\hyph10\underline{\hyph}\hyph11110000$ (the bit at position~$p$ is underlined).
Avoiding such conflicts is achieved by the rule~$\alpha_6$ defined in~\eqref{eq:alpha6}.

The second problem is that the specification `let $w$ be the first block in~$x$ to the right of~$\gamma$' does not give a valid connector if~$x$ has only a single block, i.e., all matched bits and all unmatched bits are consecutive.
We illustrate this with a concrete example; see Figure~\ref{fig:connect-sketch}~(b).
Let $(n,k)=(21,9)$ and $p=16$, and consider the vertex $x=111110000011110\underline{0}00\hyph\hyph\hyph\in X$, which has two gliders~$\gamma$ and~$\gamma'$ of speeds~$v_1=4$ and~5, respectively, the slower one~$\gamma$ at position~$p$, so $\psi(V(x))=54$.
However, there is \emph{no} connector $\{x,y\}$ for which $\psi(V(y))\trr \psi(V(x))$.
Indeed, removing the visible pair of matched bits of~$\gamma$ and adding the matched pair at positions~$(n-2,n-1)$ or $(n-1,n)$ yields vertices~$y$ and~$y'$ with $\psi(V(y))=\psi(V(y'))=531$.
Also, removing the visible pair of bits of~$\gamma'$ and adding the pair at positions~$(n-2,n-1)$ or~$(n-1,n)$ yields vertices~$\ty$ and~$\ty'$ with $\psi(V(\ty))=\psi(V(\ty'))=441$.
Neither $531$ nor $441$ are lexicographically larger than~$54$.
However, for~$y$ and~$y'$ we have $\psi(V(y))=\psi(V(y'))=\psi(V(x))\boxminus 1$ by Lemma~\ref{lem:partition}~(iii), so the lexicographic decrease is very small.
Let $\gamma$ be the speed~1 glider in~$y$, and note that $y'$ differs from~$y$ only in pushing~$\gamma$ to the right by one position.
Formally, by Lemma~\ref{lem:open-move} and the definition of the mapping~$h_\gamma$ given in Section~\ref{sec:slow-move} we have $y'=h_\gamma(y)$.
We now track the movement of~$\gamma$ along the cycle~$C(y)$ until it reaches position~$p$, which happens in the vertex~$z:=\chi(y,\gamma)$, and there are two potential issues with~$z$.
The first is if in~$z$ the first block~$w$ to the right of~$\gamma$ equals $w=111000$, then the connector~$\{z,\alpha(z)\}$ would lead back to a vertex with the same speed set~$\psi(V(\alpha(z)))=54=\psi(V(x))$ as we started with (i.e., $w$ contains only the glider whose speed was previously decremented and is now being incremented back to its original value), leading to a potentially infinite loop of connectors with no lexicographic increase.
This issue is prevented by the rules~$\alpha_7$ and~$\alpha_8$ defined in~\eqref{eq:alpha7} and~\eqref{eq:alpha8}.
The second potential issue is that $z$ might be another single-block vertex, in which the only available connectors displace the speed~1 glider~$\gamma$ or create more speed~1 gliders, which would aggravate the lexicographic decrease.
However, in such a situation, Lemma~\ref{lem:push} comes to our rescue, as it ensures that the vertex $z':=h_\gamma(z)$ differs from~$z$ only in pushing~$\gamma$ one position to the right, which means that~$z'$ is \emph{not} a single-block vertex, so we can reach a lexicographically larger vertex~$\alpha(z')$ along a suitable connector~$\{z',\alpha(z')\}$, in our example $\psi(V(\alpha(z')))=63\trr 54=\psi(V(x))$.
Overall, the lexicographic decrease $54\trr 531$ along the first connector~$\{x,y'\}$ is compensated by a lexicographic increase $531\trl 63$ along the second connector~$\{z',\alpha(z')\}$, which is an increase $54\trl 63$ altogether.
In general, we use Lemma~\ref{lem:partition}~(ii) to show that $\psi(V(\alpha(z')))\trr \psi(V(z'))\boxplus 1$, which then gives $\psi(V(\alpha(z')))\trr \psi(V(z'))\boxplus 1=\psi(V(y'))\boxplus 1=(\psi(V(x))\boxminus 1)\boxplus 1=\psi(V(x))$.
The conditional definition $\alpha(x):=y$ or $\alpha(x):=y'$, depending on whether $z=\chi(y,\gamma)$ is a single-block vertex or not, requires that the number of unmatched bits in~$x$ is at least~3, which is equivalent to $n\geq 2k+3$, and this is why this part of the argument does not extend to the cases~$n=2k+2$ or~$n=2k+1$.
In the rules~$\alpha_2$ and~$\alpha_4$ defined in~\eqref{eq:alpha2} and~\eqref{eq:alpha4}, one sees the conditional pushing of the speed~1 glider, and in rule~$\alpha_1$ defined in~\eqref{eq:alpha1} one sees the compensating lexicographic improvement after pushing the speed~1 glider one position to the right in a single-block vertex.
\qed

\begin{proof}
We define $\ell:=n-2k$, which is the number of unmatched bits in any vertex~$x\in X_{n,k}$.
Note that by the assumption~$n\geq 2k+3$ we have $\ell\geq 3$, which will be crucial.
Throughout the proof, we fix a position~$p\in[n]$ arbitrarily.
For any~$x\in X_{n,k}$ and $i\in\mu(x)$, we write $\gamma(x,i)=(A,B)\in\Gamma(x)$ for the glider with~$i\in A\cup B$.

As before, we will denote vertices from~$X_{n,k}$ under parenthesis matching by strings of length~$n$ over the alphabet~$\{1,0,\hyph\}$.
We now introduce regular expressions to specify subsets of vertices from~$X_{n,k}$ that satisfy certain constraints.
In those expressions, the symbol~$*$ is a wildcard character that represents a string of any length, possibly zero, of the symbols~$\{1,0,\hyph\}$.
Moreover, for $b\in\{1,0,\hyph\}$ we write $b^*$ for any number of occurrences of the symbol~$b$, possibly zero.
For example, if $(n,k)=(7,3)$ then $10\,\hyph^*\,10\,*$ denotes the set of vertices~$\{101010\hyph,1010\hyph 10,10\hyph 1010\}$.
Disjunction of two regular expressions~$x$ and~$y$ is denoted as~$(x\,|\,y)$.
For example, if $(n,k)=(6,2)$ then $(1100\,|\,\hyph^2)*$ denotes the set of vertices~$\{1100\hyph\hyph,\hyph\hyph1100,\hyph\hyph1010\}$.
We write $D^+:=D\setminus\{\varepsilon\}$ for the set of all nonempty Dyck words, which represent nonempty sequences of matched pairs of bits.
There may be additional constraints on certain substrings in the expressions, captured by propositional formulas.
For example, if $(n,k)=(11,4)$, then $1^b\,0^b\,\hyph^c\,w\,*$ with the constraint formula $b\geq 2\wedge c\in\{1,2\}\wedge w\in D^+$ represents the set of vertices $\{1100\hyph1010\hyph\hyph,1100\hyph10\hyph10\hyph,1100\hyph10\hyph\hyph10,1100\hyph1100\hyph\hyph,1100\hyph\hyph1010\hyph,1100\hyph\hyph10\hyph10,1100\hyph\hyph1100\hyph,\allowbreak 111000\hyph10\hyph\hyph,111000\hyph\hyph10\hyph\}$.
Furthermore, our regular expressions contain underlined symbols to indicate strings that have to be shifted cyclically so that one of the underlined symbols is at position~$p$.
For example, if $(n,k)=(5,2)$ and $p=2$ then $110\underline{0\hyph1}$ denotes the set of vertices~$\{00\hyph 111,0\hyph 1110,\hyph 11100\}$.
As before, in a connector~$\{x,y\}$, we shade the visible pairs of matched bits in which~$x$ and~$y$ differ, which provides a visual aid to highlight the change.

We define mappings~$\alpha_i$, $i=1,\ldots,9$, whose domains and images are subsets of~$X_{n,k}$ such that $\{x,\alpha_i(x)\}$ is a connector for all $x\in\dom(\alpha_i)$ as follows:
\begin{subequations}
\label{eq:alphai}
\begin{align}
\alpha_1:\;\;
\begin{bmatrix}
x \\ \alpha_1(x)
\end{bmatrix}
&=
\begin{bmatrix}
u\,\underline{\hyph}\,\vis{1}\,\vis{0}\,\hyph^{\ell-2}\,\hyph \\
u\,\underline{\vis{0}}\,\hyph\,\hyph\,\hyph^{\ell-2}\,\vis{1} \\
\end{bmatrix},
&& u\in D^+;  \label{eq:alpha1} \\
%
\alpha_2:\;\;
\begin{bmatrix}
x \\ \alpha_2(x)
\end{bmatrix}
&=
\begin{bmatrix}
\vis{1}\,1^{a-1}\,\underline{0^{a-1}\,\vis{0}}\,\hyph\,\hyph\,\hyph\,* \hspace{25.5mm} \\
\begin{cases}
\hyph\,1^{a-1}\,\underline{0^{a-1}\,\hyph}\,\vis{1}\,\vis{0}\,\hyph\,* & \text{if } \chi(y,\gamma)\notin X_4 \\
\hyph\,1^{a-1}\,\underline{0^{a-1}\,\hyph}\,\hyph\,\vis{1}\,\vis{0}\,* & \text{otherwise}
\end{cases}
\end{bmatrix},
&& \nu(x)\geq 2\,\wedge\,a=v_1 \text{ even}; \label{eq:alpha2} \\
%
\alpha_3:\;\;
\begin{bmatrix}
x \\ \alpha_3(x)
\end{bmatrix}
&=
\begin{bmatrix}
\vis{1}\,1^{a-1}\,\underline{0^{a-1}\,\vis{0}}\,\hyph^c\,\hyph\,w\,\hyph\,* \\
\hyph\,1^{a-1}\,\underline{0^{a-1}\,\hyph}\,\hyph^c\,\vis{1}\,w\,\vis{0}\,* \\
\end{bmatrix},
&& \hspace{-2cm} a=v_1 \text{ even} \,\wedge\, c\in\{0,1\} \,\wedge\, w\in D^+; \label{eq:alpha3} \\
%
\alpha_4:\;\;
\begin{bmatrix}
x \\ \alpha_4(x)
\end{bmatrix}
&=
\begin{bmatrix}
\underline{\vis{1}\,1^{a-1}}\,0^{a-1}\,\vis{0}\,\hyph\,\hyph\,\hyph\,\hyph^c\,* \hspace{25.5mm} \\
\begin{cases}
\underline{\hyph\,1^{a-1}}\,0^{a-1}\,\hyph\,\vis{1}\,\vis{0}\,\hyph\,\hyph^c\,* & \text{if } \chi(y,\gamma)\notin X_4 \\
\underline{\hyph\,1^{a-1}}\,0^{a-1}\,\hyph\,\hyph\,\vis{1}\,\vis{0}\,\hyph^c\,* & \text{otherwise}
\end{cases}
\end{bmatrix},
&& 
\parbox[c]{5cm}{$\nu(x)\geq 2 \,\wedge\, a=v_1 \text{ odd} \,\wedge\,$ \\
$(a\geq 3 \,\vee\, c=\ell-3)$;
} \label{eq:alpha4} \\
%
\alpha_5:\;\;
\begin{bmatrix}
x \\ \alpha_5(x)
\end{bmatrix}
&=
\begin{bmatrix}
u\,\underline{\vis{1}\,1^{a-1}}\,0^{a-1}\,\vis{0}\,\hyph^c\,\hyph\,w\,\hyph\,({*}\,\hyph\,|\,\varepsilon) \\
u\,\underline{\hyph\,1^{a-1}}\,0^{a-1}\,\hyph\,\hyph^c\,\vis{1}\,w\,\vis{0}\,({*}\,\hyph\,|\,\varepsilon)
\end{bmatrix},
&& \notag \\
& && \hspace{-8cm}
\parbox[c]{14cm}{
$a=v_1 \text{ odd} \,\wedge\, (a=1 \,\vee\, c\in\{0,1\}) \,\wedge\, u\in D \,\wedge\, w\in D^+ \,\wedge\,$ \\
$\lnot\big(a=1 \,\wedge\, c=0 \,\wedge\, \nu(x)\geq 3 \,\wedge\, b:=v_2<v_3 \,\wedge\, u=1^b\, 0^b \big) \,\wedge\,$ \\
$\Big(\lnot \big(a=1 \,\wedge\, \nu(x)\geq 3 \,\wedge\, b:=v_2<v_3 \,\wedge\, w=1^b\, 0^b \big) \,\vee\, (a=1\,\wedge\, c=0\,\wedge\,w=10)\Big)$;
} \label{eq:alpha5} \\
%
\alpha_6:\;\;
\begin{bmatrix}
x \\ \alpha_6(x)
\end{bmatrix}
&=
\begin{bmatrix}
\vis{1}\,1^{b-1}\,0^{b-1}\,\vis{0}\,\underline{1}\,0\,\hyph\,w\,\hyph\,(*\,\hyph\,|\,\varepsilon) \\
\hyph\,1^{b-1}\,0^{b-1}\,\hyph\,\underline{1}\,0\,\vis{1}\,w\,\vis{0}\,(*\,\hyph\,|\,\varepsilon)
\end{bmatrix},
&& \hspace{-16mm} b=v_2<v_3 \,\wedge\, w\in D^+; \label{eq:alpha6} \\
%
\alpha_7:\;\;
\begin{bmatrix}
x \\ \alpha_7(x)
\end{bmatrix}
&=
\begin{bmatrix}
\underline{\vis{1}}\,\vis{0}\,\hyph^*\,\hyph\,1^b\,0^b\,\hyph^*\,\hyph\,w'\,\hyph\,* \\
\underline{\hyph}\,\hyph\,\hyph^*\,\hyph\,1^b\,0^b\,\hyph^*\,\vis{1}\,w'\,\vis{0}\,*
\end{bmatrix},
&& \hspace{-16mm} b=v_2<v_3 \,\wedge\, b\geq 2 \,\wedge\, w'\in D^+; \label{eq:alpha7} \\
%
\alpha_8:\;\;
\begin{bmatrix}
x \\ \alpha_8(x)
\end{bmatrix}
&=
\begin{bmatrix}
u\,\underline{1}\,0\,\hyph\,\hyph^c\,\vis{1}\,1^{b-1}\,0^{b-1}\,\vis{0}\,\hyph^*\,\hyph \\
u\,\underline{1}\,0\,\vis{0}\,\hyph^c\,\hyph\,1^{b-1}\,0^{b-1}\,\hyph\,\hyph^*\,\vis{1}
\end{bmatrix},
&& \hspace{-16mm} b=v_2<v_3 \,\wedge\, (b\geq 2 \,\vee\, c=1) \,\wedge\, u\in D^+; \label{eq:alpha8} \\
%
\alpha_9:\;\;
\begin{bmatrix}
x \\ \alpha_9(x)
\end{bmatrix}
&=
\begin{bmatrix}
u\,\underline{\vis{1}}\,\vis{0}\,\hyph^c\,\hyph\,\hyph\,\hyph\,1\,0\,\hyph^*\,\hyph \\
u\,\underline{\hyph}\,\hyph\,\hyph^c\,\vis{1}\,\vis{0}\,\hyph\,1\,0\,\hyph^*\,\hyph
\end{bmatrix},
&& \hspace{-16mm} 1<v_3 \,\wedge\, c\geq 0 \,\wedge\, u\in D^+. \label{eq:alpha9}
\end{align}
\end{subequations}

These specifications use the regular expressions defined before.
Furthermore, preimage and image of each mapping are written as a column vector, to facilitate their positionwise comparison, which allows immediate verification that these are indeed valid connectors (recall~\eqref{eq:conn}).
These expressions contain several occurrences of substrings of the form~$\vis{1}\,1^{a-1}\,0^{a-1}\,\vis{0}$, which could of course be simplified to~$1^a\,0^a$, i.e., $a$ consecutive 1s follows by $a$ consecutive 0s, but this would make the positionwise comparison harder, as the second expression is~$\hyph\,1^{a-1}\,0^{a-1}\,\hyph$.
Note that the definitions of~$\alpha_i$, $i\in\{2,4\}$, each involve a distinction between two possible cases for~$\alpha_i(x)$.
The condition $\chi(y,\gamma)\notin X_4$ stated in the `if' branch of these case distinctions will be specified later at the end of this proof.

For $i=1,\ldots,9$ we define $X_i:=\dom(\alpha_i)$ and $Y_i:=\img(\alpha_i)=\alpha_i(X_i)$.
Furthermore, the set of connectors~$\cU\seq\cX_{n,k}$ we use to prove the theorem is defined, as one would expect at this point, by
\begin{equation}
\cU:=\bigcup_{i=1}^9 \big\{\{x,\alpha_i(x)\}\mid x\in X_i\big\}.
\end{equation}
To complete the proof, we have to show that the sets in~$\cU$ are pairwise disjoint, that the sets in~$\cU$ and~$\cS_r$ are pairwise disjoint for a suitable value of~$r$, and that~$\cH_{n,k}[\cU]$ is connected.

We first argue that the connectors in~$\cU$ are pairwise disjoint.
As we have not yet specified the conditions in the `if' case of the definitions of~$\alpha_2$ and~$\alpha_4$, we prove the slightly stronger statement that the connectors are disjoint for both of the two exclusive cases, in whichever combination they may occur.

For any vertex~$x\in X_{n,k}$ with $\nu(x)\geq 2$ and any glider~$\gamma\in\Gamma(x)$ that is the rightmost glider in a train of the minimum speed~$v(\gamma)=\min V(x)=v_1$, we define a vertex~$\chi(x,\gamma)$ on the cycle~$C(x)$ as follows.
If $v(\gamma)$ is even, then let $t\geq 0$ be such that~$\gamma^t$ satisfies the conditions stated in Lemma~\ref{lem:flip-vertex} for $b:=0$ and~$i:=p$, and define~$\chi(x,\gamma):=x^t$.
If $v(\gamma)$ is odd, then let $t\geq 0$ be such that~$\gamma^t$ satisfies the conditions stated in Lemma~\ref{lem:flip-vertex} for $b:=1$ and~$i:=p$, and define~$\chi(x,\gamma):=x^t$.
In words, the vertex~$\chi(x,\gamma)$ is obtained by moving along the cycle~$C(x)$ starting from~$x$ and tracking the movement of the glider~$\gamma$ until one of its $b$-bits is at position~$p$.

Clearly, the set of all vertices from~$X_{n,k}$ under the mapping~$\chi$ is specified by the regular expression~$x=\underline{1^a\,0^a}\,\hyph\,*$ with the condition~$\nu(x)\geq 2\,\wedge\,a=v_1$.
We first argue that every such vertex~$x$ belongs to exactly one of the sets~$X_i$, $i=2,\ldots,9$, and that these sets are all disjoint.
This argument is given in the decision tree shown in Figure~\ref{fig:connect}.
Each branching distinguishes several exclusive cases, captured by logical conditions at the nodes.
Furthermore, the conditions imposed on~$x$ are captured by a regular expression that gets more refined as we move towards the leaves and extra conditions are imposed.
At the leaf nodes, one can check that $x$ belongs to precisely one of the sets~$X_i$.
Note that there are two branches in the tree ending with~$X_4$ and two branches ending with~$X_5$.
By taking the disjunction of conjunctions of logical conditions of all root-leaf paths for a particular set~$X_i$, $i\in\{2,\ldots,9\}$, we obtain precisely the regular expressions and corresponding logical conditions stated in the definition of~$\alpha_i$ before.

\begin{figure}
\includegraphics[page=1]{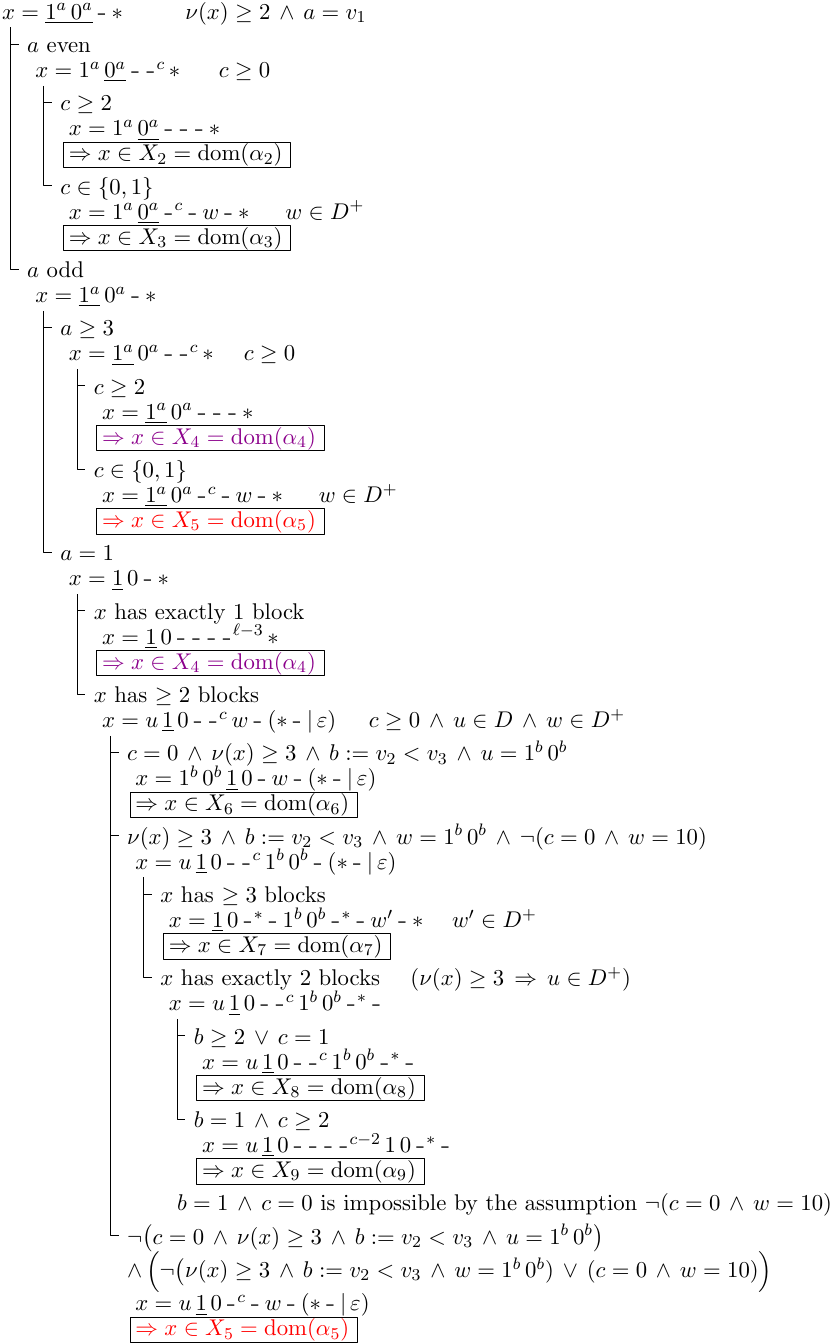}
\caption{Illustration of the proof of Theorem~\ref{thm:connect}.}
\label{fig:connect}
\end{figure}

We have shown that $X_i$, $i=2,\ldots,9$ are pairwise disjoint sets.
Note also that these sets are all disjoint from~$X_1$, as for $x\in X_1$ we have $x_p=\hyph$, whereas $x'_p\in\{0,1\}$ for all $x'\in X_2\cup\cdots\cup X_9$.

Furthermore, it is easy to check that each of the mappings~$\alpha_i$, $i=1,\ldots,9$, is invertible and thus an injection.
This is because in each of the regular expressions that specifies the image~$\alpha_i(x)$, the shaded visible pair of bits $\vis{1}\,\vis{0}$ can be identified uniquely and exchanged accordingly to obtain the preimage~$x$.
For example, the preimage of $y=\alpha_1(x)=u\,\underline{\vis{0}}\,\hyph\,\hyph\,\hyph^{\ell-2}\,\vis{1}$ is obtained by removing the visible pair of bits that includes the 0-bit at position~$p$ and adding the pair at positions~$(p+1,p+2)$ instead, yielding $x=u\,\underline{\hyph}\,\vis{1}\,\vis{0}\,\hyph^{\ell-2}\,\hyph$.

Consequently, we still need to verify that $Y_i\cap X_j=\emptyset$ for all $1\leq i,j\leq 9$ and $Y_i\cap Y_j=\emptyset$ for all $1\leq i<j\leq 9$.
We do this pinpointing, for any pair of such bitstrings, a particular property in which they differ.
In the remainder of the proof we sometimes refer to the speed sets of several different bitstrings~$x\in X_{n,k}$, and to avoid ambiguities we then explicitly specify the argument~$x$ for the different speed values as~$V(x)=\{v_1(x),\ldots,v_{\nu(x)}(x)\}$.

{\bf Case 1:} $y=\alpha_1(x)\in Y_1$.
We have $y_p=0$, whereas $x'_p\in\{1,\hyph\}$ for $x'\in X_i,Y_i$, $i=4,\ldots,9$.
The vertex~$y$ has exactly one visible pair of matched bits, whereas every $x'\in X_1,X_2,Y_2,X_3,Y_3$ has at least two visible pairs of matched bits.

{\bf Case 2:} $y=\alpha_2(x)\in Y_2$.
The vertex~$y$ has the following properties:
\begin{enumerate}[label=(\roman*),leftmargin=8mm,topsep=0mm]
\item $y_p\in\{0,\hyph\}$.
\item If $y_p=0$, then $v(\gamma(y,p))$ is odd.
\item If $y_p=\hyph$, then either $y_{p+1}=1$ and $y$ contains at least three blocks, or $y_{p+1}=\hyph$.
\item The first block entirely to the right of position~$p$ starts with~$10$.
\end{enumerate}
If $x'\in X_1$, then $x'_p=\hyph$, $x'_{p+1}=1$ and $x'$ contains exactly two blocks, contradicting~(iii).

If $x'\in X_2\cup X_3$, then we have $x'_p=0$ and $v(\gamma(x',p))$ is even, contradicting~(ii).
If $y'\in Y_3$, then the first block entirely to the right of position~$p$ starts with $1\,w\,0$ for some~$w\in D^+$, contradicting~(iv).
If $x'\in X_4,X_5,X_6,X_7,X_8,X_9,Y_6,Y_8$, then we have $x'_p=1$, contradicting~(i).
If $y'=\alpha_4(x')\in Y_4$, then if $y'_p=1$ we have a contradiction to~(i).
Otherwise we have $y'_p=\hyph$ and the first block entirely to the right of position~$p$ starts with~$1^{a-1}\,0^{a-1}$ for some odd~$a$, which contradicts~(iv) if $a\geq 3$.
Otherwise we have $a=1$, and then the definition~\eqref{eq:alpha4} implies that~$x'$ has only a single block, and all $\hyph$ are consecutive.
Then the equality $y=y'$ implies that $x=1^{a'}\,0^{a'}\,\hyph^\ell$ for some $a'\geq 2$, i.e., $\nu(x)=1$, contradicting the condition~$\nu(x)\geq 2$ in~\eqref{eq:alpha2}.
If $y'\in Y_5$, then either $y'_p=1$, which contradicts~(i), or $y'_p=\hyph$ and the first block entirely to the right of position~$p$ is $1^{a-1}\,0^{a-1}$ for some odd~$a$ or starts with~$1\,w\,0$ for $w\in D^+$, contradicting~(iv).
If $y'\in Y_7$, then the first block entirely to the right of position~$p$ is $1^b\,0^b$ for $b\geq 2$, contradicting~(iv).
If $y'\in Y_9$, then the equality $y=y'$ implies that $x$ contains a substring~$\hyph 10$, which means that $v_1(x)=1$, contradicting the condition that $v_1(x)$ must be even in~\eqref{eq:alpha2}.

{\bf Case 3:} $y=\alpha_3(x)\in Y_3$.
The vertex~$y$ has the following properties:
\begin{enumerate}[label=(\roman*),leftmargin=8mm,topsep=0mm]
\item $y_p\in\{0,\hyph\}$.
\item If $y_p=0$, then $v(\gamma(y,p))$ is odd.
\item If $y_p=\hyph$, then $y_{p+2}=1$.

\item The first block entirely to the right of position~$p$ starts with~$1\,w\,0$ for some $w\in D^+$.
\end{enumerate}
If $x'\in X_1,Y_9$, then the first block entirely to the right of position~$p$ starts with~$10$, contradicting~(iv).
If $x'\in X_2,X_3$, then we have $x'_p=0$ and $v(\gamma(x',p))$ is even, contradicting~(ii).
If $x'\in X_4,X_5,X_6,X_7,X_8,X_9,Y_6,Y_8$, then we have $x'_p=1$, contradicting~(i).
If $y'=\alpha_4(x')\in Y_4$, then if $y'_p=1$ there is a contradiction to~(i).
Otherwise we have $y'=\underline{\hyph}\,1^{a-1}\,0^{a-1}\,\hyph\,\hyph^*\,1\,0\,*$ for some odd~$a\geq 1$.
Consequently, if $a=1$ then the first block in~$y'$ entirely to the right of position~$p$ starts with~$10$, contradicting~(iv).
On the other hand, if $a\geq 3$ then the equality $y=y'$ implies that $x$ contains a substring~$\hyph 10$, which means that $v_1(x)=1$, contradicting the condition that $v_1(x)$ must be even in~\eqref{eq:alpha3}.
If $y'=\alpha_5(x')\in Y_5$, then if $y'_p=1$ we have a contradiction to~(i).
Otherwise we have $y'_p=\hyph$ and then a conflict~$y=y'$ either implies that $y=y'=\hyph\,1^{a-1}\,0^{a-1}\,\underline{\hyph}\,\hyph\,1\,w\,0\,*$ or $y=y'=\hyph\,1^{a-1}\,0^{a-1}\,\underline{\hyph}\,1\,w\,0\,\hyph\,\hyph^*\,1\,w'\,0\,*$ for some even $a\geq 2$ with $a=v_1(x)$ and $w,w'\in D^+$.
In the first case, $x'=\hyph\,1^{a-1}\,0^{a-1}\,\underline{1}\,0\,\hyph\,w\,\hyph\,*$ violates the conditions in the second row of the conjunction in~\eqref{eq:alpha5}, which is a contradiction.
In the second case, we have $x'=\hyph\,1^{a-1}\,0^{a-1}\,\underline{1}\,1\,w\,0\,0\,\hyph^*\,\hyph\,w'\,\hyph\,*$.
Note that $v_1(w)\geq a$, and consequently $v_1(1\,1\,w\,0\,0)\geq a+2>a-1$.
It follows that $v_1(x')=a-1$, and since the 1-bit in~$x'$ at position~$p$ does not belong to a glider of the minimum speed by Lemma~\ref{lem:outer-steps}, we have again a contradiction to the definition~\eqref{eq:alpha5}.
If $y'\in Y_7$, then we have $y'_p=\hyph$ and $y'_{p+2}=\hyph$, contradicting~(iii).

{\bf Case 4:} $y=\alpha_4(x)\in Y_4$.
The vertex~$y$ has the following properties:
\begin{enumerate}[label=(\roman*),leftmargin=8mm,topsep=0mm]
\item $y_p\in\{1,\hyph\}$.
\item If $y_p=1$ then $v(\gamma(y,p))$ is even.
\item $y_{[p,p+1,p+2]}\neq \hyph\,1\,0$.
\item $y_{[p-1,p]}\neq 01$.
\item The first block entirely to the right of position~$p+1$ starts with~$10$.
\end{enumerate}
If $x'\in X_1$, then we have $x'_{[p,p+1,p+2]}=\hyph\,1\,0$, contradicting~(iii).
If $x'\in X_2,X_3$, then we have $x'_p=0$, contradicting~(i).
If $x'\in X_4,X_5,X_6,X_7,X_8,X_9,Y_6$, then we have $x'_p=1$ and $v(\gamma(x',p))$ is odd, contradicting~(ii).
If $y'\in Y_5$, then the first block entirely to the right of position~$p+1$ starts with $1\,w\,0$ for some $w\in D^+$, contradicting~(v).
If $y'\in Y_7$, then the first block entirely to the right of position~$p+1$ starts with $1^b\,0^b$ for some $b\geq 2$, contradicting~(v).
If $y'\in Y_8$, then we have $y'_{[p-1,p]}=01$, contradicting~(iv).
If $y'\in Y_9$, then the equality $y=y'$ implies that $v_1(x)=1$ and then the definition~\eqref{eq:alpha4} implies that~$x$ must have a single block, and all $\hyph$ are consecutive.
Then $y$ has only two blocks, whereas $y'$ has three blocks, a contradiction.

{\bf Case 5:} $y=\alpha_5(x)\in Y_5$.
The vertex~$y$ has the following properties:
\begin{enumerate}[label=(\roman*),leftmargin=8mm,topsep=0mm]
\item $y_p\in\{1,\hyph\}$.
\item If $y_p=1$ then $v(\gamma(y,p))$ is even.
\item $y_{[p,p+1,p+2]}\neq \hyph\,1\,0$.
\item $y_{p-1,p}\neq 01$.
\item The first block entirely to the right of position~$p+1$ starts with $1\,w\,0$ for some $w\in D^+$.
\end{enumerate}
If $x'\in X_1$, then we have $x'_{[p,p+1,p+2]}=\hyph\,1\,0$, contradicting~(iii).
If $x'\in X_2,X_3$, then we have $x'_p=0$, contradicting~(i).
If $x'\in X_4,X_5,X_6,X_7,X_8,X_9,Y_6$, then we have $x'_p=1$ and $v(\gamma(x',p))$ is odd, contradicting~(ii).
If $y'=\alpha_7(x')\in Y_7$, then a conflict~$y=y'$ implies $y=y'=\underline{\hyph}\,\hyph\,\hyph^*\,\hyph\,1^b\,0^b\,\hyph^*\,1\,w'\,0\,*$ for $b:=v_2(x')<v_3(x')$, $b\geq 2$, and some $w'\in D^+$.
This however implies $x=\underline{1}\,0\,\hyph^*\,\hyph\,\hyph\,1^{b-1}\,0^{b-1}\,\hyph\,\hyph^*\,1\,w'\,0\,*$, so $x$ violates the conditions in the third row of the conjunction in~\eqref{eq:alpha5}, a contradiction.
If $y'\in Y_8$, then we have $y'_{[p-1,p]}=01$, contradicting~(iv).
If $y'\in Y_9$, then the first block entirely to the right of position~$p+1$ starts with~$10$, contradicting~(v).

{\bf Case 6:} $y=\alpha_6(x)\in Y_6$.
The vertex~$y$ has the following properties:
\begin{enumerate}[label=(\roman*),leftmargin=8mm,topsep=0mm]
\item $y_p=1$.
\item The glider $\gamma(y,p)$ is not open.
\item $y_{p+2}=1$.
\end{enumerate}
If $x'\in X_1,X_2,X_3,Y_7,Y_9$, then we have $x'_p\in\{0,\hyph\}$, contradicting~(i).
If $x'\in X_4,X_5,X_6,X_7,X_8,X_9$, then the glider~$\gamma(x',p)$ is open, contradicting~(ii).
If $y'\in Y_8$, then we have $y'_{p+2}=0$, contradicting~(iii).

{\bf Case 7:} $y=\alpha_7(x)\in Y_7$.
The vertex~$y$ has the following properties:
\begin{enumerate}[label=(\roman*),leftmargin=8mm,topsep=0mm]
\item $y_p=\hyph$.
\item $y_{p+1}=\hyph$.
\item The first block entirely to the right of position~$p$ starts with~$1^b\,0^b$ for some~$b\geq 2$.
\end{enumerate}
If $x'\in X_1$, then we have $x'_{p+1}=1$, contradicting~(ii).
If $x'\in X_2,X_3,X_4,X_5,X_6,X_7,X_8,X_9,Y_8$, then we have $x'_p\in\{0,1\}$, contradicting~(i).
If $y'\in Y_9$, then the first block entirely to the right of position~$p$ starts with~$10$, contradicting~(iii).

{\bf Case 8:} $y=\alpha_8(x)\in Y_8$.
If $x'\in X_1,X_2,X_3,Y_9$, then we have $x'_p\in\{0,\hyph\}$, whereas $y_p=1$.
If $x'\in X_4,X_5,X_6,X_7,X_8,X_9$, then we have $x'_{[p-1,p+2]}\neq 0100$, whereas $y_{[p-1,p+2]}=0100$.

{\bf Case 9:} $y=\alpha_9(x)\in Y_9$.
If $x'\in X_1$, then we have $x'_{p+1}=1$, whereas $y_{p+1}=\hyph$.
If $x'\in X_2,X_3,X_4,X_5,X_6,X_7,X_8,X_9$, then we have $x'_p\in\{0,1\}$, whereas $y_p=\hyph$.

We now argue that the sets in~$\cU$ and~$\cS_r$ are pairwise disjoint for a suitable value of~$r\in\{0,\ldots,n-1\}$.
The connectors in~$\cU$ that contain a vertex~$s_i$ as defined in~\eqref{eq:si} all arise as images of~$\alpha_1$ and~$\alpha_5$.
Specifically, we have $s_{p+c}=\underline{\hyph}\,\hyph^c\,1^k\,0^k\,\hyph^{\ell-1-c}\in Y_5$ for all $c=1,\ldots,\ell-1$ and $s_{p+\ell}=1^k\,0^{k-1}\,\underline{0}\,\hyph^\ell\in Y_1$, i.e., all those connectors~$s_i$ satisfy
\begin{equation}
\label{eq:siU}
p+1\leq i\leq p+\ell.
\end{equation}
Note that the pairs of sets~$\cS_r$ defined in~\eqref{eq:Sr} contain only vertices~$s_i$ with $r\leq i\leq r+(g-2)+k+1$
Consequently, if we choose $r:=p+\ell+1$, then we have
\begin{equation}
\label{eq:siS}
p+\ell+1\leq i\leq (p+\ell+1)+(g-2)+k+1\leq p+n=p+1+(n-1),
\end{equation}
where we used $\ell=n-2k$ and the simple fact $g=\gcd(n,k)\leq k$ in the upper bound estimates.
From~\eqref{eq:siU} and~\eqref{eq:siS} we see that the sets in~$\cU$ and~$\cS_r$ are pairwise disjoint, as desired.

It remains to argue that~$\cH_{n,k}[\cU]$ is connected.
For this we consider the number partitions~$\psi(V(x))$ of vertices~$x\in X_i$, $i\in\{1,\ldots,9\}$, and we argue how they change with the connectors~$\{x,\alpha_i(x)\}$.
Specifically, we prove that the number partitions increase lexicographically, either directly along one connector, or along a sequence of two connectors where the first step is a small lexicographic decrease and the second step compensates for this, giving a lexicographic increase overall.
This last part of the proof also specifies the conditions for the case distinctions in the definitions~\eqref{eq:alpha2} and~\eqref{eq:alpha4}.

For each of the connectors~$\{x,\alpha_i(x)\}$, $i=1,\ldots,9$, defined in~\eqref{eq:alphai} we first describe how the number partitions~$\psi(V(x))$ and~$\psi(V(\alpha_i(x)))$ compare lexicographically.
For any $x\in X_1$, we obtain from~\eqref{eq:alpha1}, \eqref{eq:Vxy} and Lemma~\ref{lem:partition}~(i)+(ii) that
\begin{subequations}
\begin{equation}
\label{eq:part1}
\psi(V(\alpha_1(x)))\trr
\begin{cases}
\psi(V(x))\boxplus 1 & \text{if } \nu(x)\geq 3 \text{ and } v_3>v_2, \\
\psi(V(x)) & \text{otherwise.}
\end{cases}
\end{equation}
Similarly, for any $x\in X_2$ the definition~\eqref{eq:alpha2} and Lemma~\ref{lem:partition}~(iii) yield
\begin{equation}
\label{eq:part2}
\psi(V(\alpha_2(x)))=\psi(V(x))\boxminus 1
\end{equation}
(note that $v_1$ is assumed to be even and therefore $v_1\geq 2$).
For any $x\in X_3$ the definition~\eqref{eq:alpha3} and Lemma~\ref{lem:partition}~(i) give
\begin{equation}
\label{eq:part3}
\psi(V(\alpha_3(x)))\trr \psi(V(x)).
\end{equation}
For any $x\in X_4$, the definition~\eqref{eq:alpha4} and Lemma~\ref{lem:partition}~(iii) prove that
\begin{equation}
\label{eq:part4}
\psi(V(\alpha_4(x)))=
\begin{cases}
\psi(V(x))\boxminus 1 & \text{if } v_1\geq 3, \\
\psi(V(x)) & \text{if } v_1=1
\end{cases}
\end{equation}
(note that $v_1$ is assumed to be odd and therefore $v_1\neq 2$).
For any $x\in X_5$ the definition~\eqref{eq:alpha5} and Lemma~\ref{lem:partition}~(i)+(ii) show that
\begin{equation}
\label{eq:part5}
\psi(V(\alpha_5(x)))\trr
\begin{cases}
\psi(V(x))\boxplus 1 & \text{if } a=1\,\wedge\,\nu(x)\geq 3\,\wedge\,v_3>v_2\,\wedge\,w\neq 10, \\
\psi(V(x)) & \text{otherwise.}
\end{cases}
\end{equation}
For any $x\in X_i$, $i\in\{6,7,8\}$, the definitions~\eqref{eq:alpha6}--\eqref{eq:alpha8} and Lemma~\ref{lem:partition}~(ii) establish that
\begin{equation}
\label{eq:part678}
\psi(V(\alpha_i(x)))\trr \psi(V(x))\boxplus 1.
\end{equation}
Lastly, for any $x\in X_9$ we obtain from the definition~\eqref{eq:alpha9} that
\begin{equation}
\label{eq:part9}
\psi(V(\alpha_9(x)))=\psi(V(x)).
\end{equation}
\end{subequations}

Let $X:=\bigcup_{i=1}^9 X_i$ and consider a vertex~$x\in X$.
We show that there is a sequence of connectors and cycles of the factor~$\cC_{n,k}$ in between them to reach a vertex~$y\in X_{n,k}$ with $\psi(V(y))\trr \psi(V(x))$.

For $x\in X_1,X_3,X_5,X_6,X_7,X_8$ this follows directly from~\eqref{eq:part1}, \eqref{eq:part3}, \eqref{eq:part5}, \eqref{eq:part678}, respectively.
It remains to consider the cases $x\in X_2,X_4,X_9$.

{\bf Case (a):} $x\in X_2$.
Consider the two connectors~$\{x,y\}$ and~$\{x,y'\}$ where $y$ and~$y'$ are obtained from~$x$ as described by the first and second case of the case distinction in~\eqref{eq:alpha2}.
From~\eqref{eq:part2} we see that $\psi(V(y))=\psi(V(x))\boxminus 1$.
As $\nu(x)\geq 2$ and $v_1(x)\geq 2$ we have $\nu(y)\geq 3$, $v_1(y)=1$ and $v_3(y)>v_2(y)$.
We clearly also have~$V(y')=V(y)$.
Let $q$ be the position of the 1-bit of the substring~$10$ in which~$y$ differs from~$x$, and define $\gamma:=\gamma(y,q)\in\Gamma(y)$ and~$\gammatilde:=\gamma(y,q-2)\in\Gamma(y)$.
By Lemma~\ref{lem:open-move} and the definition of the mapping~$h_\gamma$ given in Section~\ref{sec:slow-move} we have $y'=h_\gamma(y)$.

We now consider the vertex~$z:=\chi(y,\gamma)\in X_i$, $i\in\{2,\ldots,9\}$, which satisfies $V(z)=V(y)$ by Lemma~\ref{lem:Vx-invariant}, and we let $t>0$ be such that~$y^t=z$.
As $v_1(y)=1$ is odd, we know that~$i\notin\{2,3\}$.

If $i\neq 4$, then we define~$\alpha_2(x):=y$, and we continue the argument as follows.
If $i=6,7,8$, then by~\eqref{eq:part678} we have
\begin{equation}
\label{eq:partpm}
\psi(V(\alpha_i(z)))\trr \psi(V(z))\boxplus 1=\psi(V(y))\boxplus 1=\big(\psi(V(x))\boxminus 1\big)\boxplus 1=\psi(V(x)),
\end{equation}
so we are done.
If $i=9$ then we have $z=u\,\underline{1}\,0\,\hyph\,\hyph^*\,1\,0\,\hyph\,\hyph^*$ for some $u\in D^+$.
Clearly, as $v_1(x)\geq 2$ there are at most two equivalence classes of gliders in~$\Gamma(y)$ and~$\Gamma(z)$ that have speed~1, and there can only be exactly two if~$v_1(x)=2$.
However, note that in~$h_\gammatilde(y)$ the two gliders of speed~1 form a train, and applying Lemma~\ref{lem:push} shows that in~$h_{\gammatilde^t}(z)$ they do not form a train (as $u\in D^+$), a contradiction to Lemma~\ref{lem:Zx-invariant}.
This shows that the case~$i=9$ cannot occur.
If $i=5$ then we have $z=u\,\underline{1}\,0\,\hyph\,\hyph^*\,w\,\hyph\,*$.
For the same reason as in the case~$i=9$, we must have $w\neq 10$, and therefore $\psi(V(\alpha_i(z)))\trr \psi(V(z))\boxplus 1$ by~\eqref{eq:part5}, so again~\eqref{eq:partpm} holds.

If $i=4$, on the other hand, then $z$ has only a single block, i.e., we have $z=u\,\underline{1}\,0\,\hyph^{\ell}$ for some $u\in D^+$.
In this case we define~$\alpha_2(x):=y'$.
By Lemma~\ref{lem:push} the vertex $z':=h_{\gamma^t}(z)=u\,\underline{\hyph}\,1\,0\,\hyph^{\ell-1}$ lies on the cycle~$C(y')$, and we have $z'\in X_1$.
From~\eqref{eq:part1} we therefore obtain
\begin{equation*}
\psi(V(\alpha_1(z')))\trr \psi(V(z'))\boxplus 1=\psi(V(y'))\boxplus 1=(\psi(V(x))\boxminus 1)\boxplus 1=\psi(V(x)),
\end{equation*}
so we are done.

{\bf Case (b):} $x\in X_4$.
Consider the two connectors~$\{x,y\}$ and~$\{x,y'\}$ where $y$ and~$y'$ are obtained from~$x$ as described by the first and second case of the case distinction in~\eqref{eq:alpha4}.
If $v_1(x)\geq 3$, then the argument continues as in case~(a) before.
It remains to consider the case~$v_1(x)=1$.
From~\eqref{eq:part4} we see that $\psi(V(y))=\psi(V(x))$.
We clearly also have~$V(y')=V(y)$.
Let $q$ be the position of the 1-bit of the substring~$10$ in which~$y$ differs from~$x$, and define $\gamma:=\gamma(y,q)\in\Gamma(y)$.

We now consider the vertex~$z:=\chi(y,\gamma)\in X_i$, $i\in\{2,\ldots,9\}$, and we let $t>0$ be such that~$y^t=z$.
As $v_1(y)=1$ is odd, we know that~$i\notin\{2,3\}$.

If $i\neq 4$, then we define~$\alpha_4(x):=y$, and we continue the argument as follows.
If $i=5,6,7,8$, then from~\eqref{eq:part5} and~\eqref{eq:part678} we have
\begin{equation*}
\psi(V(\alpha_i(z)))\trr \psi(V(z))=\psi(V(y))=\psi(V(x)).
\end{equation*}
If $i=9$, then we have $z=u\,\underline{1}\,0\,\hyph^c\,\hyph\,\hyph\,\hyph\,1\,0\,\hyph\,\hyph^*$ for some $c\geq 0$ and $u\in D^+$.
Furthermore, we have $V(\alpha_9(z))=V(z)$ by~\eqref{eq:part9}.
Let $\tq$ be the position of the 1-bit of the substring~$10$ in which $\alpha_9(z)$ differs from~$z$, and define $\gammatilde:=\gamma(\alpha_9(z),\tq)\in\Gamma(z)$.
We observe that~$\tz:=\chi(\alpha_9(z),\gammatilde)$ satisfies $\tz\in X_j$ with $j\notin \{2,3,4,9\}$ (apply Lemma~\ref{lem:push} to $\gammatilde$ and~$\alpha_9(z)$) and therefore
\begin{equation*}
\psi(V(\alpha_j(\tz)))\trr \psi(V(\tz))=\psi(V(\alpha_9(z)))=\psi(V(z))=\psi(V(y))=\psi(V(x)).
\end{equation*}

If $i=4$, on the other hand, then we define~$\alpha_4(x):=y'$.
By Lemma~\ref{lem:push} the vertex $z':=h_{\gamma^t}(z)$ lies on the cycle~$C(y')$, and we have $z'\in X_1$.
Using~\eqref{eq:part1} it follows that
\begin{equation*}
\psi(V(\alpha_1(z')))\trr \psi(V(z'))=\psi(V(y'))=\psi(V(x)).
\end{equation*}

{\bf Case (c):} $x\in X_9$.
This case can be settled in the same way as the subcase~$i=9$ in case~(b).

This completes the proof of the theorem.
\end{proof}

\subsection{Proof of Theorem~\ref{thm:Knk}}

\begin{proof}[Proof of Theorem~\ref{thm:Knk}]
The graphs~$K(3,1)$, $K(4,1)$ and~$K(6,2)$ can easily be checked to admit a Hamilton cycle.
On the other hand, $K(5,2)$ is the Petersen graph, which is well known not to have a Hamilton cycle.
Furthermore, it was shown in~\cite[Thm.~1]{MR4273468} that $K(2k+1,k)$ has a Hamilton cycle for all~$k\geq 3$.
Combining this result with~\cite[Thm.~1]{MR2836824} proves that $K(2k+2,k)$ has a Hamilton cycle for all~$k\geq 3$.
To cover the remaining cases, let $k\geq 1$ and $n\geq 2k+3$.
By Theorem~\ref{thm:connect}, there is a set~$\cU\seq\cX_{n,k}$ and a choice for $r$ where~$\cS_r$ is the set defined in~\eqref{eq:Sr} such that the sets in~$\cU$ are pairwise disjoint, the sets in~$\cU$ and~$\cS_r$ are pairwise disjoint, and $\cH_{n,k}[\cU]$ is a connected graph.
Let $\cT$ be an inclusion minimal subset of~$\cU$ such that~$\cH_{n,k}[\cT]$ is connected, i.e., this graph is a spanning tree.
By the pairwise disjointness of the sets in~$\cU\cup \cS_r$, no two of the 4-cycles~$C_4(x,y)$, $\{x,y\}\in\cT$, defined in Lemma~\ref{lem:connector} or 4-cycles $C_4'(x,y)$, $\{x,y\}\in\cS_r$, defined in Lemma~\ref{lem:join} have an edge in common with one of the cycles of the factor~$\cC_{n,k}$ in~$K(n,k)$ defined in~\eqref{eq:factor}.
Furthermore, by the minimality of~$\cT$, no two of these 4-cycles have an edge in common that does not lie on one of the cycles of the factor (otherwise, $\cT$ would contain two edges connecting the same pair of nodes), so these 4-cycles are pairwise edge-disjoint.
Consequently, as $\cH_{n,k}[\cT]$ is connected, the symmetric difference of the edge set of the cycle factor~$\cC_{n,k}$ with the 4-cycles~$C_4(x,y)$, $\{x,y\}\in\cT$, and with the 4-cycles $C_4'(x,y)$, $\{x,y\}\in\cS_r$, is a Hamilton cycle in~$K(n,k)$.
\end{proof}

\section{Open questions}

An interesting open problem is to develop an efficient algorithm for computing a Hamilton cycle in the Kneser graph~$K(n,k)$, i.e., an algorithm whose running time is polynomial in~$n$ and~$k$, not only polynomial in~$N:=\binom{n}{k}$, the size of the Kneser graph.
Our construction does not give such an algorithm for fundamental reasons.
Most importantly, our proof does not reveal the number of cycles in the cycle factor, so it is not even clear how many gluings via 4-cycles have to be used to join them to a Hamilton cycle.
Clearly, if there are $c$ cycles, then $c-1$ gluings will be needed, but we do not know how to compute $c$ more efficiently than actually computing all the cycles of our factor, which takes time polynomial in~$N$.
Even if we knew how to compute~$c$ efficiently, we do not know how to efficiently compute an inclusion minimal set of $c-1$ connectors that joins all the cycles (our proof uses many more connectors than will eventually be needed).

Furthermore, for any of the families of vertex-transitive graphs shown in Figure~\ref{fig:theorems}, it would be very interesting to investigate whether they admit a Hamilton decomposition, i.e., a partition of all their edges into Hamilton cycles, plus possibly a perfect matching.
This problem was raised by Meredith and Lloyd~\cite{MR0321782}, and by Biggs~\cite{MR556008} for the odd graphs~$O_k=K(2k+1,k)$.
Gould's survey~\cite{MR1106528} mentions the analogous problem for the middle levels graphs~$M_k=H(2k+1,k)$.
In these two cases, we do not even know \emph{two} edge-disjoint Hamilton cycles, so the problem seems to be very hard in general.
On the other hand, for Johnson graphs~$J(n,2)$ with odd~$n$, two edge-disjoint Hamilton cycles are known~\cite{MR4046775}.
To tackle this problem in general, it might be helpful to first consider decompositions of the edges of the graph into cycle factors; see~\cite{MR2128031}.

A strengthening of the concept of containing a Hamilton cycle is to contain the $r$th power of a Hamilton cycle.
To this end, Katona~\cite{MR2181045} conjectured that $K(n,k)$ contains the $r$th power of a Hamilton cycle for $r:=\max\{1,\lfloor n/k\rfloor-2\}$, which means that any $r+1$ consecutive vertices along the Hamilton cycle are disjoint sets, and he proved this for $k=2$ using Walecki's theorem.
Theorem~\ref{thm:Knk} confirms this conjecture for the cases~$2k+1\leq n\leq 4k-1$ (where $r=1$).
It seems plausible that Katona's conjecture holds even for $r:=\lceil n/k\rceil-2$, and our theorem confirms this stronger variant of the conjecture for the cases~$2k+1\leq n\leq 3k$ (where $r=1$).

\section*{Acknowledgements}

We thank Petr Gregor and Pascal Su for several inspiring discussions about Kneser graphs, and we also thank Petr Gregor for providing feedback on the first section of this paper.
We are extremely grateful to one of the STOC reviewers and one of the reviewers of the journal version, whose careful reading and checking of all proof details helped eliminating several technical errors, and to substantially improve the presentation.

\bibliographystyle{alpha}
\bibliography{refs}

\end{document}